\documentclass[11pt,a4paper,reqno]{article} 
\usepackage{amsmath}
\usepackage{amsthm}
\usepackage{enumerate}
\usepackage[inline]{enumitem}
\usepackage{amssymb}

\usepackage{verbatim}
\usepackage{url}
\usepackage[latin1]{inputenc}

\usepackage{caption}
\usepackage{graphicx,color}
\def\Zf{{\mathbb Z}}
\def\Rf{{\mathbb R}}

\def\N{{\mathbb N}}
\def\Z{{\mathbb Z}}

\def\R{{\mathbb R}}

\def\zd{{\mathbb Z}^d}

\def\Rc{{\mathcal R}}

\def\bfp{\mathbf{P}}

\numberwithin{equation}{section}

\def\munup{\Phi_p(\nu)}
\def\xpip{X^{\nu,p}}
\def\xnup{X^{\nu,p}}
\def\jointlaw{{\mathbf P}_{\nu,p}}
\def\ttinv{TT^{-1}}
\def\rer{{\rm RER}}
\def\cp{{\rm CP}}
\def\exch{{\rm exch}}
\def\pure{{\rm pure}}
\def\conn{{\rm conn}}
\def\stat{{\rm stat}}
\def\symm{{\rm symm}}
\def\simple{{\rm simple}}
\def\part{{\rm Part}}

\def\xibfp{\xi_{{\bf p},p}}
\def\xihalf{\xi_{{\bf p},1/2}}
\def\pf{{\bf Proof. }}
\def\xipp{\xi_{{\bf p},p}}
\def\paintbox{\nu_{{\bf p}}}

\def\be{\begin{equation}}
\def\ee{\end{equation}}
\def\bea{\begin{equation*}}
\def\eea{\end{equation*}}
\def\bal{\begin{aligned}}
\def\eal{\end{aligned}}

\def\eps{\varepsilon}

\def\Pr{{\mathbf P}}

\newtheorem{thm}{Theorem}
\newtheorem{lma}[thm]{Lemma}
\newtheorem{cor}[thm]{Corollary}
\newtheorem{prop}[thm]{Proposition}

\newtheorem{df}[thm]{Definition}

\newtheorem{question}[thm]{Question}
\newtheorem{example}[thm]{Example}

\newtheorem{remark}[thm]{Remark}

\numberwithin{thm}{section}

\theoremstyle{remark}

\theoremstyle{definition}

\def\ks{\mbox{-}}
\def\part{{\rm Part}}

\begin{document}

\title{Generalized Divide and Color models}
\author{Jeffrey E. Steif\footnote{Department of Mathematics, Chalmers University of Technology and Gothenburg University, Sweden. E-mail: steif@chalmers.se. Research supported by the Swedish research council and the Knut and Alice Wallenberg foundation.} \and Johan Tykesson\footnote{Department of Mathematics, Chalmers University of Technology and Gothenburg University, Sweden. E-mail: johant@chalmers.se. Research supported by the Knut and Alice Wallenberg foundation.} }
\maketitle

{\it The first author dedicates this paper to the memory of Jonathan Kaminsky \\
(1978-2016).}

\begin{abstract}
In this paper, we initiate the study of ``Generalized Divide and Color Models''.
A very special interesting case of this is the 
``Divide and Color Model'' (which motivates the name we use)
introduced and studied by Olle H\"{a}ggstr\"{o}m.

In this generalized model, one starts with a finite or countable set 
$V$, a random partition of $V$ and a parameter $p\in [0,1]$. The corresponding
Generalized Divide and Color Model is the $\{0,1\}$-valued process indexed by $V$ obtained by 
independently, for each partition element in the random partition
chosen, with probability $p$, assigning all the elements of the partition element the value 1,
and with probability $1-p$, assigning all the elements of the partition element the value 0.

Some of the questions which we study here are the following. Under what situations can different
random partitions give rise to the same color process? 
What can one say concerning exchangeable random partitions? 
What is the set of product measures that a color process stochastically dominates?
For random partitions which are translation invariant, what ergodic properties
do the resulting color processes have? 

The motivation for studying these processes is twofold; on the one hand, we believe that this is 
a very natural and interesting class of processes that deserves investigation and 
on the other hand, a number of quite varied well-studied processes 
actually fall into this class such as (1) the Ising model,
(2) the fuzzy Potts model, (3) the stationary distributions for the Voter Model,
(4) random walk in random scenery and of course (5) the original Divide and Color Model.

\end{abstract}

\newpage

\tableofcontents


\section{Introduction}

\subsection{Overview}\label{ss.overview}
In this paper, we initiate the study of a large class of processes
which we call ``Generalized Divide and Color Models''. 
The name is motivated by a model, introduced and studied by 
Olle H\"{a}ggstr\"{o}m~\cite{OH01}, called the 
``Divide and Color Model'', which is a special case of the class we look at here;
this special case will be described later in this section.

We believe that this general class of models warrants investigation,
partly because it seems to be a very natural class and partly because a number
of very  different processes studied in probability theory fall into this class, as 
described in Subsection~\ref{ss.Examples}.

We now describe this class somewhat informally; formal definitions will be given in 
Subsection~\ref{s.Defn.Not}. We start with a finite or countable set $V$. In the first step,
a random partition of $V$ (with an arbitrary distribution) is chosen and in the 
second step, independently, for each partition element in the random partition
chosen in the first step, with probability $p$, all the elements of the
partition element are assigned the value 1 and 
with probability $1-p$, all the elements of the partition element are assigned the value 0.
This yields in the end a $\{0,1\}$-valued process indexed by $V$, which we call a 
``Generalized Divide and Color Model'' and it is this process which
will be our focus.  Note that this process depends on, in addition of course to the set
$V$, the distribution of the random partition and the parameter $p$. A trivial
example is when the random partition always consists of singletons, in which case
we simply obtain an i.i.d.\ process with parameter $p$.

\subsection{Definitions and notation}\label{s.Defn.Not}

Let $V$ be a finite or countable set and let $\rm{Part}_V$ be the set of all partitions of $V$. 
Elements of $V$ will be referred to as vertices. Elements of a partition will be referred to 
either as equivalence classes or clusters. If $\pi\in \rm{Part}_V$ and $v\in V$, we let 
$\pi(v)$ denote the partition element of $\pi$ containing $v$. 

For any measurable space $(S,\sigma(S))$, let ${\mathcal P}(S)$ denote the 
set of probability measures on $(S,\sigma(S))$. If $\pi\in \rm{Part}_V$ and $K\subseteq V$, 
let $\pi_{K}$ denote the partition of $K$ induced from $\pi$ in the obvious way. On $\rm{Part}_V$ 
we consider the $\sigma$-algebra $\sigma(\rm{Part}_V)$ generated by $\{\pi_K\}_{K\subset V,\,|K|<\infty}$.

We denote the set of all probability measures on $(\rm{Part}_V,\sigma(\rm{Part}_V))$ 
by $\rm {RER}_V$ where $\rm {RER}$ stands for ``random equivalence relation''.  When $V$ 
has a natural set of translations (such as $\Zf^d$), we let $\rm{RER}^{\rm{stat}}_V$ 
("stat" for stationary) denote the elements of $\rm {RER}_V$ which are invariant under 
these translations. When $V$ is a graph (such as $\Zf^d$ with nearest neighbor edges), we let 
$\rm{RER}^{\rm{conn}}_V$ denote the subset of $\rm {RER}_V$ which are supported on partitions for 
which each cluster is connected in the induced graph. Finally, we let $\rm{RER}^{\rm{exch}}_V$ ("exch" for exchangeable)
denote the elements of $\rm {RER}_V$ which are invariant under all permutations of $V$ which fix all but
finitely many elements.

For each finite or countable set $V$ and for each $p \in  (0,1)$,  we now introduce a mapping 
$\Phi_p$ from $\rm {RER}_V$ to probability measures on $\{0,1\}^{V}$.  The image of some 
$\nu\in\rm {RER}_V$ will be called the ``color process'' or ``Generalized Divide and Color Model''
associated to $\nu$ with parameter $p$ and 
is defined as follows. Let $\pi\in \rm{Part}_V$ be picked at random according to $\nu$. 
For each partition element $\phi$ of $\pi$, we assign \emph{all} vertices in $\phi$ the value 
$1$ with probability $p$ and the value $0$ with probability $1-p$, independently for different 
partition elements. This yields for us a $\{0,1\}^V$-valued random object, $\xnup$, whose distribution 
is denoted by $\Phi_p(\nu)$. (Clearly $\Phi_p(\nu)$ is affine in $\nu$.)
We will also refer to $\xnup$ as the \emph{color process} associated to $\nu$ with 
parameter $p$. This clearly corresponds, in a more formal way, to the generalized divide and color 
model introduced in Subsection~\ref{ss.overview}. Finally, we let $\cp_{V,p}$ (CP for ``color process'') 
be the image of $\rm {RER}_V$ under
$\Phi_p$ and we also let $\cp^{*}_{V,p}$ be the image under $\Phi_p$ of the relevant subset 
$\rm{RER}^{*}_V$ of $\rm{RER}_V$ (${*}$ is {\rm{stat}}, {\rm{conn}} or {\rm{exch}}.)

We usually do not consider the cases $p=0$ or $1$ for they are of course trivial.
We let $|\cdot |_1$ denote the $L^1$ norm on $\zd$. 

We end this section with the following elementary observation.
For any $\nu\in \rm {RER}_V$, $p\in [0,1]$ and $u,v\in V$, we have, letting $E$ denote
the event that $u$ and $v$ are in the same cluster,
\begin{equation}\label{e.nonneg.cor}
{\mathbf P}(\xnup(u)=\xnup(v)=1)=p{\mathbf P}(E)+ p^2{\mathbf P}(E^c)\ge p^2
={\mathbf P}(\xnup(u)=1) {\mathbf P}(\xnup(v)=1)
\end{equation}
and hence $\xnup$ has nonnegative pairwise correlations. Note trivially that
$\xnup$ is pairwise independent if and only if it is i.i.d.

\subsection{Examples of color processes}\label{ss.Examples}

It turns out that a number of random processes which have been studied in probability theory
have representations as color processes. In this subsection, we give five such key examples. 
There is a slight difference between the first two examples and the last three examples.
In the first two examples, the known model corresponds to a color process with respect to a particular 
RER at a specific value of the parameter $p$ but not for other values of $p$, while in the last three examples,
the known model corresponds to all the color processes with respect to a particular RER as $p$ varies over all values.

\subsubsection{The Ising Model}

For simplicity, we stick to finite graphs here. While the results here are
{\sl essentially} true also for infinite graphs as well, there are some issues which arise 
in that case but they will not concern us here. Let $G=(V,E)$ be a finite graph.

\begin{df}\label{df.Ising}
The Ising model on $G=(V,E)$ with coupling constant $J\in {\mathbb R}$ and external field $h\in {\mathbb R}$ 
is the probability measure $\mu_{G,J,h}$ on $\{-1,1\}^V$ given by
$$
\mu_{G,J,h}(\{\eta(v)\}_{v\in V}):=
e^{J\sum_{\{v,w\}\in E} \eta(v)\eta(w)+ h\sum_v \eta(v)}/Z
$$
where $Z=Z(G,J,h)$ is a normalization constant.
\end{df}

It turns out that $\mu_{G,J,0}$ is a color process when $J\ge 0$; 
this corresponds to the famous FK (Fortuin-Kasteleyn) or 
so-called random cluster representation. To explain this, we first need to introduce the following model.

\begin{df}\label{df.RC}
The FK or random cluster model on $G=(V,E)$ with parameters $\alpha\in [0,1]$ and 
$q\in (0,\infty)$ is the probability measure $\nu^{\rm{RCM}}_{G,\alpha,q}$ on
$\{0,1\}^E$ given by
$$
\nu^{\rm{RCM}}_{G,\alpha,q}(\{\eta(e)\}_{e\in E}):=\alpha^{N_1}(1-\alpha)^{N_2}q^C/Z
$$
where $N_1$ is the number of edges in state 1, $N_2$ is the number of edges in state 0, $C$ is the 
resulting number of connected clusters and $Z=Z(G,\alpha,q)$ is a normalization constant.
\end{df}

Note, if $q=1$, this is simply an i.i.d.\ process with parameter $\alpha$. We think of $\nu^{\rm{RCM}}_{G,\alpha,q}$ 
as an RER on $V$ by looking at the clusters of the percolation realization; i.e., $v$ and $w$ are in the same 
partition if there is a path from $v$ to $w$ using edges in state 1. 

The following theorem from \cite{FK} tells us that the Ising Model with $J\ge 0$ and $h=0$ is indeed 
a color process. We however must identify $-1$ with $0$. See also \cite{ES}.

\begin{thm}\label{t.FKIsing} (\cite{ES},\cite{FK})
For any graph $G$ and any $J\ge 0$, 
$$
\mu_{G,J,0}=\Phi_{1/2}(\nu^{\rm{RCM}}_{G,1-e^{-2J},2}).
$$
\end{thm}

See \cite{OHRCrep} for a nice survey concerning various random cluster representations.
We remark that while for all $p$, $\Phi_{p}(\nu^{\rm{RCM}}_{G,\alpha,2})$ is of course a color process,
we do not know if this corresponds to anything natural when $p\neq \frac{1}{2}$. We mention
that, if $G$ is the complete graph, then an alternative way to see that the Ising model with $J\ge 0$
and 0 external field is a color process is to combine Theorem~\ref{t.mainp12} later in this paper
with the fact that the Ising model on the complete graph
can be extended to an infinitely exchangeable process. This latter fact was proved in \cite{Pap}
where the technique is credited to Kac \cite{Kac}; see also Theorem 1.1 in \cite{LST}.
We end by mentioning that for the Ising model on the complete graph on 3 vertices, there are other RERs, besides
the random cluster model, that generate it and that in some sense, the 
random cluster model is not the most natural generating RER; see remark (iii) after 
Question~\ref{q.ising}.

\subsubsection{The Fuzzy Potts Model}

Again for simplicity, we stick to finite graphs here and so let $G=(V,E)$ be a finite graph.

\begin{df}\label{df.Potts} For $q\in \{2,3,\ldots,\}$,
the $q$-state Potts model on $G=(V,E)$ with coupling constant $J$ (and no external field)
is the probability measure $\mu^{\rm{Potts}}_{G,J,q}$ on $\{1,\ldots,q\}^V$ given by
$$
\mu^{\rm{Potts}}_{G,J,q}(\{\eta(v)\}_{v\in V}):=
e^{J\sum_{\{v,w\}\in E}I_{\{\eta(v)=\eta(w)\}}}/Z
$$
where $Z=Z(G,J,q)$ is a normalization constant.
\end{df}

\begin{df}\label{df.PottsFuzzy}
For $G,q$ and $J$ as in Definition~\ref{df.Potts} and parameter $\ell \in \{1,\ldots,q-1\}$,
the fuzzy $q$-state Potts model on $G$ with parameters $J$ and $\ell$,
denoted by $\mu^{\rm{Potts,Fuzzy}}_{G,J,q,\ell}$, is obtained by taking a 
realization from $\mu^{\rm{Potts}}_{G,J,q}$ and changing each $i\in \{1,\ldots,\ell\}$ to a 1 and
each $i\in \{\ell+1,\ldots,q\}$ to a 0.
\end{df}

It turns out that $\mu^{\rm{Potts,Fuzzy}}_{G,J,q,\ell}$ is also a color process for $J\ge 0$.

\begin{thm}\label{t.PottsFuzzy} (\cite{ES},\cite{FK})
For any graph $G$, and any $J\ge 0,q$ and $\ell$ as above,
$$
\mu^{\rm{Potts,Fuzzy}}_{G,J,q,\ell}=\Phi_{\frac{\ell}{q}}(\nu^{\rm{RCM}}_{G,1-e^{-2J},q}).    
$$
\end{thm}

This follows easily from an extension of Theorem~\ref{t.FKIsing} which says that one can obtain a 
realization of $\mu^{\rm{Potts}}_{G,J,q}$ by taking a realization of $\nu^{\rm{RCM}}_{G,1-e^{-J},q}$ 
and ``coloring'' each cluster independently and uniformly from $\{1,\ldots,q\}$.
We again remark that while for all $p$, $\Phi_{p}(\nu^{\rm{RCM}}_{G,\alpha,q})$ is 
also of course a color process, we do not know if this corresponds to anything natural when $p$ is not of the 
form $\frac{\ell}{q}$.

\subsubsection{The (Classical) Divide and Color Model}

Unlike the previous examples discussed in this subsection, this model is {\sl defined} as 
a color process. In this model, which was introduced and studied in \cite{OH01}, one first performs
ordinary percolation with some parameter $\alpha$ on a finite or infinite graph $G$ and then
considers the RER corresponding to the clusters which result. The divide and color model is then defined to be 
the color processes coming from this RER as $p$ varies. Of course, using the terminology of the
previous two examples, this is simply $\Phi_{p}(\nu^{\rm{RCM}}_{G,\alpha,1})$. Some papers dealing with 
this model are the following: \cite{B}, \cite{BBT1} and \cite{BCM}.

\subsubsection{Stationary distributions for the Voter Model}

The Voter Model on ${\mathbb Z}^d$ is a continuous time Markov process with
state space  $\{0,1\}^{{\mathbb Z}^d}$; an element of $\{0,1\}^{{\mathbb Z}^d}$
specifies for each location (voter) in ${\mathbb Z}^d$ whether it is in state
0 or 1 representing two possible opinions. Heuristically, the Markov process
evolves as follows: each location in ${\mathbb Z}^d$ at rate 1 chooses a neighbor at 
random and then changes its state to that of its neighbor. (If the 
chosen neighbor has the same state, then nothing happens.)  A detailed description
of this process and the results described below can be found in \cite{DurrIPS},
\cite{LiggIPS} and \cite{LiggIPSeasy}. Clearly, the two states consisting of all 0's or of all 1's
are fixed states and hence the two point masses at these configurations as well as their convex 
combinations are stationary distributions. It turns out that in dimensions 1 or 2, these are the only
stationary distributions while in $d\ge 3$, there is a continuum of extremal 
stationary distributions indexed by $[0,1]$, denoted by $\{\mu_p\}_{p\in [0,1]}$. For each $p$,
$\mu_p$ is a translation invariant ergodic measure and is obtained by starting the Markov
process i.i.d.\ with density $p$ and taking the limiting distribution as time goes to infinity. This 
dichotomy between $d\le 2$ and $d\ge 3$ is exactly due to the recurrence/transience dichotomy in these cases. 

While it is by no means obvious, it turns out, based on the analysis of the voter model
carried out in the above references, that for each $d\ge 3$, there is an RER $\nu_d$ 
on ${\mathbb Z}^d$ such that for each $p\in [0,1]$, $\mu_p=\Phi_p(\nu_d)$.
This is also true for $d\le 2$ but then $\mu_p$ is taken to be the (nonergodic)
measure corresponding to a $(p,1-p)$ convex combination of the point mass at all 1's and the
point mass at all 0's and $\nu_d$ is concentrated on the partition which has only one partition element,
all of $\zd$. For all $d\ge 1$, the RER $\nu_d$ corresponds to ``coalescing random walks''
and is described as follows. Start independent continuous time rate 1 simple random walkers at 
each location of ${\mathbb Z}^d$, any two of which coalesce upon meeting.
Run the random walkers until time $\infty$ and then declare two locations
$x,y\in {\mathbb Z}^d$ to be in the same partition if the two random walkers starting at 
$x$ and $y$ ever coalesce. Note that for $d\le 2$ we have, due to recurrence, that this yields
one partition element, ${\mathbb Z}^d$, which is consistent with our description of $\nu_d$ above.

For $d\ge 3$, all the equivalence classes will be infinite with 0 density. Transience of random walk
implies clusters must have 0 density. The formula for return probabilities easily yields the fact that the expected
size of the cluster of the origin is infinite. Finally, the fact that the cluster size is in fact 
infinite a.s.\ can be found in \cite{Griff}.

\subsubsection{Random Walk in Random Scenery}

Let $(X_i)_{i\ge 1}$ be an i.i.d.\ sequence of random variables taking values in
${\mathbb Z}^d$. Let $(S_n)_{n\ge 1}$ be the associated random walk defined by
$S_0=0$ and $S_n=\sum_{i=1}^n X_i$ for $n\ge 1$. Next, let $\{C^p_z\}_{z\in {\mathbb Z}^d}$ 
be an i.i.d.\ process taking the value $1$ with probability $p$ and taking the value $0$ 
with probability $1-p$. Finally, letting, for $k\ge 0$, $Y^p_k:=C^p_{S_k}$, we call
$(Y^p_k)_{k\ge 0}$ ``Random Walk in Random Scenery'' since the process gives the ``scenery'' 
at the location of the random walker.

It turns out that $(Y^p_k)_{k\ge 0}$ is also in fact a color process which can be seen as follows.
We define an RER $\nu$ on ${\mathbb N}$ by declaring $i,j\ge 0$ to be in the same partition if 
$S_i=S_j$. It is then straightforward to see that $(Y^p_k)_{k\ge 0}$ has distribution $\Phi_p(\nu)$.

Although it is not so natural when thinking of random walk in random scenery,
it is sometimes useful to have the index set being ${\mathbb Z}$ instead of ${\mathbb N}$ which
can be done as follows. One starts with an i.i.d.\ process $(X_i)_{i\in {\mathbb Z}}$ and then 
defines $S_n$ as above for $n\ge 0$ and for $n \le -1$ to be $-\sum_{i=n+1}^0 X_i$. Finally, one
defines $Y^p_k$ to be $C^p_{S_k}$ for any $k\in {\mathbb Z}$. The strange definition of $S_n$ for negative 
$n$ in fact insures that $(Y^p_k)_{k\in {\mathbb Z}}$ is a stationary process. Moreover, the 
process $(X_k,Y^p_k)_{k\in {\mathbb Z}}$ is also a stationary process and is called a
generalized $\ttinv$-process.  (The name $\ttinv$ comes from the case of simple random
walk in $1$ dimension where $T$ denotes the left shift by 1 of $\{C^p_z\}_{z\in {\mathbb Z}}$:
the idea then is that from the walker's perspective, the latter sequence is shifted to the
left or right depending on the step of the walker.)
One can generalize further by allowing $(X_i)_{i\in {\mathbb Z}}$ 
to be an arbitrary stationary process rather than requiring it to be i.i.d., in which case
the random walk in random scenery would still be a color process.

If $(X_i)_{i\in {\mathbb Z}}$ yields a recurrent random walk, then a.s.\ all the equivalence 
classes are infinite and have 0 density (provided $X_1$ is not identically 0),
while if $(X_i)_{i\in {\mathbb Z}}$ yields a transient random walk, then all the equivalence 
classes are finite a.s.

\subsection{Summary of paper}

In this subsection, we summarize the different sections of the paper.

Section~\ref{s.finitecase} deals exclusively with the case that $V$ is the finite set $[n]:=\{1,2,\ldots,n\}$. 
A first natural question is whether, for fixed $p$, the map $\Phi_p\,:\,\rer_{[n]}\to \cp_{[n],p}$ is 
injective or not. One can also ask this same question when $\rer_{[n]}$ is replaced with 
$\rer_{[n]}^{\exch}$. Moreover, one can also address the question of whether there can be
two distinct (exchangeable) RERs such that their corresponding color processes agree for \emph{all} 
values of $p$. For each of these questions, we identify a phase transition in $n$. These are given in
Theorem~\ref{t.bigfinitetheorem} which is the main result in the finite case. 
We also obtain more refined results in this section as well as develop some general results. 

In Section~\ref{s.exchangeable}, we stick to color processes arising from exchangeable RERs on $\N$.
We first there remind the reader of Kingman's characterization of such RERs; see Theorem~\ref{t.kingman}.
Some of the obtained results are as follows. For $p=1/2$, it is shown that the set of color processes are 
exactly the collection of exchangeable processes which exhibit 
$0\ks 1$-symmetry; see Theorem~\ref{t.mainp12}.
While Proposition~\ref{p.Russ} tells us that, for each $p\in (0,1)$, $\Phi_p$ is injective when restricted to
the extremal elements of $\rer_{\N}^{\exch}$ (the so-called paint-boxes), it is shown that, for $p=1/2$,
$\Phi_p$ is highly non-injective on $\rer_{\N}^{\exch}$ and the subset where ``$\Phi_p$ is injective'' is
characterized; see Theorem~\ref{p.unprop}. It turns out however that the behavior for $p\neq 1/2$ seems 
quite different and $\Phi_p$ is ``much more injective''.

In Section~\ref{s.conn}, we look at a very specific type of color process; namely those where $V={\mathbb Z}$
and the classes are connected and hence are simply intervals.

In Section~\ref{s.dom}, we study the question of stochastic domination of product measures
for the set of color processes. More specifically, given an RER and $p\in (0,1)$, we consider the maximum density
product measure which the corresponding color process dominates. Of particular interest is the limit, as $p\to 1$
of this maximum density which often is not 1; this is related to the large deviation picture of the number of 
clusters intersecting a large box. In addition to obtaining various general results, the case of 
$\rer_{\N}^{\exch}$ as well as our various models from Subsection~\ref{ss.Examples} are analyzed in detail.

In Section~\ref{s.transfer}, we move into our ``ergodic theory'' section. Here we consider stationary color 
processes indexed by $\zd$ and study their ergodic behavior. Some of the obtained results are as follows. 
Theorem~\ref{t.ergod1} tells us that if there is positive probability of a positive density cluster, then
ergodicity is ruled out. On the other hand, Theorem~\ref{t.finiteclust} tells us that if all clusters are 
finite a.s., then the color process inherits all of the ergodic properties of the generating RER. These two 
results tell us that the interesting cases are when the RER has infinite clusters but all with 0 
density a.s. Various results in this case are obtained as well as other questions looked at.

Finally, in Section~\ref{s.ques}, we present a number of questions and further directions which we feel might 
be interesting to pursue.

\section{The finite case}\label{s.finitecase}

In this section, we restrict ourselves to the case when $V$ is finite. In the first and main subsection,
we state and prove Theorem~\ref{t.bigfinitetheorem} concerning uniqueness of the representing
RER and present further refined results.
The second subsection deals with some other general results in the finite case.

\subsection{Uniqueness of the representing RER in the finite case}

It is natural to ask, for various color processes, whether the representing RER is unique.
We give in this subsection fairly detailed answers to this in the finite case.
Recall $p\in (0,1)$.

We begin by giving an alternative description of $\rer^{\exch}_{[n]}$ which is as follows. 
A partition of the \emph{integer} $n$ is given by an integer $s\ge 1$ and positive integers 
$k_1\le k_2\le \ldots \le k_s$ such that $\sum_i k_i=n$. We denote by 
$[k_s\ks k_{s-1}\ks \ldots\ks k_1]$ the set of all partitions of (the set) $[n]$ that can be 
written as $\{C_1,\ldots,C_s\}$ where $|C_i|=k_i$. It is easy to see that
$\rer_{[n]}^{\exch}$ are those $\nu\in\rer_{[n]}$ such that if 
$\pi$ and $\pi'$ belong to the same $[k_s\ks k_{s-1}\ks \ldots \ks k_1]$, then
$\nu(\pi)=\nu(\pi')$. In this way, $\rer_{[n]}^{\exch}$ 
can be identified with probability measures on partitions of the integer $n$. 

The following is the main result in the finite case.

\begin{thm}\label{t.bigfinitetheorem}
{\bf (A).} The map
$$\Phi_{1/2}\,:\,\rer_{[n]}\to \cp_{[n],1/2}$$
is injective if $n=2$ and non-injective if $n\ge 3$.

{\bf (B).} The map
$$\Phi_{1/2}\,:\,\rer_{[n]}^{\exch}\to\cp_{[n],1/2}^{\exch}$$
is injective if $n=2$ and non-injective if $n\ge 3$.

{\bf (C).} If $p\neq 1/2$, then the map
$$\Phi_p\,:\,\rer_{[n]}\to \cp_{[n],p}$$
is injective for $n=2,3$ and non-injective for $n\ge 4$.

{\bf (D).} If $p\neq 1/2$, then the map
$$\Phi_p\,:\,\rer_{[n]}^{\exch}\to \cp_{[n],p}^{\exch}$$
is injective if $n=2,3$ and non-injective if $n\ge 4$.

{\bf (E).} There are $\nu_1\neq \nu_2 \in \rer_{[n]}$ such that 
$\Phi_p(\nu_1)=\Phi_p(\nu_2)$ for all $p\in [0,1]$ if and only if $n\ge 4$. 

{\bf (F).} There are $\nu_1\neq \nu_2 \in \rer_{[n]}^{\exch}$ such that 
$\Phi_p(\nu_1)=\Phi_p(\nu_2)$ for all $p\in [0,1]$ if and only if $n\ge 6$. 
\end{thm}

\pf Before starting with any of the parts, we first show that in each of these
parts, we have monotonicity in $n$; for {\bf (A)}-{\bf (D)}, this means that the relevant
map being non-injective for $n$ implies it is non-injective for $n+1$ and for
{\bf (E)} and {\bf (F)}, this means that if we have such a pair of measures as described 
for $n$, then we have such a pair for $n+1$. To do this, we first note that there
are simple injections from $\rer_{[n]}$ into $\rer_{[n+1]}$ and from 
$\rer_{[n]}^{\exch}$ into $\rer_{[n+1]}^{\exch}$. For the first one, given 
$\nu\in\rer_{[n]}$, we can let $T(\nu)\in\rer_{[n+1]}$ be such that $n+1$ is its own
cluster and the partition on $[n]$ is distributed according to $\nu$. 
For the second one, given 
$\nu\in\rer_{[n]}^{\exch}$, we construct $S(\nu)\in\rer_{[n+1]}^{\exch}$ as follows.
For every partition $s,k_1, \ldots, k_s$ of $n$, let 
$S({\nu})([k_s\ks k_{s-1}\ks\ldots\ks k_1 \ks 1]) :=\nu([k_s\ks k_{s-1}\ks\ldots\ks k_1])$. 
(Note that, unlike for $T$, the projection of $S(\nu)$ to $[n]$ is not $\nu$.)
Finally, it is easy to check that if $\mu$ and $\nu$ give the same color process
in {\bf (A)}-{\bf (D)} or satisfy the properties in {\bf (E)} or {\bf (F)}, 
then this will also hold
for the extended measures $T(\mu)$ and $T(\nu)$ or $S(\mu)$ and $S(\nu)$, as the case may be.

{\bf (A).} In view of the above monotonicity, we only need to look at $n=2$ and 3. 
First consider the case $n=2$. We represent $\nu\in \cp_{[2]}$ as the probability vector $(q_1,q_2)$ 
where $q_1:=\nu(\{\{1\},\{2\}\})$ and $q_2:=1-q_1=\nu(\{\{1,2\}\})$. Observe that $\Phi_p(\nu)((0,1))=q_1p(1-p)$. 
The injectivity now follows immediately, not just for $p=1/2$ but for all $p\in (0,1)$, since 
$\nu$ is determined by $q_1$. 

Next, consider the case $n=3$. We write $\nu\in \rer_{[3]}$ as $(q_1,\ldots,q_5)^t$ where 
$q_1:=\nu(\{\{1\},\{2\},\{3\}\})$, $q_2:=\nu(\{\{1,2\},\{3\}\})$, $q_3:=\nu(\{\{1\},\{2,3\}\})$, 
$q_4:=\nu\{\{\{1,3\},\{2\}\}\}$ and $q_5:=\nu(\{1,2,3\})$. In addition, we write $\Phi_{1/2}(\nu)$ 
as $(p_{111},p_{110},p_{101},p_{011},p_{100},p_{010},p_{001},p_{000})^t,$ where 
$p_{ijk}=\Phi_{1/2}(\nu)((i,j,k))$. Let $\nu_1=(2/3,0,0,0,1/3)$ and $\nu_2=(0,1/3,1/3,1/3,0)$.
Note in fact, $\nu_1,\nu_2\in \rer_{[3]}^{{\rm exch}}$.  Straightforward calculations which are 
left to the reader give that 
$$\Phi_{1/2}(\nu_1)=\Phi_{1/2}(\nu_2)= (1/4,1/12,1/12,1/12,1/12,1/12,1/12,1/4),$$
and the non-injectivity follows. 

{\bf (B).} Again, we only need to look at $n=2$ and 3. These are however
contained in {\bf (A)} since (i) it is easier to be injective on a subset 
(in fact, in this case, $\rer_{[2]}=\rer_{[2]}^{\exch}$) and (ii)
the examples there showing non-injectivity for $n=3$ are in fact exchangeable. 

{\bf (C).} This time, by monotonicity, we only need to look at $n=3$ and 4.
For $n=3$, $$\Phi_p(\nu)=L_p \nu,$$

where $L_p$ is the matrix given by 

\begin{equation}\label{e.3matr}
L_p=\left( \begin{array}{ccccc} p^3 & p^2 & p^2 & p^2 & p \\ 
p^2(1-p) & p(1-p) & 0 & 0 & 0 \\ p^2(1-p) & 0 & 0 & p(1-p) & 0 \\ p^2(1-p) & 0 & p(1-p) & 0 & 0 \\ p(1-p)^2 & 0 & p(1-p) & 0 & 0 \\ p(1-p)^2 & 0 & 0 & p(1-p) & 0 \\ p(1-p)^2 & p(1-p) & 0 & 0 & 0 \\ (1-p)^3 & (1-p)^2 & (1-p)^2 & (1-p)^2 & (1-p) \end{array} \right),
\end{equation}
where we use the same notation and ordering as in ({\bf A}).
Suppose that $p\neq 1/2$. Let $\nu=(q_1,\ldots,q_5)^t$ and $\nu'=(q_1',\ldots,q_5')^t$. We must show that 
if $\Phi_p(\nu)=\Phi_p(\nu')$, then $\nu=\nu'$. So suppose that $\Phi_p(\nu)=\Phi_p(\nu')$. Denote the entries of 
$\Phi_p(\nu')$ by $p_{111}',p_{110}',\ldots$. Calculating the entries in $\Phi_p(\nu)$ and $\Phi_p(\nu')$ 
(using~\eqref{e.3matr}) gives $p_{011}=p^2(1-p)q_1+p(1-p)q_3$ and $p_{100}=p(1-p)^2 q_1+p(1-p) q_3$, and the 
same formulas for $p_{011}'$ and $p_{100}'$ with $q_1$ and $q_3$ replaced with $q_1'$ and $q_3'$. Observe that 
$$
p_{011}-p_{100}= (2p-1)p(1-p)q_1,
$$
and 
$$
p_{011}'-p_{100}'=(2p-1)p(1-p)q_1'.
$$

Since $\Phi_p(\nu)=\Phi_p(\nu')$ and $p\neq 1/2$, we get that $q_1=q_1'$. 
From the facts that $p_{100}=p_{100}'$ and $q_1=q_1'$ it follows that $q_3=q_3'$. By symmetry, it then follows 
that $q_2=q_2'$ and $q_4=q_4'$. Hence, $\nu=\nu'$.

For the $n=4$ case, we first let $g(p):=p(1-p)$ and then define $\nu_1$ and 
$\nu_2=\nu_2(p)\in \rer_{[4]}^{\exch}$ as follows. Let 
$\nu_1([4])=\nu_1([3\ks 1])=\nu_1([2\ks 2])=\nu_1([2\ks 1\ks 1])=\nu_1([1\ks 1\ks 1\ks 1])=1/5,$ and 
let $\nu_2([4])=1/5+g(p)/10$, $\nu_2([3\ks 1])=1/5-2 g(p)/5$, $\nu_2([2\ks 2])=1/10+3 g(p)/10$, 
$\nu_2([2\ks 1\ks 1])=2/5$ and $\nu_2([1\ks 1\ks 1\ks 1])=1/10.$ Straightforward calculations which are left 
to the reader show that for all $p$, $\Phi_p(\nu_1)=\Phi_p(\nu_2(p))$, from which the non-injectivity follows.

We mention that the (nonexchangeable) construction in part ({\bf E}) below also could
have been used here in this case; however, we would still need the above for (D).

{\bf (D).} Again, by monotonicity, we only need to look at $n=3$ and 4. 
These are however contained in {\bf (C)} since (i) it is easier to be injective on a subset 
and (ii) the examples there showing non-injectivity for $n=4$ are in fact exchangeable. 

{\bf (E).} 
Again, by monotonicity, we only need to look at $n=3$ and 4. 
The case $n=3$ follows from Part ({\bf C}).   
Now consider the case $n=4$ and define $\nu_1$ by letting
$$
\nu_1( \{ \{ 1,3\}, \{2 \}, \{4\} \} )=\nu_1( \{ \{1 \} ,\{3\} , \{ 2,4\}\} )=1/3
$$ 
and 
$$
\nu_1(\{\{ 1,2\}, \{3,4 \}\})=\nu_1(\{\{ 1,4\}, \{2,3 \}\})=1/6.
$$ 
Then define $\nu_2$ by letting
\begin{eqnarray*}\lefteqn{\nu_2( \{\{1,2 \} , \{ 3\} ,\{4\}  \})=\nu_2(\{  \{ 1\} ,\{2,3\}, \{4\} \} )}\\ & & =\nu_2(\{  \{ 1\}, \{2 \},\{3,4\}   \})=\nu_2(\{  \{ 1,4\}, \{2 \},\{3\}   \})=1/6,\end{eqnarray*}  and 
$$
\nu_2(\{  \{ 1,3\}, \{2,4\}   \})=1/3.
$$
Observe that $\nu_1$ and $\nu_2$ are each invariant under rotations and reflections.
Straightforward calculations show that for $i=1,2$,  $\Phi_p(\nu_i)((1,1,1,1))=2p^3/3+p^2/3$, 
$\Phi_p(\nu_i)((0,1,1,1))=(1-p)p^2/3$, $\Phi_p(\nu_i)((1,1,0,0))=p(1-p)/6$ and 
$\Phi_p(\nu_i)((1,0,1,0))=p(1-p)/3$. Since $\nu_1$ and $\nu_2$ are each invariant under 
rotations and since the roles of $1$ and $0$ get switched when $p$ is replaced by $1-p$, 
we conclude that $\Phi_p(\nu_1)=\Phi_p(\nu_2)$ for all $p$.

{\bf (F).}  By monotonicity, we only need to look at $n=5$ and 6. 
For the case of $n= 5$, we will make important use of Lemma~\ref{l.js} below, which we believe
can be of independent interest. We 
state and prove it after the completion of the present proof. Assume now, by way
of contradiction, that there exist $\nu_1\neq \nu_2$ in $\rer_{[5]}^{\exch}$ such that
$\Phi_p(\nu_1)=\Phi_p(\nu_2)$ for all $p$. We now want to ``singularize'' $\nu_1$ and $\nu_2$.
Let $m$ be the largest subprobability measure dominated by both $\nu_1$ and $\nu_2$.
Since $\Phi_p$ is affine, it is easy to see that we also have that
$\Phi_p(\frac{\nu_1-m}{|\nu_1-m |_1})=\Phi_p(\frac{\nu_2-m}{|\nu_2-m |_1})$ 
for all $p$. The latter two measures are singular with respect to each other.
The conclusion is that we may now assume that we have
$\nu_1\neq \nu_2$ in $\rer_{[5]}^{\exch}$ which are singular and such that
$\Phi_p(\nu_1)=\Phi_p(\nu_2)$ for all $p$. 

We now make use of Lemma~\ref{l.js} several times. The application of part (i) is always made with $S=[5]$. 
By Lemma~\ref{l.js} (i) and the assumed singularity between $\nu_1$ and $\nu_2$, we can conclude that $\nu_1$ 
and $\nu_2$ both vanish on $[5]$, $[2\ks 1\ks 1\ks 1]$ and $[1\ks 1\ks 1\ks 1\ks 1]$. Also, Lemma~\ref{l.js} (ii) 
tells us that $\nu_1$ and $\nu_2$ give the same measure to $[4\ks 1]$ and hence they both vanish there by singularity. 
At this point, we know that both $\nu_1$ and $\nu_2$ are concentrated on $[3\ks 2]$, $[3\ks 1\ks 1]$ and 
$[2\ks 2\ks 1]$. Again using Lemma~\ref{l.js} (i) and singularity shows that $\nu_1$ and $\nu_2$ vanish on $[3\ks 2]$. 
Next, Lemma~\ref{l.js} (ii) and singularity then shows that $\nu_1$ and $\nu_2$ vanish on $[3\ks 1\ks 1]$. 
Hence, both $\nu_1$ and $\nu_2$ are both concentrated on $[2\ks 2 \ks 1]$ which is a contradiction since 
they are singular probability measures.

For the case $n= 6$ we define two probability measures $\nu_1$ and $\nu_2$ on partitions of the integer 
$6$ as follows. First let  
$$
\nu_1([4\ks 2])=1/3\mbox{ and }\nu_1([3\ks 2\ks 1])=2/3.
$$ 
Then let 
$$
\nu_2([4\ks 1\ks 1])=\nu_2([3\ks 3])=\nu_2([2\ks 2\ks 2])=1/3.
$$ 

Let $A_k$ be the event that there are exactly $i$ ones in the color process. Exchangeability implies that 
if $\Phi_p(\nu_1)(A_k)=\Phi_p(\nu_2)(A_k)$ for  $k=0,1,\ldots 6$ and all $p$, then $\Phi_p(\nu_1)=\Phi_p(\nu_2)$ 
for all $p$. Simple calculations left to the reader show that for $i=1,2$,
$$
\Phi_p(\nu_i)(A_6)=\frac{p^2}{3}+\frac{2p^3}{3},\mbox{ }\Phi_p(\nu_i)(A_5)=\frac{2p^2(1-p)}{3}\mbox{ and }\Phi_p(\nu_i)(A_4)=\frac{p}{3}+\frac{p^2}{3}-\frac{2p^3}{3}.
$$

Since we have, for $i=1,2$, $\Phi_{p}(\nu_i)(A_0)=\Phi_{1-p}(\nu_i)(A_6)$, $\Phi_{p}(\nu_i)(A_1)=\Phi_{1-p}(\nu_i)(A_5)$, 
$\Phi_{p}(\nu_i)(A_2)=\Phi_{1-p}(\nu_i)(A_4)$ and $\Phi_{p}(\nu_i)(A_3)=1-\sum_{k\neq 3} \Phi_{p}(\nu_i)(A_k)$, 
we can finally conclude that $\Phi_p(\nu_1)= \Phi_p(\nu_2)$ for all $p$. 
\qed

\bigskip

Next, we give the lemma which was used repeatedly in the proof of {\bf (F)} 
in Theorem~\ref{t.bigfinitetheorem} above.

\begin{lma}\label{l.js}
Let $\nu_1,\nu_2\in \rer_{[n]}$. Then each one
of the following conditions implies that
$\Phi_p(\nu_1)\neq \Phi_p(\nu_2)$ for some $p$. \\
(i). For some $S\subseteq [n]$, the distribution
of the number of equivalence classes of $\pi_S$ is different under $\nu_1$ and $\nu_2$.\\
(ii). For some $T\ge 1$, 
the mean of the number of equivalence classes whose size is equal to $T$ is 
different under $\nu_1$ and $\nu_2$.\\
(iii). For some $C\subseteq [n]$, the probability that $C$ is an equivalence class
is different under $\nu_1$ and $\nu_2$.\\
\end{lma}

\pf (i).  For the given $S$, let $F$ be the event that the color process is identically $1$ on $S$, 
and let $N$ be the number of equivalence classes of $\pi_S$. Then for all $p$ and $i=1,2$,
$$
\Phi_p(\nu_i)(F)=E_{\nu_i}(p^{N}).
$$  
By assumption, some coefficient in these two polynomials in $p$ are different
and hence $\Phi_p(\nu_1)$ and $\Phi_p(\nu_2)$ give $F$ different probability for some $p$. 

(ii). For the given $T$ let $X$ be the number of equivalence classes of size equal to $T$, and
suppose that $E_{\nu_1}(X)\neq E_{\nu_2}(X)$.  Let $K$ be the event that the color process contains 
exactly $T$ $1$'s. Then $\Phi_p(\nu_1)(K)=p E_{\nu_1}(X) + O(p^2)$ as $p\to 0$
and similarly for $\nu_2$. We conclude that $\Phi_p(\nu_1)$ and $\Phi_p(\nu_2)$ give the event 
$K$ different probability for small $p$.

(iii). For the given $C$, let $D$ be the event that $C$ is a cluster and let $H$ 
be the event that the color process is identically 1 exactly on $C$. Then 
$\Phi_p(\nu_1)(H)=\nu_1(D) p +O(p^2)$ as $p\to 0$ and similarly for $\nu_2$. 
We conclude that $\Phi_p(\nu_1)$ and $\Phi_p(\nu_2)$ give $H$ different probability for small $p$. \qed

\medskip\noindent
\begin{remark} (i) Concerning Theorem~\ref{t.bigfinitetheorem}(E,F),
it might at first be surprising that one can find distinct and exchangeable $\mu$ 
and $\nu$ such that $\Phi_p(\mu)= \Phi_p(\nu)$ for all $p$ since there are infinitely many $p$. However,
since all the functions of $p$ that arise are polynomials in $p$ of degree at most $n$,
we are essentially in a finite dimensional situation. Another way to see this is that
if $\Phi_p(\mu)= \Phi_p(\nu)$ for $n+1$ many values of $p$, then this holds for all $p$.  \\
(ii). We describe how we came up with the example for the $n=6$ case.  The negations of 
conditions (i) and (ii) of Lemma~\ref{l.js} for $S=[6]$ give a set of linear equations that must hold in order 
for two RERs to have the same color process. With the help of Mathematica, the nullspace of the coefficient 
matrix of the linear system was calculated. By looking at the positive and negative part of one of the vectors 
of the nullspace, the two measures $\nu_1$ and $\nu_2$ were then constructed.
\end{remark}

The next result, Proposition~\ref{p.kerprop}, 
describes our injectivity results in more linear algebraic terms
and goes into more detail concerning what happens in the non-injective case.
In particular, in the case of non-injectivity, it is natural to try to identify 
``where $\Phi_p$ is non-injective''. The next definition captures this notion.

\begin{df}
Let $V$ be a finite or countable set. Let ${\mathcal R}\subseteq \rer_V$ and $p\in (0,1)$. 
We say that $\nu \in {\mathcal R}$ is $({\mathcal R},p)$-unique if 
$\Phi_p(\nu')\neq \Phi_p(\nu)$ for all $\nu'\in {\mathcal R}\setminus \{\nu\}$.
\end{df}

\begin{prop}\label{p.kerprop}
Let $n\ge 2$, $p\in (0,1)$ and consider the map 
$$
\Phi_p\,:\,\rer_{[n]}\to \cp_{[n]}.
$$
Noting that $\Phi_p$, being affine, extends to the vector space of signed measures on
${\rm Part}_{[n]}$ and denoting this extension by $\Phi_p^*$,
the following four statements hold:
\begin{enumerate}
\item[(i).] $\Phi_p$ is non-injective if and only if ${\rm Ker}(\Phi_p^*)\neq \{{\bf 0}\}$.
\item[(ii).] Suppose that $n\ge 2$ and $p\in (0,1)$. Then $\nu\in \rer_{[n]}$ is not 
$(\rer_{[n]},p)$-unique if and only if there is 
$v\in {\rm Ker}(\Phi_p^*)\setminus \{{\bf 0}\}$ 
such that $v_i\ge 0$ for all $i\in ({\rm supp} \,\nu)^c$. 
\item[(iii).] If ${\rm Dim}({\rm Ker}(\Phi_p^*))=1$, then there is a unique pair $\nu_1,\nu_2\in \rer_{[n]}$, 
singular with respect to each other, such that $\Phi_p(\nu_1)=\Phi_p(\nu_2)$.
\item[(iv).] If ${\rm Dim}({\rm Ker}(\Phi_p^*))\ge 2$, then there infinitely many distinct pairs 
$\nu_1,\nu_2\in \rer_{[n]}$, singular with respect to each other, such that $\Phi_p(\nu_1)=\Phi_p(\nu_2)$.

\end{enumerate}

Moreover, if ${\mathcal R}$ is a closed and convex subset of $\rer_{[n]}$
and $\Phi_{p,\langle {\mathcal R}\rangle}^*$ is the restriction of $\Phi_{p}^{*}$ to $\langle{\mathcal R}\rangle$, 
the subspace spanned by $\Rc$, then (i) and (ii) still hold with $\rer_{[n]}$, $\Phi_p$ and $\Phi_p^*$ replaced by $\Rc$,
$\Phi_p|_{\Rc}$ and $\Phi_{p,\langle \Rc \rangle}^{*}$. Also, if in addition $\Rc$ is such that that 
$\nu_1,\nu_2 \in \Rc$ and $\nu_1\neq \nu_2$ imply that
\begin{equation}\label{e.condition}
\frac{\nu_1- (\nu_1 \wedge \nu_2)}{|\nu_1- (\nu_1 \wedge \nu_2)|_1}\in \Rc, 
\end{equation}
then (iii) and (iv) hold with $\rer_{[n]}$ and $\Phi_p^*$ replaced by $\Rc$ and 
$\Phi_{p,\langle \Rc \rangle}^*$.
\end{prop}

\pf (i). First, ${\rm Ker}(\Phi_p^*)=\{{\bf 0}\}$ trivially implies injectivity. 
Now suppose that ${\rm Ker}(\Phi_p^*)\neq \{{\bf 0}\}$. 
Let $\nu\in \rer_{[n]}$ be such that if we let $\pi_1,\ldots$ be an enumeration of ${\rm Part}_{[n]}$ we have 
$\nu(\pi_i)\in (0,1)$ for all $i$. Pick $u\in {\rm Ker}(\Phi_p^*)\setminus \{{\bf 0}\}$. Since $\nu(\pi_i)\in (0,1)$ 
for all $i$ and ${\rm Part}_{[n]}$ is finite, we can pick $\epsilon>0$ such that $\nu(\pi_i)+\epsilon u_i>0$ for 
all $i$. Let $\nu'=\nu+\epsilon u$. It is easy to show that $\sum_i u_i=0$ for any $u\in {\rm Ker}(\Phi_p^*)$ and 
so we have $\nu'\in \rer_{[n]}$. Moreover, $\Phi_p(\nu')=\Phi_p(\nu)$, finishing the proof.

(ii).
Suppose that $\nu\in \rer_{[n]}$ is such that there is 
$v\in {\rm Ker}(\Phi_p^*)\setminus \{{\bf 0}\}$ with $v_i\ge 0$ for all $i\in ({\rm supp}\, \nu)^c$. 
In similar fashion as in the proof of part (i), we get that if $\epsilon>0$ is sufficiently small, 
then $\nu':=\nu+\epsilon v$ belongs to $\rer_{[n]}$ and moreover, $\Phi_p(\nu)=\Phi_p(\nu')$. 
Hence $\nu$ is not  $(\rer_{[n]},p)$-unique.

For the other direction, suppose that $\nu$ is not $(\rer_{[n]},p)$-unique. Then we can 
pick $\nu'\in\rer_{[n]}$ such that $\nu'\neq \nu$ and $\Phi_p(\nu)=\Phi_p(\nu')$ in 
which case ${\bf 0}\neq v:=\nu'-\nu\in {\rm Ker}(\Phi_p^*)$. Moreover, since $\nu'=v+\nu$ 
it follows that $v_i\ge 0$ for all $i\in ({\rm supp}\,\nu)^c$ since otherwise $\nu'$ 
would have a negative entry.
 
(iii). Suppose that ${\rm Dim}({\rm Ker}(\Phi_p^*))=1$. 
Pick $w\in {\rm Ker}(\Phi_p^*)\setminus \{{\bf 0}\}$. Write $w=w_{+}-w_{-}$ where $(w_{+})_i=w_i$ 
if $w_i\ge 0$ and $(w_{+})_i=0$ if $w_i<0$. Then, letting $\nu_1:=2 w_{+}/|w|_1$ and 
$\nu_2:=2 w_{-}/|w|_1$, we have $\nu_1,\nu_2\in \rer_{[n]}$, $\nu_1\neq\nu_2$ and since $w\in {\rm Ker}(\Phi_p^*)$ 
we have $\Phi_p(\nu_1)=\Phi_p(\nu_2)$. It is also clear that $\nu_1$ and $\nu_2$ are singular 
with respect to each other. It remains to prove uniqueness. For this, assume that 
$\nu_1',\nu_2'\in \rer_{[n]}$ satisfy $\Phi_p(\nu_1')=\Phi_p(\nu_2')$ and that $\nu_1'$ and 
$\nu_2'$ are singular with respect to each other. Since $\Phi_p$ is affine, 
$\nu_1'-\nu_2'\in {\rm Ker}(\Phi_p^*)$, and since ${\rm Dim}({\rm Ker}(\Phi_p^*))=1$ it follows 
that $\nu_1'-\nu_2'= c (\nu_1-\nu_2)$ for some $c\neq 0$. If $c>0$, then by singularity, 
$\nu_1'=c\nu_1$ and $c=1$. Hence, $\nu_1=\nu_1'$ and $\nu_2=\nu_2'$. Similarly, $c<0$ implies 
$\nu_1=\nu_2'$ and $\nu_2=\nu_1'$. Hence, the uniqueness is established.

(iv). Now instead assume that ${\rm Dim}({\rm Ker}(\Phi_p^*))\ge 2$. Let $v$ and $w$ be two 
linearly independent elements in ${\rm Ker}(\Phi_p^*)$. It follows that either $2 v_+/|v|_1$ differs
from $2 w_+/|w|_1$ or $2 v_-/|v|_1$ differs from $2 w_-/|w|_1$ (or both). Without loss of generality,
we assume the first.
For $a\ge 0$, let $u(a):=2 (a v +w)/|av+w|_1$ and let $\nu_1(a):=u(a)_{+}$ and let 
$\nu_2(a):=u(a)_{-}$, defined as in part (iii). Then for every $a$, $\nu_1(a),\nu_2(a)\in \rer_{[n]}$, 
$\Phi_p(\nu_1(a))=\Phi_p(\nu_2(a))$ and $\nu_1(a)$ and $\nu_2(a)$ are singular with respect to 
each other. Observe that $\nu_1(a)$ is continuous in $a$, $\nu_1(0)= 2 w_+/|w|_1$
and $\nu_1(a)\to 2 v_+/|v|_1$ as $a\to \infty$. The latter are distinct and hence $(\nu_1(a))_{a\ge 0}$ 
contains an uncountable collection of distinct elements from $\rer_{[n]}$. 

Finally we observe that the extensions mentioned to certain $\Rc\subseteq \rer_{[n]}$ require 
easy modifications of the given proofs.\qed

\begin{remark}
(i). Taking $\Rc\subset \rer_{[3]}$ to be 
$$
\Rc=\{\nu_1,\nu_2,\nu_3\}:=\{(1,0,0,0,0),(0,0,0,0,1),(0,1/3,1/3,1/3,0)\},
$$ 
we have that  ${\rm Ker}(\Phi_{1/2,\langle \Rc\rangle}^*)$ is nonempty (indeed, by Example~\ref{ex.n3} below 
we have that $2\nu_1+\nu_2-3\nu_3\in {\rm Ker}(\Phi_{1/2,\langle \Rc \rangle }^*))$ but
$\Phi_p$ is injective  on $\Rc$; hence we need some convexity assumption on $\Rc$.\\
(ii). If $\Rc$ is either the set of probability measures supported on some
fixed subset of ${\rm Part}_{[n]}$ or $\Rc$ is the set of probability measures invariant under
some group action (such as $\rer_{[n]}^{\exch}$), then all of the last 
conditions in Proposition~\ref{p.kerprop} hold and hence so do (i)-(iv).\\
(iii). An example of a closed and convex set $\Rc\subset \rer_{[3]}$ where (iii) fails when $p=1/2$ is 
$$
\{(q_1,\ldots,q_5)\in \rer_{[3]}\,:\,q_5\le \min(q_1,q_2,q_3,q_4)\}.
$$
($q_1,\ldots,q_5$ are defined as they were in the proof of Theorem~\ref{t.bigfinitetheorem}(A).)
To see this, first observe that $\nu_1:=(\frac{3}{7},\frac{1}{7},\frac{1}{7},\frac{1}{7},\frac{1}{7})$ 
and $\nu_2:=(\frac{1}{7},\frac{2}{7},\frac{2}{7},\frac{2}{7},0)$ are in $\Rc$ and
$\Phi_{1/2}(\nu_1)=\Phi_{1/2}(\nu_2)$. Hence ${\rm Ker}(\Phi_{1/2}^*|_{\Rc})$ has dimension at least 1
while this dimension is at most 1 since Example~\ref{ex.n3} (given below) shows that ${\rm Ker}(\Phi_{1/2}^*)$ 
has dimension 1. Now part (iii) of Proposition~\ref{p.kerprop} applied to $\rer_{[3]}$ gives that there is 
only one pair of singular measures in $\rer_{[3]}$ with the same $\Phi_{1/2}$ value, namely
$(\frac{2}{3},0,0,0,\frac{1}{3})$ and $(0,\frac{1}{3},\frac{1}{3},\frac{1}{3},0)$. 
Since the first is not in $\Rc$, we do not have such a singular pair there, showing (iii) fails.
As must be the case, (\ref{e.condition}) fails and one can immediately check that 
it fails for $\nu_1=(\frac{3}{7},\frac{1}{7},\frac{1}{7},\frac{1}{7},\frac{1}{7})$ 
and $\nu_2=(\frac{1}{7},\frac{2}{7},\frac{2}{7},\frac{2}{7},0)$, whose difference 
is in ${\rm Ker}(\Phi_{1/2}^*)$. However, it is easy to see that (iii) can never fail the ``other
way'', namely that if the dimension of the relevant kernel is 1, then there are at most one desired pair
of singular measures; to see this, one notes that the proof given goes through verbatim for any $\Rc\subset \rer_{[3]}$.
\end{remark}

\begin{example}\label{ex.n3}
As we saw in Theorem~\ref{t.bigfinitetheorem}, $\Phi_{1/2}\,:\,\rer_{[3]}\to \cp_{[3],1/2}$ 
is not injective. Using Proposition~\ref{p.kerprop}(ii),
we can determine exactly which $\nu\in \rer_{[3]}$ are $(\rer_{[3]},1/2)$-unique. 
Recall that we write $\Phi_{1/2}(\nu)=L_{1/2}\nu$. The first four rows of $L_{1/2}$
will be the same as the last four (unlike in the $p\neq 1/2$ case). 
The first four rows of $L_{1/2}$ are given by

\begin{equation*}
(L_{1/2})_{1\le i \le 4, 1\le j\le 5}=\left( \begin{array}{ccccc} 1/8 & 1/4 & 1/4 & 1/4 & 1/2 \\ 1/8 & 1/4 & 0& 0& 0 \\ 1/8 & 0 & 0 & 1/4 & 0 \\ 1/8 & 0 & 1/4 & 0 & 0  \end{array} \right).
\end{equation*}

Elementary algebraic calculations show that the kernel of $L_{1/2}$ is spanned by

\begin{equation}\label{e.kernel3}
\left( \begin{array}{c} 2 \\ -1 \\-1 \\ -1 \\ 1  \end{array} \right).
\end{equation}

Using Proposition~\ref{p.kerprop}(ii)
and~\eqref{e.kernel3} we can conclude that for $\nu\in \rer_{[3]}$:

\begin{enumerate}
\item If $|{\rm supp}\, \nu| =1$ then $\nu$ is $(\rer_{[3]},1/2)$-unique.
\item If $|{\rm supp}\, \nu| =2$ then  $\nu$  is not $(\rer_{[3]},1/2)$-unique if and only if ${\rm supp}\,\nu=\{1,5\}.$
\item If $|{\rm supp}\, \nu| =3$ then  $\nu$  is not $(\rer_{[3]},1/2)$-unique if and only if \\
${\rm supp}\,\nu=\{2,3,4\}, \{1,2,5\}, \{1,3,5\}$ or $\{1,4,5\}.$
\item If $|{\rm supp}\, \nu| =4$ then  $\nu$  is not  $(\rer_{[3]},1/2)$-unique.
\item If $|{\rm supp}\, \nu| =5$ then  $\nu$  is not $(\rer_{[3]},1/2)$-unique.
\end{enumerate}

\end{example}

Using $(iii)-(iv)$ of Proposition~\ref{p.kerprop} applied to $\rer_{[n]}$ and $\rer_{[n]}^{\exch}$, 
we can obtain the following corollary. 
This corollary only deals with cases where we already have established non-injectivity.

\begin{cor}

(i). If $p=1/2$ then there is a unique singular pair $\nu_1,\nu_2\in \rer_{[n]}$ such that 
$\Phi_p(\nu_1)=\Phi_p(\nu_2)$ if $n=3$ and infinitely many such pairs if $n\ge 4$.\\
(ii). If $p=1/2$ then there is a unique singular pair $\nu_1,\nu_2\in \rer_{[n]}^{\exch}$ such that 
$\Phi_p(\nu_1)=\Phi_p(\nu_2)$ if $n=3$ and infinitely many such pairs if $n\ge 4$.\\
(iii). If $p\neq 1/2$ then there are infinitely many distinct singular pairs 
$\nu_1,\nu_2\in \rer_{[n]}$ such that $\Phi_p(\nu_1)=\Phi_p(\nu_2)$ if $n\ge 4$.\\
(iv).  If $p\neq 1/2$ then there is a unique singular pair $\nu_1,\nu_2\in \rer_{[n]}^{\exch}$ such 
that $\Phi_p(\nu_1)=\Phi_p(\nu_2)$ if $n=4$ and infinitely many such pairs if $n\ge 5$.
\end{cor}

\pf First we show the following monotonicity property: If  $n$ is such that $\rer_{[n]}$ contains 
infinitely many pairs of singular measures $\nu_1,\nu_2\in \rer_{[n]}$ with 
$\Phi_p(\nu_1)=\Phi_p(\nu_2)$, then the same holds for $n+1$. To see this, assume that 
$\nu_1,\nu_2\in \rer_{[n]}$ are singular with $\Phi_p(\nu_1)=\Phi_p(\nu_2)$. Let 
$T\,:\,\rer_{[n]}\to \rer_{[n+1]}$ be the injection from the proof of 
Theorem~\ref{t.bigfinitetheorem}. Then it is straightforward to verify that $T(\nu_1)$ and 
$T(\nu_2)$ are singular and give the same color process. The same proof using the injection $S$ 
(instead of $T$) from the proof of Theorem~\ref{t.bigfinitetheorem} shows that the same 
monotonicity property holds for $\rer_{[n]}^{\exch}$.

In the general case, ($\rer_{[n]}$), the dimension of the domain of our
operator will be the number of partitions of the set $[n]$ and 
the dimension of the image space will be $2^n$.
In the exchangeable case, ($\rer_{[n]}^{\exch}$),
the dimension of the domain of our operator will be the number of partitions of the integer $n$ and 
the dimension of the image space will be $n+1$.

(i). By Example~\ref{ex.n3}, we have that ${\rm Dim}({\rm Ker}(\Phi_{1/2}^*))=1$ if $n=3$.
For $n=4$, we have a mapping from a 15-dimensional space to a 16-dimensional space.
However, since
$p=1/2$, the probability on the latter has a $0\ks 1$-symmetry and so the range is at most
8-dimensional. From this, we conclude that ${\rm Dim}({\rm Ker}(\Phi_{1/2}^*))\ge 7$ 
and hence $(i)$ follows from Proposition~\ref{p.kerprop}(iii,iv) and the above monotonicity.
We mention that Mathematica shows that indeed ${\rm Dim}({\rm Ker}(\Phi_{1/2}^*))=7$.

(ii). One also can check directly that
${\rm Dim}({\rm Ker}(\Phi_{1/2,\langle \rer_{[3]}^\exch \rangle}^*))=1$
(which essentially follows from (i) also). For $n=4$, 
$\Phi_{1/2,\langle \rer_{[4]}^\exch \rangle}^*$ maps from a
5-dimensional space to a 5-dimensional space and one easily checks that the 
range is 3-dimensional and therefore 
${\rm Dim}({\rm Ker}(\Phi_{1/2,\langle \rer_{[4]}^\exch \rangle}^*))=2$.
Hence $(ii)$ follows from Proposition~\ref{p.kerprop}(iii,iv) and the above monotonicity.

(iii). For $n=4$, $\Phi_{p}^*$ maps from a 15-dimensional space to a 16-dimensional space. 
Mathematica claims to give a basis (depending on $p$) for the kernel which is 3-dimensional.
One can then check by hand that this proposed basis is linearly independent and
belongs to the kernel.
Hence, $(iii)$ follows from Proposition~\ref{p.kerprop}(iii,iv) and the above monotonicity.
(Note that Mathematica is not needed for the formal proof.)

(iv). Finally, with $p\neq 1/2$, if $n=4$,
one can check by hand that $\Phi_{p,\langle \rer_{[4]}^\exch \rangle}^*$,
which maps from a 5-dimensional space to a 5-dimensional space, has a range which is 
4-dimensional and hence
$$
{\rm Dim}({\rm Ker}(\Phi_{p,\langle \rer_{[4]}^\exch \rangle}^*))=1.
$$
If $n=5$, $\Phi_{p,\langle \rer_{[5]}^\exch \rangle}^*$ 
maps a 7-dimensional space into a 6-dimensional space.
Mathematica claims to give a basis (depending on $p$) for the kernel which is 2-dimensional.
One can then check by hand that this proposed basis is linearly independent and
belongs to the kernel. Hence
$(iv)$ follows from Proposition~\ref{p.kerprop}(iii,iv) and the above monotonicity.
(Note that Mathematica is not needed for the formal proof.) \qed

\subsection{Other geneneral results in the finite case}

\begin{prop}\label{t.mainfinitetheorem2}
If $\mu\in {\mathcal P}(\{0,1\}^{[2]})$, then $\mu\in \cp_{[2]}$ if and only if $\mu$ satisfies 
non-negative pairwise correlations and $\mu((1,0))=\mu((0,1))$.
\end{prop}

\begin{proof} 
The "only if" direction is immediate. For the other direction, let 
$$
\nu(\{\{1\},\{2\}\})=\frac{\mu((0,1))}{(\mu((1,1))+\mu((0,1)))(\mu((0,1))+\mu((0,0)))}
$$
and $p=\mu((1,1))+\mu((0,1))$. Then the assumption of non-negative pairwise correlations implies that 
$\nu(\{\{1\},\{2\}\}) \le 1$ and a straightforward calculation shows that $\Phi_p(\nu)=\mu$, as desired.
\end{proof}

\begin{df}
A measure $\mu$ on $\{0,1\}^{n}$ is said to be 
exchangeable if it is invariant under all permutations of $[n]$.
\end{df}

If we move to $n=3$, then it turns out that non-negative pairwise correlations 
and exchangeability (the latter no longer being necessary for being a 
color process with $n=3$) do not suffice for being a color process as is 
shown by the following example.
We consider the distribution $\frac{1}{9}m_1+\frac{8}{9}m_2$ where $m_1,m_2$
are product measures with respective densities $.9$ and $.45$. 
This is exchangeable and has non-negative pairwise correlations. 
Since the marginals
are $1/2$ but the process does not exhibit $0\ks 1$-symmetry (see next
definition), it cannot be a color process. 

\begin{df}\label{df.01symm}
A measure $\mu$ on $\{0,1\}^n$ is said to be $0\ks 1$-symmetric if for any $\xi\in \{0,1\}^n$,
we have $\mu(\xi)=\mu(\hat{\xi})$ where we define $\hat{\xi}$ by letting 
$\hat{\xi}(i)=1-\xi(i)$ for all $i\in [n]$.
\end{df}

The following result characterizes color processes for $n=3$ in the special case $p=1/2$.

\begin{prop}\label{l.symm}
Let $\mu$ be a probability measure on $\{0,1\}^3$. Then
$\mu\in\cp_{[3],1/2}$ if and only if $\mu$ has non-negative pairwise correlations
and is $0\ks 1$-symmetric.
\end{prop}

\pf The "only if" direction is immediate. For the other direction,
let $p_1=\mu(1,1,1)=\mu(0,0,0), p_2=\mu(1,1,0)=\mu(0,0,1), p_3=\mu(1,0,1)=\mu(0,1,0)$ and
$p_4=\mu(0,1,1)=\mu(1,0,0)$ where clearly $\sum_i p_i =1/2$. 
Let $q_1=\nu(\{1,2,3\})$, $q_2=\nu(\{\{1,2\},\{3\}\})$, 
$q_3=\nu\{\{\{1,3\},\{2\}\}\}$, $q_4=\nu(\{\{1\},\{2,3\}\})$ and 
$q_5=\nu(\{\{1\},\{2\},\{3\}\})$. Without loss of generality, we may assume 
that $p_2\le \min\{p_3,p_4\}$. We then take
$q_1:= 2(p_1+p_2-p_3-p_4), q_2:=0, q_3:= 4p_3-4p_2, q_4:= 4p_4-4p_2$ and $q_5:= 8p_2$.
One can immediately check that $\sum_i q_i =1$ with no assumptions. The key point is to
show that $q_i\in [0,1]$ for each $i$. After this, it is easy to check that this $\nu$ works
and this is left to the reader.

To establish $q_i\in [0,1]$ for each $i$, we will of course use the 
non-negative pairwise correlations assumption. The latter assumption easily yields
$p_1+p_2\ge 1/4$, $p_1+p_3\ge 1/4$ and $p_1+p_4\ge 1/4$. Recall also 
$\sum_i p_i =1/2$ and $p_2\le \min\{p_3,p_4\}$. These are all that will be used. 

If $p_2=1/8 + \epsilon$ for some $\epsilon>0$, then
$\sum_i p_i =1/2$ and $p_2\le \min\{p_3,p_4\}$ imply that 
$p_1\le 1/8 -3 \epsilon$, contradicting $p_1+p_2\ge 1/4$. Hence
$p_2\le 1/8$ and so $q_5\in [0,1]$. Next, $q_1\ge 0$ since
$p_1+p_2\ge 1/4$ and $\sum_i p_i =1/2$. The latter also gives that $q_1\le 1$.
Next $p_2\le \min\{p_3,p_4\}$ yields $q_3\ge 0$.
If $p_3=1/4 + \epsilon$ for some $\epsilon>0$, then 
$\sum_i p_i =1/2$ yields that $p_1+p_2< 1/4$, contradicting one of our inequalities.
Therefore $p_3\le 1/4$ implying $q_3\le 1$. Lastly, $q_4$ is handled exactly as $q_3$.\qed

Unfortunately, we don't have any nice characterization of
$\cp_{[3],p}$ for $p\neq 1/2$ since we don't have a good replacement for the $0\ks 1$-symmetry
in this case. The next result shows that Proposition \ref{l.symm} has no extension to larger $n$,
even if exchangeability is assumed.

\begin{prop}\label{l.nonnegcorr}
For each $n\ge 4$, there is a measure $\mu$ on $\{0,1\}^{[n]}$ which is 
exchangeable, $0\ks 1$-symmetric and has non-negative pairwise correlations 
but for which $\mu\notin\cp_{[n],1/2}$.
\end{prop}

\pf Consider the measure $\mu$ on $\{0,1\}^{[n]}$ which is uniform on all points belonging to
levels 1 or $n-1$ where level $i$ refers to those elements which have $i$ 1's. 
Exchangeability and $0\ks 1$-symmetry are obvious. Next, 
we have $$E_{\mu}[X(1)X(2)]=\frac{1}{2}\times \frac{(n-2)}{n} =\frac{1}{2}-\frac{1}{n}$$
so that 
$$\mbox{Cov}_{\mu}(X(1),X(2))=\frac{1}{2}-\frac{1}{n} -\frac{1}{4}=\frac{1}{4}-\frac{1}{n},$$
which is non-negative if and only if $n\ge 4$. Finally,
since $\mu$ assigns measure $0$ to the configuration $\{1,\ldots,1\}$, $\mu\notin \cp_{[n],1/2}$.\qed

We recall the following two definitions.
\begin{df}
A probability measure on $\{0,1\}^{[n]}$ is called {\it positively associated} 
if any two increasing events are positively correlated. 
\end{df}

\begin{df}\label{d.FKGL}
A probability measure on $\{0,1\}^{[n]}$ is said to satisfy the {\it FKG lattice condition}, if, whenever
all but two of the variables are conditioned on, then the remaining two variables are
(conditionally) positively correlated. 
\end{df}

The famous FKG Theorem (see \cite{FKG}) says that if a measure on $\{0,1\}^{[n]}$ has full support and 
satisfies the FKG lattice condition, then, whenever some of the variables are conditioned on, then the 
(conditional) distribution of the remaining variables is positively associated (and so, in 
particular, the measure itself is positively associated).

One can show that the example right before Definition~\ref{df.01symm} satisfies the FKG lattice condition.
This shows that exchangeability and the FKG lattice condition do not necessarily lead to being a 
color process. Interestingly, although color processes of course always have non-negative pairwise correlations, 
they are not necessarily positively associated as shown by the following simple example.

\begin{example}\label{e.posass4}
Define $\nu\in \rer_{[4]}$ to be $\{\{1,2\},\{3\},\{4\}\}$ with probability $1/2$ and 
$\{\{1\},\{2\},\{3,4\}\}$ with probability $1/2$. 
Let $A$ be the event that $X^{\nu,1/2}(1)=X^{\nu,1/2}(2)=1$ and $B$ the event that $X^{\nu,1/2}(3)=X^{\nu,1/2}(4)=1$. 
Then ${\bf P}(A)={\bf P}(B)=3/8$ but ${\bf P}(A\cap B)=1/8<9/64={\bf P}(A){\bf P}(B)$.
\end{example}

While we have not bothered to check, we suspect that all color processes for $n=3$ are in
fact positively associated; this is certainly true for $n=2$. There are results concerning
positive association for color processes associated to the RER corresponding (using the percolation clusters)
to the FK model given in Definition~\ref{df.RC}. Positive association was proved, in chronological order,
(1) for $q\ge 1$ and $p\in [1/q,1-1/q]$ in \cite{OHfuzzy},  (2) for $q=1$ and $p\in [0,1]$ in \cite{OH01} and
(3) for $q\ge 1$ and $p\in [0,1]$ in \cite{KW}. Interestingly, in this last mentioned paper, the authors 
conjecture that this is true for all $q>0$ and bring up the question of positively association in the 
general setup of divide and color models that we study in this paper.

\section{Color processes associated to infinite exchangeable random partitions}\label{s.exchangeable}

In this section, we restrict ourselves to color processes arising from so-called infinite
exchangeable random partitions. In Subsection~\ref{s.definetti}, we recall the notions of simplices,
infinite exchangeable processes and infinite exchangeable random partitions as well as the 
central de Finetti's and Kingman's Theorems concerning such objects. In Subsection~\ref{s.ecp}, we 
develop some general results which apply for all values of $p$. It turns out that the map 
$\Phi_p$ seems to have very different properties depending on whether $p=1/2$ or $p\neq 1/2$, being 
``much more injective'' in the latter case. (Recall, analogously, that Theorem~\ref{t.bigfinitetheorem}(A) 
and (C) (or (B) and (D)) in Section~\ref{s.finitecase} tells us that for $n=3$, we have 
injectivity in the $p\neq 1/2$ case and non-injectivity in the $p= 1/2$ case.)
In Subsection~\ref{s.symmcase}, we restrict to the $p= 1/2$ case, 
characterizing the set of color processes as those which exhibit $0\ks 1$-symmetry (Theorem~\ref{t.mainp12})
and characterizing ``where $\Phi_{1/2}$ is injective'', i.e., which $\nu\in \rer_{\N}^{\exch}$ are
$\rer_{\N}^{\exch}$-unique (Theorem~\ref{p.unprop}).
In Subsection~\ref{s.NONsymmcase}, we restrict to the $p\neq 1/2$ case, obtaining some results
which might suggest that $\Phi_{p}$ is injective in this case. In Subsection~\ref{s.gaussian},
we look at threshold Gaussian and stable processes.

\subsection{Background: Simplices and de Finetti's and Kingman's Theorems}\label{s.definetti}

We first recall Choquet's Theorem (see \cite{glasner}, p. 367).

\begin{thm}\label{t.choquet} 
If $Q$ is a metrizable compact convex subset of a locally convex topological vector space, then
for each $x\in Q$, there is a probability measure $\mu$ on the extremal elements $\emph{ext}(Q)$ of $Q$
for which $x$ is the barycenter (average) of $\mu$ in the sense that for all continuous affine functions
$f$ on $Q$,
$$
f(x)=\int_{\emph{ext}(Q)} f d\mu.
$$
\end{thm}

\begin{df}
If $Q$ is a metrizable compact convex subset of a locally convex topological vector space, then
$Q$ is a {\it simplex} if for all $x\in Q$, the representing $\mu$ in Choquet's Theorem is unique.
\end{df}

The following example is illustrative and will appear soon. Let $C_3$ be the set of probability measures
on $[0,1]$ in the weak$^*$ topology, $C_2$ be the subset consisting of probability measures with mean $1/2$
and $C_1$ the further subset consisting of probability measures which are symmetric about $1/2$.
Clearly $C_1\subseteq C_2\subseteq C_3$ and each $C_i$ is a metrizable compact convex set in this topology
for which Choquet's Theorem is applicable. Interesting, while $C_1$ and $C_3$ are simplices,
$C_2$ is not, as can be checked. The extremal elements of $C_3$ are the point masses while 
the extremal elements of $C_1$ are measures of the form $\frac{\delta_{1/2+a}+\delta_{1/2-a}}{2}$.

Next, let ${\rm Perm}_{\N}$ denote the space of permutations on $\N$ which fix all but finitely many elements. 

\begin{df}
A stochastic process $(X(i))_{i\in \N}$ is said to be exchangeable if for any 
$\sigma \in {\rm Perm}_{\N}$,  $(X(\sigma(i)))_{i\in \N}$ and $(X(i))_{i\in \N}$ are equal in distribution. 
\end{df}

The following is de Finetti's Theorem (see \cite{DURRETT}, p.228).

\begin{thm}\label{t.definetti}
Given a real-valued exchangeable process $X$, there is a unique random distribution $\Xi$ on ${\mathbb R}$ 
such that $X$ is obtained by first choosing $\Xi$ and then letting $X$ be i.i.d.\ 
with distribution $\Xi$.  It follows that this set of exchangeable processes is a simplex whose extremal
elements are product measures.
\end{thm}

In this paper, we mainly consider processes which are $\{0,1\}$-valued.

\begin{df}
Let ${\rm EP}_{\N}$ denote the space of exchangeable processes on $\N$ taking values in $\{0,1\}^\N$.  
For $p\in [0,1]$, let ${\rm EP}_{\N,p}$ denote the space of elements in ${\rm EP}_{\N}$ whose
marginal distribution has mean $p$.
\end{df}

Mostly, we will refer to the elements of  ${\rm EP}_{\N,p}$ as probability measures, but 
sometimes as processes. If $\nu\in {\rm EP}_{\N}$, then de Finetti's Theorem says that 
there exists a unique probability measure $\rho_\nu$ on $[0,1]$ such that

\begin{equation}\label{e.nuxi}
\nu=\int_{s=0}^1 \Pi_s \, d\rho_{\nu}(s),
\end{equation}
where $\Pi_s$ denotes product measure on $\{0,1\}^{\N}$ with density $s$. 
In this case, $\Xi$ is concentrated on $\{0,1\}$ and hence is parameterized by $[0,1]$.
We therefore have a bijection between ${\rm EP}_{\N}$ and probability measures on $[0,1]$. 
In what follows, we will denote by $\xi_{\nu}$ a random variable with law $\rho_{\nu}$. 
Similarly, given a random variable $\xi$ on $[0,1]$ we will by $\nu_{\xi}$ denote the 
exchangeable process obtained by~\eqref{e.nuxi} where $\rho_{\nu}$ is taken to be the law of 
$\xi$; i.e., $\xi$ has distribution $\rho_{{\nu}_{\xi}}$.

Given a real-valued exchangeable process $X$ and $h\in {\mathbb R}$, we let $Y^h=(Y^h(i))_{i=0}^{\infty}$ 
be the ``$h$-threshold process obtained from $X$'' defined by $Y^h(i)=1\{X(i)\ge h\}$. 
Clearly $Y^h\in {\rm EP}_{\N}$ and it is of interest to determine if $Y^h$ is a color process. 
In Section~\ref{s.gaussian}, we will see that this is the case for the $0$-threshold Gaussian and stable
processes.

Next, we find the probability measure $\rho_{Y^h}$ corresponding to $Y^h$. 
Recall the definition of $\Xi$ used in the representation of $X$ above. Observe 
that for any $k\ge 1$, any sequence of integers $0\le n_1<\ldots <n_k$ and any choices of 
$i_{n_1},\ldots,i_{n_k}\in \{0,1\}$ we have 
\begin{equation}\label{e.projectexch}
P(Y^h(n_1)=i_{n_1},\ldots,Y^h(n_k)=i_{n_k})=E\left[\Xi([h,\infty))^{\sum_{j=1}^k i_{n_j}}(1-\Xi([h,\infty)))^{k-\sum_{j=1}^k i_{n_j}}\right].
\end{equation}
From~\eqref{e.projectexch} it follows that $\rho_{Y^h}$ is the law of $\Xi([h,\infty))$, or 
equivalently, $\xi_{Y^h}=\Xi([h,\infty))$.

\medskip

For $\sigma\in {\rm Perm}_{\N}$ and $\pi \in {\rm Part}_{\N}$ define $\sigma \pi \in {\rm Part}_{\N}$ 
by letting $\sigma\pi(x)=\sigma \pi(y)$ if and only if $\pi(\sigma^{-1}(x))=\pi(\sigma^{-1}(y))$.  
The ``$-1$'' is present to ensure that we have a "group action". For $\nu\in \rer_{\N}$ and 
$\sigma\in {\rm Perm}_{\N}$, let $\sigma\circ \nu \in \rer_{\N}$ be defined as 
$\sigma\circ \nu (\cdot)=\nu(\sigma^{-1}(\cdot))$. 

\begin{df}
We say that $\nu\in \rer_{\N}$ is \emph{exchangeable} if for any $\sigma \in {\rm Perm}_{\N}$ we have 
$\sigma \circ \nu=\nu$. The space of exchangeable RERs on $\N$ will be denoted by $\rer_{\N}^{\exch}$.
\end{df}

Of course, $\N$ can be replaced by any countable set here since there is no "geometric structure" since
we are considering all permutations but we use $\N$ for simplicity. 

The following is the first step in introducing our collection of exchangeable RERs.

\begin{df}
We say that ${\bf p}=(p_1,p_2,\ldots)$ is a \emph{paint-box} if $p_i\ge 0$ for all $i$, 
$p_i\ge p_{i+1}$ for all $i$, and $\sum_i p_i\le 1$.
\end{df}

Given a paint-box ${\bf p}=(p_1,p_2,\ldots)$, we obtain an element of
$\rer_{\N}^{\exch}$ as follows. Define the random equivalence classes $(S_i)_{i\ge 1}$ by 
putting each element of ${\mathbb N}$ independently in $S_i$ with probability $p_i$ and with 
probability $1-\sum_i p_i$ put it in its own equivalence class. We denote this {\rm RER} by 
$\nu_{{\bf p}}$. It follows easily that $\nu_{{\bf p}}\in \rer_{\N}^{{\rm \exch}}$. 

\begin{remark}
We use slightly different terminology for paint-boxes than what is used in~\cite{JB06}, where
it is the RER $\nu_{\bf p}$, rather than the vector ${\bf p}$, which is called a paint-box.
\end{remark}

\begin{df}
The subset of $\rer^{\exch}_{\N}$ which consists of RERs obtained from paint-boxes will be denoted 
by $\rer^{\exch,\pure}_{\N}$.
\end{df}

We can obtain more elements in $\rer_{\N}^{\exch}$ by taking convex combinations
and in fact generalized convex combinations of the elements in $\rer_{\N}^{\exch,\pure}$.
It is immediate that all of these are in $\rer_{\N}^{\exch}$. Kingman's famous theorem 
(Theorem~\ref{t.kingman} below, see also~\cite{JB06})
says that these account for all of the elements of $\rer_{\N}^{\exch}$.
Moreover, the uniqueness in this theorem tells us that $\rer_{\N}^{\exch}$ is a
simplex whose extremal elements are $\rer_{\N}^{\exch,\pure}$.

\begin{thm}\label{t.kingman}
{\bf (Kingman) }Suppose that $\nu\in {\rm RER}_{\N}^{\exch}$. Then there is a unique probability measure 
$\rho=\rho_{\nu}$ on ${\rm RER}^{\exch,\pure}_{\N}$ such that 
$$
\nu=\int_{\nu_{{\bf p}}\in {\rm RER}^{\exch,\pure}_{\N}}\nu_{\bf p}\,d\rho(\nu_{\bf p}).
$$
\end{thm}

\subsection{Infinite exchangeable color processes}\label{s.ecp}

Our first result says that ${\rm CP}_{\N,p}^{\exch}$ (which recall was defined to be the image of
$\rer_{\N}^{\exch}$ under $\Phi_p$) is simply ${\rm EP}_{\N,p}\cap {\rm CP}_{\N,p}$.

\begin{prop}\label{p.imageprop}
For any $p\in[0,1]$,
$$
{\rm CP}_{\N,p}^{\exch}={\rm EP}_{\N,p}\cap {\rm CP}_{\N,p}.
$$
\end{prop}

\pf  The containment $\subseteq$ is clear. Assume that $\mu\in {\rm EP}_{\N,p}\cap {\rm CP}_{\N,p}$. 
Then there is some $\nu\in \rer_{\N}$ 
such that $\Phi_p(\nu)=\mu$. We will be done if we find some $\nu'\in \rer_{\N}^{\exch}$ such that 
$\Phi(\nu')=\mu$. We will construct such a $\nu'$ from $\nu$. Let ${\rm Perm}_{[n]}$ denote the set 
of permutations on $[n]$ and let 
$$
\nu_n=\frac{1}{|{\rm Perm}_{[n]}|}\sum_{\sigma\in {\rm Perm}_{[n]}}\sigma\circ \nu,
$$
where it is understood that a $\sigma\in {\rm Perm}_{[n]}$ is viewed as an element of
${\rm Perm}_{\N}$ which fixes all $k$ larger than $n$.
Since $\mu\in {\rm EP}_{\N,p}$ and $\Phi_p$ commutes with permutations it follows that 
$\Phi_{p}(\sigma\circ \nu)=\sigma\circ\Phi_p(\nu)=\sigma\circ \mu =\mu$ for any 
$\sigma\in {\rm Perm}_{[n]}$. In particular, $\Phi_p(\nu_n)=\mu$ for all $n$.
Clearly  $\nu_n$ is invariant under permutations of $[n]$ (meaning that $\sigma \circ \nu_n=\nu_{n}$ 
for any $\sigma\in {\rm Perm}_{[n]}$), so that in particular the restriction of $\nu_n$ to $[n]$ 
belongs to $\rer_{[n]}^{\exch}$.  By compactness, we can choose some subsequence $n_k$ so that 
$\nu_{n_k}$ converges to some $\nu_{\infty}$ as $k\to \infty$. It is clear that 
$\nu_{\infty}\in \rer_{\N}^{\exch}$ and $\Phi_p(\nu_{\infty})=\mu$ follows from the easily shown 
fact that $\Phi_p(\cdot)$ is continuous.\qed

We now show that the mixing random variable $\xi$ for the color process corresponding to a paintbox is a 
so-called Bernoulli convolution.

\begin{lma}\label{l.xipaint}
Fix $p\in [0,1]$ and a paintbox ${\bf p}=(p_1,p_2,\ldots)$.
For the associated color process, let $\xi_{{\bf p},p}$ be the representing random variable in $[0,1]$
in de Finetti's Theorem. Then, in distribution,
\begin{equation}\label{e.xiequiv}
\xibfp=(1-\sum_{i\ge 1} p_i) p+\frac{1}{2} \sum_{i\ge 1} p_i +\frac{1}{2} \sum_{i\ge 1} p_i Z_i,
\end{equation}
where the $Z_i$ are i.i.d.\ random variables with $P(Z_i=1)=p$ and $P(Z_i=-1)=1-p$. 
If $p=1/2$,~\eqref{e.xiequiv} simplifies to

\begin{equation}\label{e.xiequi2}
\xihalf=\frac{1}{2}+\frac{1}{2}\sum_{i\ge 1} p_i Z_i.
\end{equation}

\end{lma}

\pf Let $p\in [0,1]$ and consider the paintbox ${\bf p}=(p_1,p_2,\ldots)$. 
Define a random subset $S$ of ${\N}$ by independently putting each $n\in \N$ in $S$ with probability 
$p$ and in $S^c$ with probability $1-p$.  Letting 
\begin{equation}\label{e.xidef}
\xi_{{\bf p},p}:=\sum_{i\ge 1} I\{i\in S\} p_i + (1-\sum_{i\ge 1} p_i) p,
\end{equation}

and $F_{\xi_{{\bf p},p}}$ be the law of $\xi_{{\bf p},p}$, it is straightforward to see that

$$\Phi_p(\nu_{\bf p})=\int_{s=0}^1 \Pi_s d F_{\xi_{{\bf p},p}}(s)\mbox{ }(=\nu_{\xi_{{\bf p},p}}).$$
Finally, one verifies that~\eqref{e.xidef} can be rewritten as~\eqref{e.xiequiv}.
\qed

\medskip

As an application of Lemma~\ref{l.xipaint} we get the identities

\begin{equation}\label{e.simplebox}
\Phi_p(\nu_{(p_1,0,\ldots)})=p \Pi_{p_1+(1-p_1)p} +(1-p) \Pi_{(1-p_1) p},
\end{equation}

\noindent and

\begin{eqnarray}\label{e.pure2box}
\lefteqn{\Phi_p(\nu_{(p_1,p_2,0,\ldots)})=p^2 \Pi_{p_1+p_2+(1-p_1-p_2)p} +(1-p) p\Pi_{p_1+(1-p_1-p_2) p}}\\  & & +(1-p) p\Pi_{p_2+(1-p_1-p_2) p}+(1-p)^2\Pi_{(1-p_1-p_2) p},\nonumber
\end{eqnarray}

\noindent which in the case $p=1/2$ simplify to 

\begin{equation}\label{e.simplebox05}
\Phi_{1/2}(\nu_{(p_1,0,\ldots)})=\frac{1}{2}( \Pi_{1/2+p_1/2} +\Pi_{1/2-p_1/2}),
\end{equation}

\noindent and

\begin{eqnarray}\label{e.pure2box05}
\lefteqn{\Phi_{1/2}(\nu_{(p_1,p_2,0,\ldots)})}\nonumber\\& & =\frac{1}{4} (\Pi_{1/2+(p_1+p_2)/2}+\Pi_{1/2+(p_1-p_2)/2}+\Pi_{1/2-(p_1-p_2)/2}+\Pi_{1/2-(p_1+p_2)/2}).
\end{eqnarray}

\noindent From~\eqref{e.simplebox05} and~\eqref{e.pure2box05} we obtain

\begin{equation}\label{e.purenonuniq05}
\Phi_{1/2}(\nu_{(p_1,p_2,0,\ldots)})=\frac{1}{2}\Phi_{1/2}(\nu_{(q_1,0,\ldots)})+\frac{1}{2} \Phi_{1/2}(\nu_{(q_2,0,\ldots)}),
\end{equation}

\noindent where $q_1=p_1+p_2$ and $q_2=p_1-p_2$. Note that this implies that
$\Phi_{1/2}\,:\,\rer_{\N}^{\exch}\to \cp_{\N,{1/2}}^{\exch}$ is not injective.

\bigskip\noindent
On the other hand, we have the following proposition, where
the key part of the proof was provided to us by Russell Lyons.

\begin{prop}\label{p.Russ}
The map 
$$
\Phi_p\,:\,\rer_{\N}^{\exch,\pure}\to \cp_{\N,p}^{\exch}
$$
is injective for every $p\in (0,1)$.
\end{prop}

\pf Fix $p\in (0,1)$ and consider two different paintboxes ${\bf p}$ and ${\bf p}'$. In view of 
Lemma~\ref{l.xipaint} and the uniqueness in de Finetti's Theorem, we need to show that 
\begin{equation}
(1-\sum_{i\ge 1} p_i) p+\frac{1}{2} \sum_{i\ge 1} p_i +\frac{1}{2} \sum_{i\ge 1} p_i Z_i,
\end{equation}
and
\begin{equation}
(1-\sum_{i\ge 1} p_i') p+\frac{1}{2} \sum_{i\ge 1} p_i' +\frac{1}{2} \sum_{i\ge 1} p_i' Z_i,
\end{equation}
have different distributions where, as before,
the $Z_i$ are i.i.d.\ random variables with $P(Z_i=1)=p$ and $P(Z_i=-1)=1-p$. 
The length of the smallest intervals containing the supports of these distributions are
$\sum_{i\ge 1} p_i$ and $\sum_{i\ge 1} p_i'$ and hence if these differ, then the 
distributions are different. Assume now that $\sum_{i\ge 1} p_i=\sum_{i\ge 1} p_i'$. In this case,
if the distributions were the same, we would also have that the distributions of
$\sum_{i\ge 1} p_i Z_i$ and $\sum_{i\ge 1} p_i' Z_i$ were the same. 

We will now be done if we prove that the Fourier transform
$$
f(z):={\mathbf E}[{\rm e}^{z \sum_{i=1}^{\infty} p_i Z_i}],\mbox{ }z\in {\mathbb C}
$$
determines the paintbox ${\bf p}$. We do this in the case $p_i>0$ for all $i\ge 1$. 
The argument is easily modified to the case $p_i=0$ for all $i$ sufficiently large. 
By independence,
$$ 
f(z)=\prod_{j=1}^{\infty}{\mathbf E}[{\rm e}^{z \,p_j Z_j}]=\prod_{j=1}^{\infty}(p\,{\rm e}^{z\, p_j}+(1-p)\,{\rm e}^{-z\,p_j}).
$$

For $j\ge 1$, let $\Delta_j=\{z\in {\mathbb C}\,:\,{\mathbf E}[{\rm e}^{z \,p_j Z_j}]=0\}$. Then
\begin{equation}\label{e.deltadef}
\Delta_j=\left\{\frac{1}{p_j} \left( \frac{\log{\left(\frac{1-p}{p}\right)}}{2} +i (\pi k+\pi/2)\right)\,:\,k\in {\mathbb Z}\right\}.
\end{equation}
Since $\sum_{j\ge 1} p_j\le 1$, we have $g_n(z)=0$ only if $z\in \Delta_j$ for some $j$.
Let $g_1(z)=f(z)$ and for $n\ge 2$ let 
$$
g_n(z)=\prod_{j=n}^{\infty} {\mathbf E}[{\rm e}^{z \,p_j Z_j}].
$$

For $n\ge 1$, let
$$
t_n=\inf\{|{\rm Im}(z)|\,:\,g_n(z)=0\}.
$$
Hence, according to~\eqref{e.deltadef}, $t_n=\pi/( 2 p_n)$. Hence, we can recover the sequence 
$(p_n)_{n\ge 1}$ from the sequence $(t_n)_{n\ge 1}$ and the result follows. \qed

\subsection{The case $p=1/2$.}\label{s.symmcase}

In this subsection, we obtain some results concerning $\Phi_{1/2}$ on $\rer_{\N}^{\exch}$. 
First observe that if $\mu \in \cp_{\N,1/2}^{\exch}$, then $\mu$ is $0\ks 1$-symmetric
and hence so is the representing random variable $\xi_\mu$; i.e.\ $\xi_\mu=1-\xi_\mu$ in law.
Interestingly, as we will see below in Theorem~\ref{t.mainp12}, this necessary condition of symmetry 
is actually a sufficient condition for being a color process when $p=1/2$.
In Theorem~\ref{p.unprop} we determine exactly which are the exchangeable 
RERs that are $(\rer_{\N}^{\exch},1/2)$-unique.

In the proofs below, we will make use of the following lemma which follows easily from de 
Finetti's theorem.

\begin{lma}\label{l.exchextremal}
Let ${\rm EP}_{\N,1/2}^{\symm}$ be the set of exchangeable processes which are $0\ks 1$-symmetric 
(or equivalently their representing distribution in $[0,1]$ is symmetric about $1/2$) and 
for $\alpha\in [0,1/2]$, let $\mu_{\alpha}:=(\Pi_{1/2+\alpha}+\Pi_{1/2-\alpha})/2$. Then 
${\rm EP}_{\N,1/2}^{\symm}$ is a simplex and
\begin{equation}\label{e.extremalset}
\emph{ext}({\rm EP}_{\N,1/2}^{\symm})=(\mu_\alpha)_{\alpha\in [0,1/2]}.
\end{equation}
\end{lma}

The following subset of ${\rm RER}_{\N}^{\exch,\simple}$ will play an important role in our 
discussions below. 

\begin{df}
The subset of $\rer^{\exch,\pure}_{\N}$ which consists of  RERs obtained from paint-boxes with 
$p_2=0$ will be denoted by $\rer^{\exch,\simple}_{\N}$.
\end{df}

Note that, using~\eqref{e.xiequi2}, we have a natural identification between $\rer^{\exch,\pure}_{\N}$,
$\{\mu_{\alpha}\}_{\alpha \in [0,1/2]}$ from Lemma~\ref{l.exchextremal} and $[0,1/2]$ via
$$
(p,0,\ldots)\leftrightarrow \mu_{p/2} \leftrightarrow p/2
$$
with the first bijection also being given by $\Phi_{1/2}$.

\begin{thm}\label{t.mainp12}
The map $\Phi_{1/2}\,:\,\rer_{\N}^{\exch}\to {\rm EP}_{\N,1/2}^{\symm}$ is onto. Moreover,
for every $\mu\in {\rm EP}_{\N,1/2}^{\symm}$ there is a unique probablity measure $\rho_{\mu}$ 
on $\rer_{\N}^{\exch,\simple}$ such that 
\begin{equation}\label{e.mainexch2}
\mu=\Phi_{1/2}\left(\int_{\nu\in \rer_{\N}^{\exch,\simple}}\nu\,d\rho_{\mu}(\nu)\right).
\end{equation}
\noindent
(Hence ${\rm EP}_{\N,1/2}^{\symm}=\cp_{\N,1/2}^{\exch}$ is 
a simplex whose extremal elements is the set $\{\mu_{\alpha}\}_{\alpha \in [0,1]}$.)
On the other hand, the map $\Phi_{1/2}\,:\,{\rm RER}_{\N}^{\exch,\pure}\to {\rm EP}_{\N,1/2}^{\symm}$ 
is not onto. 
\end{thm}

\pf We start with~\eqref{e.mainexch2}. As already observed right before Theorem~\ref{t.mainp12},
if ${\bf p}_{\alpha}= (2\alpha,0,\ldots)$ with $\alpha\in [0,1/2]$, then
\begin{equation}\label{e.degen}
\Phi_{1/2}(\nu_{{\bf p}_\alpha})=\mu_\alpha.
\end{equation}
Hence $\mu_\alpha \in \cp_{\N,1/2}^{\exch}$. Now pick an arbitrary $\mu\in {\rm EP}_{\N,1/2}^{\symm}$. 
By Lemma~\ref{l.exchextremal} there is a unique law $F_\mu$ on $[0,1/2]$ such that 
\begin{equation}\label{e.muident}
\mu=\int_{0}^{1/2} \mu_{\alpha} d F_{\mu}(\alpha).
\end{equation}
It follows from the affine property of $\Phi_{1/2}$ that
\begin{equation}\label{e.convpunch}
\Phi_{1/2}\left(\int_{0}^{1/2}  \nu_{{\bf p}_\alpha} d F_{\mu}(\alpha)\right)=\int_{0}^{1/2} \Phi_{1/2}( \nu_{{\bf p}_\alpha}) d F_{\mu}(\alpha)\stackrel{~\eqref{e.degen}}{=}\int_{0}^{1/2} \mu_\alpha d F_{\mu}(\alpha)\stackrel{~\eqref{e.muident}}{=}\mu,
\end{equation}
and~\eqref{e.mainexch2} follows. The uniqueness of $\rho_\mu$ follows from the comment before Theorem~\ref{t.mainp12}.

Next, we need to prove that there exist elements of ${\rm EP}_{\N,1/2}^{\symm}$ which can not be 
obtained as the image of some element of $\rer_{\N}^{\exch,\pure}$ under $\Phi_{1/2}$. Consider a 
paintbox ${\bf p}=(p_1,p_2,\ldots)$. Recall $\xi_{{\bf p},1/2}$ from~\eqref{e.xiequi2}. Then

\begin{equation}\label{e.phintrepr}
\Phi_{1/2}(\nu_{\bf p})=\int_{s=0}^1 \Pi_s d F_{\xi_{{\bf p},1/2}}(s),
\end{equation}
where $F_{\xi_{{\bf p},1/2}}$ is the law of $\xi_{{\bf p},1/2}$.
From~\eqref{e.xiequi2} and~\eqref{e.phintrepr}, we see that it suffices to find a random variable 
$W$ in $[0,1]$ which is symmetric around $1/2$ which can not be written as
\begin{equation}\label{e.zrepr}
W=\frac{1}{2}+\frac{1}{2}\sum_i p'_i Z_i,
\end{equation}
for any paintbox ${\bf p}'=(p_1',\ldots)$ where the $\{Z_i\}$'s are as in the proof of 
Proposition~\ref{p.Russ}. Take $W$ to be a random variable with $P(W=1)=P(W=0)=3/8$ 
and $P(W=1/2)=1/4.$  Now, if $W$ has the above representation, then we must have $p_i'\neq 0$ 
for $i=1,2$ and $p_i'=0$ for all $i\ge 3$, since $W$ has three possible values. However, we then obtain
\begin{equation}\label{e.contreq}
P(W=\frac{1}{2}+\frac{p_1'+p_2'}{2}) = P(W=\frac{1}{2}+\frac{p_1'-p_2'}{2}) =
\end{equation}
\begin{equation}
P(W=\frac{1}{2}+\frac{p_2'-p_1'}{2}) =P(W=\frac{1}{2}-\frac{p_1'+p_2'}{2})=1/4.
\end{equation}
Since we assumed that $P(W=1/2)>0$, we must have $p_2'-p_1'=0.$ However then according to~\eqref{e.contreq} 
we get $P(W=1/2)=1/2$, which is a contradiction. Hence $W$ does not have the 
representation~\eqref{e.zrepr} and the result follows. \qed

\begin{cor}\label{l.nuq1}
For any $\nu\in \rer_{\N}^{\exch}$ there is a unique probability measure $\rho=\rho_{\nu}$ on 
$[0,1]$ such that 
\begin{equation}\label{e.combsimple1}
\Phi_{1/2}(\nu)=\Phi_{1/2}\left(\int_{0}^1 \nu_{(p,0,\ldots)}\,d\rho(p)\right).
\end{equation}
\end{cor}

\pf We have that $\Phi_{1/2}(\nu)\in {\rm EP}_{\N,1/2}^{\symm}$. Now~\eqref{e.combsimple1} follows 
immediately from~\eqref{e.mainexch2} and the comment preceding Theorem~\ref{t.mainp12}. \qed

We have seen in the previous subsection that $\Phi_{1/2}$ is not injective. The following 
characterizes exactly the subset of $\rer_{\N}^{\exch}$ on which $\Phi_{1/2}$ is injective.

\begin{thm}\label{p.unprop}
If $\nu\in \rer_{\N}^{\exch}$, then $\nu$ is  $(\rer_{\N}^{\exch},1/2)$-unique if and only if 
$\nu\in \rer_{\N}^{\exch,\simple}$.
\end{thm}

\pf 
If $\nu=\nu_{(p,0,\ldots)}$, then the support of $\xi_\nu$ is 
$\{\frac{1}{2}+\frac{p}{2},\frac{1}{2}-\frac{p}{2}\}$. The 
$\xi$ corresponding to every other $\nu'\in \rer_{\N}^{\exch,\pure}$ has part of its support outside of 
the above set. Hence any $\nu'\in \rer_{\N}^{\exch}$ other than $\nu$ has its
corresponding $\xi$ having part of its support outside of this set. It follows that 
$\nu$ is $(\rer_{\N}^{\exch},1/2)$-unique.

For the other direction, fix $\nu\in \rer_{\N}^{\exch}\setminus \rer_{\N}^{\exch,\simple}$.  By 
Corollary~\ref{l.nuq1} and the fact that $\nu$ is not simple, it suffices to consider the case when we can write
\begin{equation}\label{e.notpointmass}
\nu=\int_{p=0}^1 \nu_{(p,0,\ldots)}d\psi(p),
\end{equation}
for some probability measure $\psi$ on $[0,1]$ where $\psi\neq \delta_t$ for any $t\in [0,1]$. Then we can 
find constants $a_1,a_2,b_1,b_2$ such that $0\le a_1<a_2<b_1<b_2\le 1$, $\psi([a_1,a_2])>0$ and 
$\psi([b_1,b_2])>0$. Let $I=[a_1,a_2]$ and $J=[b_1,b_2]$ and $K=[0,1]\setminus (I\cup J)$. For any 
$T\subset [0,1]$ such that $\psi(T)>0$ let $\tilde{\psi}_T:=\psi_T/\psi(T)$ where $\psi_T$ stands for 
the restriction of $\psi$ to $T$.

Without loss of generality, assume that $\psi(J)\ge \psi(I)$. Observe that 

\begin{eqnarray*}\label{e.psieq}
\lefteqn{\psi=\psi(I)\tilde{\psi}_I+\psi(J)\tilde{\psi}_J+\psi(K)\tilde{\psi}_K}\\ & & =\psi(K)\tilde{\psi}_K +(\psi(J)-\psi(I))\tilde{\psi}_J+2\psi(I)(\tilde{\psi}_I/2+\tilde{\psi}_J/2).
\end{eqnarray*}

Hence,

\begin{eqnarray*}
\lefteqn{\nu=\psi(K)\int_{p\in K}\nu_{(p,0,\ldots)}d\tilde{\psi}_K(p)+(\psi(J)-\psi(I))\int_{p\in J}\nu_{(p,0,\ldots)}d\tilde{\psi}_J(p)}\\ & & +2\psi(I)\left(\frac{1}{2}\int_{p\in I}\nu_{(p,0,\ldots)}d\tilde{\psi}_I(p)+\frac{1}{2}\int_{p\in J}\nu_{(p,0,\ldots)}d\tilde{\psi}_J(p)\right).
\end{eqnarray*}

We now focus on the last term in the sum above. Let 

$$
\rho=\frac{1}{2}\int_{p\in I}\nu_{(p,0,\ldots)}d\tilde{\psi}_I(p)+\frac{1}{2}\int_{p\in J}\nu_{(p,0,\ldots)}d\tilde{\psi}_J(p),
$$ 
and observe that $\rho\in \rer_{\N}^{\exch}$ since $\tilde{\psi}_I$ is a probability measure on $I$ and 
$\tilde{\psi}_J$ is a probability measure on $J$. Since $\Phi_{1/2}$ is affine and $\psi(I)>0$, we will be done 
if we can find $\rho'\in \rer_{\N}^{\exch}$ such that $\rho'\neq \rho$ but $\Phi_{1/2}(\rho)=\Phi_{1/2}(\rho')$. 
We let
$$
\rho'=\int_{p_1\in J}\int_{p_2\in I} \nu_{\left((p_1+p_2)/2,(p_1-p_2)/2,0,\ldots\right)} d\tilde{\psi}_I(p_2)d\tilde{\psi}_J(p_1),
$$
where we recall that $p_1>p_2$ for $p_1\in J$ and $p_2\in I$. Clearly, $\rho'\in \rer_{\N}^{\exch}$. 
Moreover, $\rho'\neq \rho$ since $\rho'$ assigns measure $1$ to those 
$\nu_{(q_1,q_2,\ldots)}\in \rer_{\N}^{\exch,\pure}$ which have $q_2\neq 0$. Since $\Phi_{1/2}$ is affine, we get
\begin{eqnarray}\label{e.affinfubini}
\lefteqn{\Phi_{1/2}(\rho')=\int_{p_1\in J}\int_{p_2\in I}\Phi_{1/2}( \nu_{\left((p_1+p_2)/2,(p_1-p_2)/2,0,\ldots\right)} )d\tilde{\psi}_I(p_2)d\tilde{\psi}_J(p_1)}\nonumber \\ & & \stackrel{~\eqref{e.purenonuniq05}}{=} \frac{1}{2}\int_{p_1\in J}\int_{p_2\in I}\Phi_{1/2}( \nu_{\left(p_1,0,\ldots\right)} )+\Phi_{1/2}( \nu_{\left(p_2,0,\ldots\right)} ) d\tilde{\psi}_I(p_2)d\tilde{\psi}_J(p_1)\nonumber\\ & & =\frac{1}{2}\int_{p_1\in J}\Phi_{1/2}( \nu_{\left(p_1,0,\ldots\right)} )d\tilde{\psi}_J(p_1)+\frac{1}{2}\int_{p_2\in I}\Phi_{1/2}( \nu_{\left(p_2,0,\ldots\right)} )d\tilde{\psi}_I(p_2) \\ & & =\Phi_{1/2}(\rho). \nonumber
\end{eqnarray}
\qed

\subsection{The case $p\neq 1/2$}\label{s.NONsymmcase}

If $p=1/2$, we have seen in the previous subsection that the map 
$\Phi_p \,:\,\rer_{\N}^{\exch}\to {\rm CP}_{\N,p}^{\exch}$
is ``highly non-injective''. In this subsection, we present evidence that, for $p\neq 0,1/2,1$, $\Phi_p$ might be
injective, although we do not manage to prove such a result.

We first introduce some notation. Let $S_0=\{\nu_{(0,\ldots)}\}$ and for $k\ge 1$, define 
$$
S_k:=\{\nu_{\bf p}\in \rer_{N}^{\exch,\pure}\,:\,{\bf p}=(p_1,\ldots,p_k,0,\ldots)\,\mbox{ with }\,p_k>0\},
$$ 
and 
$$
S_{\infty}:=\{\nu_{\bf p}\in \rer_{N}^{\exch,\pure}\,:\,{\bf p}=(p_1,\ldots)\,\mbox{ with }\,p_i>0\,\,\,\forall i\}.
$$ 
Then the $S_k$'s are disjoint and $\rer_{\N}^{\exch,\pure}=\cup_{0\le k\le\infty} S_k$.

The following result from \cite{MR2538010} (see Theorem 1.3 there) tells us what needs to be verified in 
order to conclude that $\Phi_p$ is injective.

\begin{thm}\label{t.czech}
If $\phi$ is a continuous affine map from a compact convex set $X$ to a simplex $Y$
such that $\phi(\emph{ext}(X))\subseteq \emph{ext}(Y)$ and $\phi$ is injective on
$\emph{ext}(X)$, then $\phi$ is injective.
\end{thm}

It is not so difficult to show (and left to the reader) that if $x\in \rm{ext}(X)$ is $\phi$-unique 
(meaning $\phi(x)\neq \phi(y)$ for all $y\neq x$), then $\phi(x)\in \rm{ext}(Y)$.
Hence, in our context, to show injectivity using Theorem~\ref{t.czech}, one needs, in addition to
Proposition~\ref{p.Russ}, to show that, for $p\neq 1/2$, (1) 
$\cp_{\N,p}^{\exch}$ is a simplex and (2) for all $0\le k\le \infty$, all elements of $S_k$ are 
$(\rer_{\N}^{\exch},p)$-unique. We are not able to show (1) (but note we have seen this is true for 
$p= 1/2$) and in the rest of the subsection, we show (2) for $S_0$, $S_1$, $S_2$ and a subset of $S_3$.

Observe first that $\Phi_p(\nu_{(0,\ldots)})=\Pi_p$, so it is easy to see that $\nu_{(0,\ldots)}$ is 
$(\rer_{\N}^{\exch},p)$-unique for every $p\in (0,1)$. The following three propositions cover the cases
$k=1,2$ and part of $k=3$.

\begin{prop}\label{p.s1unique}
Suppose that $\nu\in S_1$. Then $\nu$ is $(\rer_{\N}^{\exch},p)$-unique for every $p\in (0,1)\setminus \{1/2\}$. 
(This is also true for $p=1/2$ by Theorem~\ref{p.unprop}.)
\end{prop}

{\bf Proof.} By symmetry we assume that $p\in (1/2,1)$. Fix $\nu=\nu_{(s,0,\ldots)}\in S_1$.  Suppose that 
$\tilde{\nu}\in \rer_{\N}^{\exch}$ is such that $\Phi_p(\nu)=\Phi_p(\tilde{\nu})$. Recall that by 
Kingman's theorem, there is a unique probability measure $\rho=\rho_{\tilde{\nu}}$ on $\rer_{\N}^{\exch,\pure}$ 
such that 
\begin{equation}\label{e.unprop10}
\tilde{\nu}=\int_{\nu_{\bf p}\in \rer_{\N}^{\exch,\pure}} \nu_{\bf p} d\rho(\nu_{\bf p}).
\end{equation}
Hence, we will be done if we show that $\Phi_p(\nu)=\Phi_p(\tilde{\nu})$ implies that $\rho=\delta_{\nu}$. 
For our fixed $\nu\in S_1$, we have, using~\eqref{e.xiequiv}, that

\begin{equation}\label{e.xi3support2}
\xi_{(s,0,\ldots),p}=\left\{\begin{array}{lll} 
y_1:=p+ s (1-p)&{\rm w.p.}&p\\ y_2:=p-sp &{\rm w.p.}&1-p
\end{array}\right.
\end{equation}

Observe that if $\paintbox\in S_k$, then $|{\rm supp}(\xipp)|\ge k+1$. 
Hence $\rho(\cup_{k\ge 2}S_k)=0$. Using~\eqref{e.xi3support2}, we see that if $\paintbox \in S_1$ then in 
order to have ${\rm supp}(\xi_{{\bf p},p})\subseteq {\rm supp}(\xi_{(s,0,\ldots)})$ we must have 
$\paintbox=\nu$. Hence $\rho(S_1\setminus \{\nu\})=0$. Finally, since $y_2<p<y_1$ for every $p\in (0,1)$, 
it follows that $\rho(S_0)=0$. Hence $\rho=\delta_{\nu}$ as claimed.\qed

\begin{prop}\label{p.uniqueprop1}
Suppose that $\nu\in S_2$. Then $\nu$ is $(\rer_{\N}^{\exch},p)$-unique for every $p\in (0,1)\setminus \{1/2\}$. 
(This is false for $p=1/2$ by Theorem~\ref{p.unprop}.)
\end{prop}

{\bf Proof.} The strategy of this proof will be the same as that of the proof of Proposition~\ref{p.s1unique}, 
but more involved since there will be more cases to deal with. By symmetry we assume that $p\in (1/2,1)$. 
We fix $\nu=\nu_{(p_1,p_2,0,\ldots)}\in S_2$ and suppose that $\tilde{\nu}\in \rer_{\N}^{\exch}$ is such 
that $\Phi_p(\nu)=\Phi_p(\tilde{\nu})$. Let $\rho=\rho_{\tilde{\nu}}$ be the unique probability measure 
on $\rer_{\N}^{\exch,\pure}$ such that 
\begin{equation}\label{e.unprop1}
\tilde{\nu}=\int_{\nu_{\bf p}\in \rer_{\N}^{\exch,\pure}} \nu_{\bf p} d\rho(\nu_{\bf p}).
\end{equation}
We will show that $\Phi_p(\nu)=\Phi_p(\tilde{\nu})$ implies that $\rho=\delta_{\nu}$. 
Again, we recall the random variable $\xipp$ from Lemma~\ref{l.xipaint}, and we will proceed by looking 
at the support of this random variable. For our fixed $\nu\in S_2$, we have, using~\eqref{e.xiequiv}, that
\begin{equation}\label{e.xisupport}
\xi_{(p_1,p_2,0,\ldots),p}=\left\{\begin{array}{lllll}
z_1:=p+(p_1+p_2)(1-p) &\mbox{ w.p. } &p^2\\ z_2:= p+p_1(1-p)-p_2 p &\mbox{ w.p. } &p(1-p)\\ z_3:= p+p_2(1-p)-p_1 p &\mbox{ w.p. } &p(1-p)\\z_4:= p-(p_1+p_2)p &\mbox{ w.p. }&(1-p)^2
\end{array} \right. 
\end{equation}
In~\eqref{e.xisupport}, we have ordered the elements of 
${\rm supp}(\xi_{(p_1,p_2,0,\ldots),p})$ in decreasing order, with the largest element on the first line.

Now, we will look at the elements in $S_0,S_1,\ldots$ in order to find those $\paintbox$ for which $\xipp$ 
has its support contained in  ${\rm supp}(\xi_{(p_1,p_2,0,\ldots),p})$. The measure $\rho$ must be supported 
on such $\xipp$'s. 

{\bf Case 1:} First we look at the single element of $S_0$, namely $\nu_{(0,\ldots)}$. We have 
\begin{equation}\label{e.xi1support}
\xi_{(0,\ldots),p}=p\,\, \mbox{ w.p. }1.
\end{equation} 
Since $p>1/2$ we have $z_1>p$ and $z_3,z_4<p$. Hence we see that if 
${\rm supp}(\xi_{(0,\ldots),p})\subseteq {\rm supp}(\xi_{(p_1,p_2,0,\ldots),p})$, then $p=z_2$ so 
that $p_1/p=p_2/(1-p)$. From this we conclude that
\begin{equation}\label{e.keep3}
\mbox{If }\rho(\nu_{(0,\ldots)})>0,\mbox{ then }p_1/p=p_2/(1-p) \mbox{ and } p=z_2.
\end{equation}

{\bf Case 2: } Assume that $\nu_{(s,0,\ldots)}\in S_1$ and recall that
\begin{equation}\label{e.xi3support}
\xi_{(s,0,\ldots),p}=\left\{\begin{array}{lll} 
y_1:=p+ s (1-p)&\,\, {\rm w.p.}&p\\ y_2:=p-sp &\,\, {\rm w.p.}&1-p
\end{array}\right.
\end{equation}
Assume now that ${\rm supp}(\xi_{(s,0,\ldots),p})\subseteq {\rm supp}(\xi_{(
p_1,p_2,0,\ldots),p})$. Since $z_3,z_4<p$ and $y_1>p$, we have $y_1=z_1$ or $y_1=z_2$.  The former 
case implies that $s=p_1+p_2$. If we instead assume that $y_1=z_2$ then we must also have $y_2=z_3$ or 
$y_2=z_4$. If $y_2=z_4$, then $s=p_1+p_2$. On the other hand, $y_1=z_2$ and $y_2=z_3$ imply 
after a short calculation that $p=1/2$, which is a contradiction. Hence we can conclude

\begin{equation}\label{e.xi6support}
\rho(S_1\setminus \{\nu_{(p_1+p_2,0,\ldots)}\})=0.
\end{equation}
Also observe that from the above it follows that

\begin{equation}\label{e.keep1}
\mbox{$s=p_1+p_2$ implies that $y_1=z_1$ and $y_2=z_4$.}
\end{equation}

{\bf Case 3:} Assume that $\nu_{(s_1,s_2,0,\ldots)}\in S_2$. We consider four subcases. 

{\it Case 3(i):} Suppose that $s_1\neq s_2$ and $p_1\neq p_2$. Then 
${\rm supp}(\xi_{(s_1,s_2,0,\ldots),p})\subseteq {\rm supp}(\xi_{(p_1,p_2,0,\ldots),p})$ implies that, since 
both supports have four elements, 
$$
{\rm supp}(\xi_{(s_1,s_2,0,\ldots),p})={\rm supp}(\xi_{(p_1,p_2,0,\ldots),p}).
$$ 
From this it is easy to conclude (using~\eqref{e.xisupport}) that $s_1=p_1$ and $s_2=p_2$ so that 
$\nu_{(s_1,s_2,0,\ldots)}= \nu_{(p_1,p_2,0,\ldots)}$. 

{\it Case 3(ii)}: Suppose that $s_1=s_2$ and $p_1=p_2$. Then, arguing similarly as in case 3(i), 
${\rm supp}(\xi_{(s_1,s_2,0,\ldots),p})\subseteq {\rm supp}(\xi_{(p_1,p_2,0,\ldots),p})$  implies that 
$\nu_{(s_1,s_2,0,\ldots)}= \nu_{(p_1,p_2,0,\ldots)}$.

{\it Case 3(iii)}: Suppose that $s_1\neq s_2$ and $p_1=p_2$. Then 
$|{\rm supp}(\xi_{(p_1,p_2,0,\ldots),p})|=3$ while $|{\rm supp}(\xi_{(s_1,s_2,0,\ldots),p})|=4$, and so 
${\rm supp}(\xi_{(s_1,s_2,0,\ldots),p})$ cannot be a subset of ${\rm supp}(\xi_{(p_1,p_2,0,\ldots),p})$.

{\it Case 3(iv)}: Suppose that $s_1=s_2$ and $p_1\neq p_2$. Then using~\eqref{e.xisupport} we see that
\begin{equation}\label{e.xi2support}
\xi_{(s_1,s_2,0,\ldots),p}=\left\{\begin{array}{lll} 
q_1:=p+2 s_1(1-p)&{\rm w.p.}&p^2\\ q_2:=p+s_1(1-2p)&{\rm w.p.}&2p(1-p)\\q_3:= p-2 s_1 p &{\rm w.p.}&(1-p)^2
\end{array}\right.\end{equation}
Assume that ${\rm supp}(\xi_{(s_1,s_2,0,\ldots),p})\subseteq {\rm supp}(\xi_{(p_1,p_2,0,\ldots),p})$.
Since $q_1=z_1$ or $q_3=z_4$, we have $2s_1=p_1+p_2$.
Using~\eqref{e.xisupport} and~\eqref{e.xi2support} we see that $(q_1-q_2,q_2-q_3)=(s_1,s_1)$ and 
$(z_1-z_2,z_2-z_3,z_3-z_4)=(p_2,p_1-p_2,p_2)$. Therefore, since $\{q_1,q_2\}=\{z_1,z_2\}$
or $\{q_2,q_3\}=\{z_3,z_4\}$, we have $s_1=p_2$, contradicting $2s_1=p_1+p_2$ since $p_1\neq p_2$.
Hence, this case can not occur either.

Putting cases $3(i)-3(iv)$ together we can now conclude that 

\begin{equation}\label{e.xi7support}
\rho(S_2\setminus \{\nu_{(p_1,p_2,0,\ldots)}\})=0.
\end{equation}

{\bf Case 4: }Assume now that $\nu_{(t_1,t_2,t_3,0\ldots)}\in S_3$. Unless $t_1=t_2=t_3$ it is 
straightforward to see that ${\rm supp}(\xi_{(t_1,t_2,t_3,0,\ldots),p})$ has at least $5$ elements. 
Hence we can conclude
\begin{equation}
\rho(S_3\setminus\{\nu_{(t,t,t,0,\ldots)}\,:\,t\in (0,1/3]\})=0.
\end{equation}
So assume now that $t_1=t_2=t_3=t$ for some $t\in (0,1/3]$. We get that, again using~\eqref{e.xiequiv} that

\begin{equation}\label{e.xi4support}
\xi_{(t,t,t,0,\ldots),p}=\left\{\begin{array}{lll} 
x_1:=p+t(3-3p)&{\rm w.p.}&p^3\\ x_2:=p+t(2-3p) &{\rm w.p.}&3p^2(1-p)\\ x_3:=p+t(1-3p)&{\rm w.p.}&3p(1-p)^2\\ x_4:=p-3tp&{\rm w.p.}&(1-p)^3
\end{array}\right.\end{equation}

Clearly, if $p_1=p_2$, then ${\rm supp}(\xi_{(t,t,t,0,\ldots),p})$ is not a subset of 
${\rm supp}(\xi_{(p_1,p_2,0,\ldots),p})$, so assume that $p_1\neq p_2$. Then in order to have 
${\rm supp}(\xi_{(t,t,t,0,\ldots),p})\subseteq {\rm supp}(\xi_{(p_1,p_2,0,\ldots),p})$ we must have 
${\rm supp}(\xi_{(t,t,t,0,\ldots),p})= {\rm supp}(\xi_{(p_1,p_2,0,\ldots),p})$. This implies that 
$x_1=z_1$ so that $t=(p_1+p_2)/3$. So we can conclude that

\begin{equation}\label{e.xi5support}
\rho(S_3\setminus \{\nu_{((p_1+p_2)/3,(p_1+p_2)/3,(p_1+p_2)/3,0,\ldots)}\})=0.
\end{equation}

So we must have
\begin{equation}\label{e.keep2}
\mbox{$t=(p_1+p_2)/3$ and $(x_1,x_2,x_3,x_4)=(z_1,z_2,z_3,z_4)$.}
\end{equation}

{\bf Case 5:} Finally we show that we do not need to consider $S_k$ for $k\ge 4$. Observe that if 
$\paintbox\in S_k$, then it is straightforward to check that $|{\rm supp}(\xipp)|\ge k+1$. Since 
$|{\rm supp}(\xi_{(p_1,p_2,0,\ldots),p})|\le 4$, we can conclude that 
\begin{equation}
\rho(S_k)=0\mbox{ for every }4\le k\le \infty.
\end{equation}

From~\eqref{e.keep3}, ~\eqref{e.xi6support}, ~\eqref{e.xi7support} and ~\eqref{e.xi5support} above, we 
see that to show that $\rho=\delta_{\nu}$ and thereby finish the proof it suffices to show that we 
cannot find $\alpha,\beta\in [0,1]$ with 
$\alpha+\beta\le 1$ such that
\begin{equation}\label{e.alfabetaeq}
\Phi_p(\nu_{(p_1,p_2,0,\ldots)})=\alpha \Phi_p(\nu_{(0,\ldots)})+\beta\Phi_p(\nu_{(p_1+p_2,0,\ldots)})+(1-\alpha-\beta)\Phi_p(\nu_{(\frac{p_1+p_2}{3},\frac{p_1+p_2}{3},\frac{p_1+p_2}{3},0,\ldots)}).
\end{equation}

Comparing~\eqref{e.xisupport} with~\eqref{e.xi1support},~\eqref{e.xi3support} and~\eqref{e.xi4support} we 
see that in order for~\eqref{e.alfabetaeq} to hold, it is necessary that 
(keeping~\eqref{e.keep3},~\eqref{e.keep1} and~\eqref{e.keep2} in mind)

\begin{equation}\label{e.finishingpunch}
\begin{array}{rl}p^2=&\beta p+(1-\alpha-\beta)p^3 \\ p(1-p) =&\alpha {\bf 1}\{\frac{p_1}{p}=\frac{p_2}{(1-p)}\} +(1-\alpha-\beta)3 p^2(1-p)\\ p(1-p)=& (1-\alpha-\beta)3 p(1-p)^2 \\ (1-p)^2=&\beta(1-p)+(1-\alpha-\beta)(1-p)^3
\end{array} 
\end{equation}
Since $p\in (0,1)$, the third equation gives that $1-\alpha-\beta\neq 0$. Therefore,
since $p\in(1/2,1)$, the right hand side of the second equation of~\eqref{e.finishingpunch} is strictly 
larger than the right hand side of the third equation. Hence, the linear system in~\eqref{e.finishingpunch} 
does not have any solution for $\alpha,\beta\in [0,1]$ with $\alpha+\beta\le 1$ when $p\in (1/2,1)$. 
\qed

\begin{prop}
Let $t\in (0,1/3]$. Then $\nu_{(t,t,t,0,\ldots)}\in S_3$ is $(\rer_{\N}^{\exch},p)$-unique for 
every $p\in (0,1)\setminus\{1/2\}$. (This is false for $p=1/2$ by Theorem~\ref{p.unprop}.)
\end{prop}

{\bf Proof.} The strategy of this proof is the same as in the proof of Proposition~\ref{p.uniqueprop1}, so we 
will be somewhat briefer. By symmetry we can assume that $p\in (1/2,1)$. Fix $t\in (0,1/3]$ and let 
$\nu:=\nu_{(t,t,t,0,\ldots)}$. Assume that $\tilde{\nu}\in \rer_{\N}^{\exch}$ is such that 
$\Phi_p(\nu)=\Phi_p(\tilde{\nu})$. Let $\rho=\rho_{\tilde{\nu}}$ be the unique probability measure on 
$\rer_{\N}^{\exch,\pure}$ such that

\begin{equation}
\tilde{\nu}=\int_{\nu_{\bf p}\in \rer_{\N}^{\exch,\pure}}\nu_{\bf p}d\rho(\nu_{\bf p}).
\end{equation}

As above, we will show that $\rho=\delta_{\nu}$.

{\bf Case 1: }First we consider $S_0=\{\nu_{(0,\ldots)}\}$. Recall that ${\rm supp}(\xi_{(0,\ldots),p})=\{p\}$. 
Using~\eqref{e.xi4support}, we see that only if $p=2/3$ can we have that 
$p\in {\rm supp}(\xi_{(t,t,t,0,\ldots),p})$. Hence, 
\begin{equation}\label{e.anothereq1}
\mbox{ If }\rho(\nu_{(0,\ldots)})>0,\mbox{ then }p=2/3.
\end{equation}

{\bf Case 2: }Now suppose that $\nu_{(s,0,\ldots)}\in S_1$. Recall that 
${\rm supp}(\xi_{(s,0,\ldots),p})=\{y_1,y_2\}$ from~\eqref{e.xi3support} and 
${\rm supp}(\xi_{(t,t,t,0,\ldots),p})=\{x_1,x_2,x_3,x_4\}$ from~\eqref{e.xi4support}. We have that 
$y_1>p$, $y_2<p$, $x_1>p$ and (since $p>1/2$), $x_3<p$. Hence if 
${\rm supp}(\xi_{(s,0,\ldots),p})\subseteq {\rm supp}(\xi_{(t,t,t,0,\ldots),p})$ it must be the case that 
$y_1=x_1$ or $y_1=x_2$. First, if $y_1=x_1$, we get that $s=3t$. If $y_1=x_2$ and $y_2=x_3$ then 
$s=t(2-3p)/(1-p)$ and $s=-t(1-3p)/p$, and these two equations give that $p=1/2$, which is a contradiction. 
Finally, if $y_1=x_2$ and $y_2=x_4$ then $s=t(2-3p)/(1-p)$ and $s=3t$, and it is easy to see that these two 
equations can not hold at the same time for any $p$. Hence, we can conclude that 

\begin{equation}\label{e.anothereq2}
\rho(S_1\setminus \{\nu_{(3t,0,\ldots)}\})=0.
\end{equation}

Also observe that
\begin{equation}\label{e.anothereq3}
\mbox{If $s=3t$ then $y_1=x_1$ and $y_2=x_4$.}
\end{equation}

{\bf Case 3: }Now assume that $\nu_{(p_1,p_2,0,\ldots)}\in S_2$. Recall from~\eqref{e.xisupport} that 
${\rm supp}(\xi_{(p_1,p_2,0,\ldots),p})=\{z_1,z_2,z_3,z_4\}$, where the four elements are distinct when 
$p_1\neq p_2$, and $z_1>z_2=z_3>z_4$ if $p_1=p_2$. Also observe that we have 
$(x_1-x_2,x_2-x_3,x_3-x_4)=(t,t,t)$ and, as before, $(z_1-z_2,z_2-z_3,z_3-z_4)=(p_2,p_1-p_2,p_2)$. 

{\it Case 3(i)}: Assume that $p_1\neq p_2$. From the above, we see that in order to have 
${\rm supp}(\xi_{(p_1,p_2,0,\ldots),p})\subseteq{\rm supp}(\xi_{(t,t,t,0,\ldots),p})$ we must have $t=p_2=p_1-p_2$, 
which implies $p_1=2p_2=2 t$.

{\it Case 3(ii)}: Assume that $p_1=p_2$. From the above, it follows that in order to have 
${\rm supp}(\xi_{(p_1,p_2,0,\ldots),p})\subseteq {\rm supp}(\xi_{(t,t,t,0,\ldots),p})$ we must have 
$t=p_2$. However, in this case $|{\rm supp}(\xi_{(p_1,p_2,0,\ldots),p})|=3$ so we must also have 
$z_1=x_1$ or $z_4=x_4$. Each of these two cases imply that $t=(p_1+p_2)/3$, which contradicts $t=p_2$ 
since $p_1=p_2$.

From {\it Case 3(i)} and {\it Case 3(ii)} we conclude that

\begin{equation}\label{e.anothereq4}
\rho(S_2\setminus \{\nu_{(2t,t,0,\ldots)}\})=0.
\end{equation}

Also observe that

\begin{equation}\label{e.anothereq5}
\mbox{If $p_1=2t$ and $p_2=t$ then $(x_1,x_2,x_3,x_4)=(z_1,z_2,z_3,z_4)$.}
\end{equation}

{\bf Case (4):} Now assume that $\nu_{(t_1,t_2,t_3,0,\ldots)}\in S_3$. If $(t_1,t_2,t_3)\neq (t',t',t')$ 
for some $t'$, then $|{\rm supp}(\xi_{(t_1,t_2,t_3,0,\ldots),p})|>|{\rm supp}(\xi_{(t,t,t,0,\ldots),p})|$. 
Next, if $t'\neq t$ and $t'\in (0,1/3]$, then by looking at~\eqref{e.xi4support} we see that 
${\rm supp}(\xi_{(t',t',t',0,\ldots),p})$ cannot be a subset of ${\rm supp}(\xi_{(t,t,t,0,\ldots),p})$. 
It follows that

\begin{equation}\label{e.anothereq6}
\rho(S_3\setminus\{\nu_{(t,t,t,0,\ldots)}\})=0.
\end{equation} 

We now finish in the same way as in the proof of Proposition~\ref{p.uniqueprop1}. From~\eqref{e.anothereq1}, 
~\eqref{e.anothereq2}, ~\eqref{e.anothereq4} and~\eqref{e.anothereq6} above, we see that to show that 
$\rho=\delta_{\nu}$ and thereby finish the proof it suffices to show that we cannot find $\alpha,\beta\in [0,1]$ 
with $\alpha+\beta\le 1$ such that
\begin{equation}\label{e.alfabetaeq2}
\Phi_p(\nu_{(t,t,t,0,\ldots)})=\alpha \Phi_p(\nu_{(0,\ldots)})+\beta\Phi_p(\nu_{(3t,0,\ldots)})+(1-\alpha-\beta)\Phi_p(\nu_{(2t,t,0,\ldots)})
\end{equation}

Comparing~\eqref{e.xi4support} with~\eqref{e.xisupport},~\eqref{e.xi1support} and~\eqref{e.xi3support} we see 
that in order for~\eqref{e.alfabetaeq2} to hold, it is necessary that (keeping~\eqref{e.anothereq1},
~\eqref{e.anothereq3} and~\eqref{e.anothereq5} in mind)

\begin{equation}\label{e.finishingpunch2}
\begin{array}{rl}p^3=&\beta p+(1-\alpha-\beta)p^2 \\ 3p^2(1-p) =&\alpha {\bf 1}\{p=2/3\} +(1-\alpha-\beta) p(1-p)\\ 3p(1-p)^2=& (1-\alpha-\beta)p(1-p) \\ (1-p)^3=&\beta(1-p)+(1-\alpha-\beta)(1-p)^2\end{array} 
\end{equation}

If $p\neq 2/3$, then the second and third equations in~\eqref{e.finishingpunch2} imply that $p=1/2$, 
finishing the proof in this case. If $p=2/3$, then the third equation implies that
$\alpha=\beta=0$, in which case the first equation does not hold, completing this case.
\qed

\subsection{Gaussian and symmetric stable exchangeable processes}\label{s.gaussian}
In this section, we first consider the exchangeable Gaussian threshold process, and then the more general case 
of exchangeable symmetric stable threshold processes. Suppose that $X$ is an exchangeable Gaussian process 
with $N(0,1)$-marginals and pairwise correlations $r\in [0,1]$. Let $\Xi$ be the random distribution used in 
the representation of $X$ from Subsection~\ref{s.definetti}. Observe that in the case $r=0$ we have $\Xi$ is $N(0,1)$ 
a.s.\ and in the case $r=1$ we have $\Xi=\delta_x$ where $x$ has distribution $N(0,1)$. For general 
$r\in [0,1]$, $\Xi$ is $N(r^{1/2} W, (1-r))$ where $W$ is $N(0,1)$. We can equivalently obtain $X$ 
as follows: Let $W,U_1,U_2,\ldots$ be i.i.d.\ $N(0,1)$ and let $X_i:=r^{1/2} W+(1-r)^{1/2} U_i$. 

Now let $Y^h$ be the $h$-threshold process obtained from $X$ as described in Subsection~\ref{s.definetti},
where $r$ is suppressed in the notation.
A straightforward calculation left to the reader shows that (recall~\eqref{e.projectexch})
\begin{equation}
\label{e.xirepr}
\xi_{Y^h}=\Xi([h,\infty])=\int_{\frac{h-r^{1/2}W}{(1-r)^{1/2}}}^{\infty} \frac{e^{-t^2/2}}{\sqrt{2\pi}}\,dt.
\end{equation}
In particular, if $h=0$ and $r=1/2$, then we see that
$$
\xi_{Y^0}=\Xi([0,\infty])=1-\Phi(-W),
$$
where $\Phi$ is the probability distribution function of the $N(0,1)$-distribution. Now $\Phi(-W)$ is 
uniformly distributed on $[0,1]$, and hence so is $\Xi([0,\infty])$. 

By symmetry and Theorem~\ref{t.mainp12}, we can conclude that for $h=0$ and any $r$, $Y^0$ is a color process.
Observe that if ${\bf p}=(p_1,\ldots)$ 
where $p_i=1/2^i$ for $i\ge 1$, then the random variable $\xi_{{\bf p},1/2}$ in~\eqref{e.xiequi2} is 
uniformly distributed on $[0,1]$. It follows that when $r=1/2$, $Y^0$ is the color process associated to the paintbox
$(1/2,1/4,1/8,\ldots)$. 

Now we move on to the symmetric stable case. Recall that a stable distribution is characterized by four 
parameters: the location parameter $\mu\in \R$, the skewness parameter $\beta\in [-1,1]$, the scale 
parameter $c\in (0,\infty)$ and the stability parameter $\alpha\in (0,2]$. Here we consider only the 
special case when $\mu=0$, $c=1$ and $\beta=0$. In this case, the characteristic function of the stable 
distribution with stability parameter $\alpha$ is given by $e^{-|t|^{\alpha}}$, $t\in \R$. We denote this 
distribution by ${\mathcal S}(\alpha)$. If $\alpha=2$, then we (essentially) get the $N(0,1)$ distribution, the case of 
which we already covered above.

We obtain an exchangeable process where the marginals are ${\mathcal S}(\alpha)$ as follows. First 
recall that if $|a|^{\alpha}+|b|^{\alpha}=1$ and $V_1,V_2\in {\mathcal S}(\alpha)$, then 
$aV_1+ b V_2\in {\mathcal S}(\alpha)$. Let $W,U_1,U_2,\ldots\in {\mathcal S}(\alpha)$ be i.i.d.\ and fix 
$a\in (0,1)$. Let $b=(1-a^{\alpha})^{1/\alpha}$ and let $X=(X(i))_{i\in \N}$ where $X_i=a W+b U_i$. 
Then $X$ is clearly exchangeable with marginals given by ${\mathcal S}(\alpha)$. Let $Y^h$ be the 
$h$-threshold process obtained from $X$. This depends on $\alpha$ and $a$ but this is suppressed in the notation.
In the same way as in the Gaussian case, one gets that 
$$
\xi_{Y^h}=1-F\left(\frac{h-aW}{b}\right)
$$
where $F$ be the distribution function of $W$. We see that in the special case of $h=0$ and 
$a=b=(1/2)^{1/\alpha}$ we have that $\xi_{Y^0}$ is uniform on $[0,1]$. 

By symmetry and Theorem~\ref{t.mainp12}, we can conclude that for $h=0$ and any $\alpha$ and $a$, 
$Y^0$ is a color process. As in the Gaussian case, we have that when $a=(1/2)^{1/\alpha}$,
$Y^0$ is the color process associated to the paintbox $(1/2,1/4,1/8,\ldots)$. In particular, the $0$-threshold
Gaussian for $r=1/2$ is the same process as the $0$-threshold stable process when $a=(1/2)^{1/\alpha}$.

\section{Connected random equivalence relations on ${\mathbb Z}$}\label{s.conn}

In this section, we focus on the class of connected RERs on ${\mathbb Z}$ thought of as a graph
with nearest neighbor edges. Therefore, in this case, all of the clusters are of the form
$\phi=\{m,m+1,\ldots,n\}$ with $-\infty\le m\le n\le \infty$. 
For $m\in \Z$, the edge between $m$ and $m+1$ will be denoted by $e_{m,m+1}$.
The next definition gives a way of creating an element from $\rer_{\Z}^{\conn}$ by using a process 
on the edges of $\Z$.

\begin{df}\label{d.rergen}
Let $\{Y(e_{n,n+1})\}_{n\in {\mathbb Z}}$ be any process on the edges of $\Z$ with state space 
$\{-1,1\}$  Define $\pi_Y$ to be the random equivalence relation on ${\mathbb Z}$ obtained as 
follows:  $m<n\in \Z$ are said to be in the same equivalence class of $\pi_Y$ if and only 
if $Y(e_{m,m+1})=\ldots=Y(e_{n-1,n})=1$.
\end{df}

Observe that $Y$ and $\pi_Y$ can be recovered from each other. It follows that $\pi_Y$ will 
inherit any property which $Y$ has. We will often say that $\pi_Y$ is induced by $Y$.

\begin{df}
Let  $\{Y(e_{n,n+1})\}_{n\in {\mathbb Z}}$ be any process with state space $\{-1,1\}$. We denote 
by $X^{Y,p}$ the color process obtained from the RER induced by $Y$ with parameter $p$.
\end{df}

In the next proposition we describe exactly which Markov chains with state space $\{0,1\}$ 
are color processes. In some sense, most of this proposition is well known.
\begin{prop}\label{t.markovcolor}
Let $Z=(Z(n))_{n\in {\mathbb Z}}$ be a Markov chain with state-space $\{0,1\}$ and transition 
probabilities $p_{0,0},p_{0,1},p_{1,0}$ and $p_{1,1}$.  The following statements are equivalent:
\begin{enumerate}
\item\label{i.it1} For all $m,n\in {\mathbb Z}$, $Cov(Z(m),Z(n))\ge 0$
\item\label{i.it2} $p_{0,1}\le p_{1,1}$
\item\label{i.it3} $(Z(n))_{n\in {\mathbb Z}}$ is a color process
\item\label{i.it4} $(Z(n))_{n\in {\mathbb Z}}$ satisfies the FKG lattice condition
\item\label{i.it5} $(Z(n))_{n\in {\mathbb Z}}$ satisfies positive associations
\end{enumerate}
\end{prop}

{\bf Proof. } $~\ref{i.it1}\Longrightarrow ~\ref{i.it2}:$ This is completely straightforward.\\
$~\ref{i.it2}\Longrightarrow ~\ref{i.it3}:$ Assume that $p_{0,1}\le p_{1,1}$. 
Let $\{Y(e_{n,n+1})\}_{n\in {\mathbb Z}}$ be an i.i.d.\ process with 
$$P(Y(e_{n,n+1})=1)=p_{1,1}-p_{0,1}=1-P(Y(e_{n,n+1})=0).$$ We now claim that the color process 
$X^{Y,p}$ where $p=p_{0,1}/(p_{0,1}+p_{1,0})$ has the same law as $Z$. First we show that 
$X^{Y,p}$ has the Markov property. Let $s:=P(Y(e_{n,n+1})=1)$. Fix $n\ge 1$ and 
$i_0,\ldots,i_n\in\{0,1\}$. We have

$$
P(X^{Y,p}(0)=i_0| X^{Y,p}(1)=i_1,\ldots,X^{Y,p}(n)=i_n)=\frac{P(X^{Y,p}(0)=i_0,\ldots,X^{Y,p}(n)=i_n)}{P(X^{Y,p}(1)=i_1,\ldots,X^{Y,p}(n)=i_n)}.
$$

We now observe that conditioned on $\{Y(e_{0,1})=0\}$ the events $\{X^{Y,p}(0)=i_0\}$ and 
$\{X^{Y,p}(1)=i_1,\ldots,X^{Y,p}(n)=i_n\}$ are conditionally independent. This follows from the fact 
that $\{Y(e_{0,1})=0\}$ implies $0$ and $1$ are in different clusters of $\pi_Y$. Hence
\begin{eqnarray*}
\lefteqn{P(X^{Y,p}(0)=i_0,\ldots,X^{Y,p}(n)=i_n|Y(e_{0,1})=0)}\\ & & =P(X^{Y,p}(0)=i_0 | Y(e_{0,1})=0) P(X^{Y,p}(1)=i_1,\ldots,X^{Y,p}(n)=i_n|Y(e_{0,1})=0)\\ & & = P(X^{Y,p}(0)=i_0) P(X^{Y,p}(1)=i_1,\ldots,X^{Y,p}(n)=i_n),
\end{eqnarray*}
where the last equality uses the fact that $Y$ is an i.i.d.\ process.
Observe that $\{Y(e_{0,1})=1\}$ implies $X^{Y,p}(0)=X^{Y,p}(1)$. We get that, again using that $Y$ is i.i.d.\,,
\begin{eqnarray*}
\lefteqn{P(X^{Y,p}(0)=i_0,\ldots,X^{Y,p}(n)=i_n|Y(e_{0,1})=1)}\\ & & ={\bf 1}\{i_0=i_1\}P(X^{Y,p}(1)=i_1,\ldots,X^{Y,p}(n)=i_n).
\end{eqnarray*}

\noindent
Hence,

\begin{eqnarray*}
\lefteqn{P(X^{Y,p}(0)=i_0,\ldots,X^{Y,p}(n)=i_n)}\\ & & =(s {\bf 1}\{i_0=i_1\}+(1-s) P(X^{Y,p}(0)=i_0))P(X^{Y,p}(1)=i_1,\ldots,X^{Y,p}(n)=i_n),
\end{eqnarray*}

\noindent
which implies

\begin{eqnarray*}
\lefteqn{P(X^{Y,p}(0)=i_0| X^{Y,p}(1)=i_1,\ldots,X^{Y,p}(n)=i_n)}\\ & & =s {\bf 1}\{i_0=i_1\}+(1-s) P(X^{Y,p}(0)=i_0),
\end{eqnarray*}
which does not depend on $i_2,\ldots,i_n$. Hence the Markov property of $X^{Y,p}$ follows.

\noindent
It remains to show that the transition probabilities coincide with those of $Z$. We have that

\begin{eqnarray*}
\lefteqn{P( X^{Y,p}(n)=1| X^{Y,p}(n-1)=1)}\\ & & = P(Y(e_{n-1,n})=1)+\frac{p_{0,1}}{p_{0,1}+p_{1,0}}P(Y(e_{n-1,n})=0)\\ & & =p_{1,1}-p_{0,1}+\frac{p_{0,1}}{p_{0,1}+p_{1,0}} (1-p_{1,1}+p_{0,1})\\ & & = p_{1,1}-p_{0,1}+\frac{p_{0,1}}{p_{0,1}+p_{1,0}} (p_{1,0}+p_{0,1})\\ & & = p_{1,1},
\end{eqnarray*}
and
\begin{eqnarray*}
\lefteqn{P( X^{Y,p}(n)=0| X^{Y,p}(n-1)=0)}\\ & & = P(Y(e_{n-1,n})=1)+\frac{p_{1,0}}{p_{0,1}+p_{1,0}}P(Y(e_{n-1,n})=0)\\ & & =p_{1,1}-p_{0,1}+\frac{p_{1,0}}{p_{0,1}+p_{1,0}} (1-p_{1,1}+p_{0,1})\\ & & = p_{1,1}-p_{0,1}+\frac{p_{1,0}}{p_{0,1}+p_{1,0}} (p_{1,0}+p_{0,1})\\ & & = p_{1,1}-p_{0,1}+p_{1,0}\\ & & =1-p_{0,1}\\ & & =p_{0,0}.
\end{eqnarray*}

From the above, it follows that $Z\stackrel{{\mathcal D}}{=}X^{Y,p},$ and so $Z$ is a color process.\\
$~\ref{i.it3}\Longrightarrow ~\ref{i.it1}:$ This follows from the fact that any color process has 
non-negative pairwise correlations. \\
$~\ref{i.it4}\Longrightarrow ~\ref{i.it5}:$ This implication was already mentioned in the paragraph 
following Definition~\ref{d.FKGL}.\\
$~\ref{i.it2}\Longrightarrow ~\ref{i.it4}:$ This is a standard but tedious calculation which we omit.\\
$~\ref{i.it5}\Longrightarrow ~\ref{i.it1}:$ This implication is trivial.
\qed

\medskip
The Ising model on $\Z$ will play an important role in this section from now on. 
However, we will define the Ising model on the edges of $\Z$ since we will use it to generate 
an RER as in Definition~\ref{d.rergen}. Since we now also want to allow a varying external field, 
we regive the definition.

\begin{df}
For $m<n$ let $E_{m,n}=\{e_{m,m+1},\ldots,e_{n-1,n}\}$. Let $J\ge 0$ and $h=(h_e)_{e\in E_{m,n}}$ be 
a sequence of real numbers. Let $\mu_{J,h}^{m,n}$ denote the Ising model with nearest neighbor 
interaction $J$ and edge varying external field $h$ on $E_{m,n}$, i.e. for any 
$x\in \{-1,1\}^{E_{m,n}}$,
$$
\mu_{J,h}^{m,n}(x)=\frac{\exp(J\sum_{i=m}^{n-2} x(e_{i,i+1})x(e_{i+1,i+2})+\sum_{i=m}^{n-1} h(e_{i,i+1}) x(e_{i,i+1}))}{Z_{m,n}}.
$$
Here $Z_{m,n}(J,h)$ is a normalizing constant making $\mu_{J,h}^{m,n}$ into a probability measure. 
The Ising model on the edges of all of $\Z$ is defined as the distributional limit
$$
\mu_{J,h}^{\Z}:=\lim_{n\to \infty\, m\to -\infty} \mu_{J,h}^{m,n},
$$
which is well known to exist.
\end{df}

\noindent
We will denote by $Y_{J,h}^{m,n}$ ($Y_{J,h}^{\Z}$) a random object with law $\mu_{J,h}^{m,n}$ ($\mu_{J,h}^{\Z}$). 

In the proof of Proposition~\ref{t.markovcolor} we saw that discrete time two-state Markov 
chains with non-negative pairwise correlations can be viewed as color processes, where the 
underlying RER is generated by an i.i.d.\ process. Theorem~\ref{t.notnmarkov} below shows that 
if, instead of an i.i.d.\ process, we use a (nontrivial) Ising model to generate an RER,
then the resulting process is not $n$-step Markov for any $n\ge 2$. 
First, we give some more preliminary results. The first 
proposition might be of independent interest.

\begin{prop}\label{p.isingmod}
Let $J\ge 0$ and let the (possibly edge dependent) external field $h$ be arbitrary.  Then for any 
$0\le k \le l\le n$ and any $p$,
\begin{equation}
{\mathcal D}(Y_{J,h}^{0,n}\, | \,X^{Y,p}(k)=1,\ldots, X^{Y,p}(l)=1)=\mu_{J,\tilde{h}_{k,l}}^{0,n}
\end{equation}
where $\tilde{h}_{k,l}(e_{i,i+1})=h(e_{i,i+1})-(\log{p})/2$ for $i=k,\ldots,l-1$ and 
$\tilde{h}_{k,l}(e_{i,i+1})=h(e_{i,i+1})$ otherwise and where we write $Y$ for $Y_{J,h}^{0,n}$. 
\end{prop}

\pf  Fix $0\le k\le l\le n$ and 
$y(e_{0,1}),\ldots,y(e_{n-1,n})\in \{-1,1\}^{\{e_{0,1},\ldots,e_{n-1,n}\}}$. Then
\begin{eqnarray}\label{e.bayes}
\lefteqn{{\mathbf P}(Y(e_{0,1})=y(e_{0,1}),\ldots,Y(e_{n-1,n})=y(e_{n-1,n})\, | \,X^{Y,p}(k)=1,\ldots, X^{Y,p}(l)=1)}\nonumber\\ & & = \frac{  {\mathbf P}(X^{Y,p}(k)=1,\ldots, X^{Y,p}(l)=1\,|\,Y(e_{0,1})=y(e_{0,1}),\ldots,Y(e_{n-1,n})=y(e_{n-1,n}))}{ {\mathbf P}( X^{Y,p}(k)=1,\ldots, X^{Y,p}(l)=1 ) }\nonumber \\ & &\mbox{  } \times {\mathbf P}(Y(e_{0,1})=y(e_{0,1}),\ldots,Y(e_{n-1,n})=y(e_{n-1,n})) .
\end{eqnarray}

Let $M(k,l)=M(k,l,Y)$ be the number of equivalence classes in $\pi_{Y}$ intersecting $\{k,..,l\}$. 
For $s=-1,1$ let $$
N_{s}(k,l)=N_{s}(k,l,Y)=|\{i\in \{k,\ldots,l-1\}\,:\,Y(e_{i,i+1})=s\}|.
$$

We observe the identities

\begin{equation}\label{e.classes1}
M(k,l)=1+N_{-1}(k,l),
\end{equation}

and

\begin{eqnarray}\label{e.classes2}
(l-k)-2 N_{-1}(k,l) =N_{1}(k,l)- N_{-1}(k,l)=\sum_{i=k}^{l-1} Y(e_{i,i+1}).
\end{eqnarray}

In what follows, the constant implicit in the proportionality sign $\propto$ is allowed to depend 
only on $J,h,k,l,n$ and $p$. We now get that

\begin{eqnarray}\label{e.bayes2}
 \lefteqn{{\mathbf P}(X^{Y,p}(k)=1,\ldots, X^{Y,p}(l)=1\,|\,Y(e_{0,1})=y(e_{0,1}),\ldots,Y(e_{n-1,n})=y(e_{n-1,n}))}\nonumber\\ & & = p^{M(k,l)}\stackrel{~\eqref{e.classes1}}\propto p^{N_{-1}(k,l)}=\left(\frac{1}{p^{1/2}}\right)^{-2 N_{-1}(k,l)}\propto \left(\frac{1}{p^{1/2}}\right)^{(l-k)-2 N_{-1}(k,l)}\nonumber \\ & &  \stackrel{~\eqref{e.classes2}}{=} \exp\left\{-\frac{\log p}{2} \sum_{i=k}^{l-1} y(e_{i,i+1})\right\}.
\end{eqnarray}
In addition we have

\begin{eqnarray}\label{e.isingnormal}
\lefteqn{{\mathbf P}(Y(e_{0,1})=y(e_{0,1}),\ldots,Y(e_{n-1,n})=y(e_{n-1,n}))}\nonumber \\ & & \propto \exp\left\{ J\sum_{i=0}^{n-2}y(e_{i,i+1}) y(e_{i+1,i+2})+\sum_{i=0}^{n-1} h(e_{i,i+1}) y(e_{i,i+1})\right\}
\end{eqnarray}

Combining~\eqref{e.bayes},~\eqref{e.bayes2} and~\eqref{e.isingnormal}, we get

\begin{eqnarray*}\label{e.final1}
\lefteqn{{\mathbf P}(Y(e_{0,1})=y(e_{0,1}),\ldots,Y(e_{n-1,n})=y(e_{n-1,n})\, | \,X^{Y,p}(k)=1,\ldots, X^{Y,p}(l)=1)}\\ & & \propto \exp\left\{ J \sum_{i=0}^{n-2}y(e_{i,i+1}) y(e_{i+1,i+2})+\sum_{i=0}^{k-1} h(e_{i,i+1}) y(e_{i,i+1}) \right. \\ & & \left. +\sum_{i=k}^{l-1} \left(h(e_{i,i+1})-\frac{\log p}{2} \right) y(e_{i,i+1}) + \sum_{i=l}^{n-1} h(e_{i,i+1}) y(e_{i,i+1}) \right\},
\end{eqnarray*} 
finishing the proof of the proposition.\qed

\begin{lma}\label{l.percus}
Let $J>0$ and let the (possibly edge dependent) external field $h$ be arbitrary.  Then for any 
$n\in {\mathbb Z}$ and $k \le l \in \Z$,

\begin{eqnarray}\label{e.strictineq}
\lefteqn{{\mathbf E}(Y_{J,h}^{\Z}(e_{n,n+1})\,| \, X^{Y_{J,h}^{\Z},p}(k)=1,\ldots, X^{Y_{J,h}^{\Z},p}(l)=1)}\nonumber \\ & & >{\mathbf E}(Y_{J,h}^{\Z}(e_{n,n+1}) \,| \, X^{Y_{J,h}^{\Z},p}(k)=1,\ldots, X^{Y_{J,h}^{\Z},p}(l-1)=1).
\end{eqnarray}
If $k=l$, then there is no conditioning on the right hand side of the above.
\end{lma}

\pf In the proof, we will work on the interval $[-N,N]$ and keep $J$ fixed, so we write 
$Y^{N}_{h}=Y^{-N,N}_{J,h}$, and in addition we write $Y_{h}=Y^{\Z}_{J,h}$. Without loss of 
generality, we can choose $n=0$ and so we will be done if we show that for any fixed $k\le l$,

\begin{eqnarray}\label{e.enoughts}
\lefteqn{\lim_{N\to \infty } {\mathbf E}(Y^{N}_{h}(e_{0,1})\,|\,X^{Y^{N}_{h},p}(k)=1,\ldots, X^{Y^{N}_{h},p}(l)=1)}\nonumber \\ & & > \lim_{N\to \infty } {\mathbf E}(Y^{N}_{h}(e_{0,1})\,|\,X^{Y^{N}_{h},p}(k)=1,\ldots, X^{Y^{N}_{h},p}(l-1)=1),
\end{eqnarray}
since the LHS and RHS in~\eqref{e.enoughts} coincide with the LHS and RHS of~\eqref{e.strictineq} 
respectively.

If $N>\max(|k|,|l|)$, then we know from Proposition~\ref{p.isingmod} that

$$
{\mathcal D}(Y^{N}_{h}\, |\, X^{Y^{N}_{h},p}(k)=1,\ldots, X^{Y^{N}_{h},p}(l-1)=1) =  {\mathcal D}(Y^{N}_{\tilde{h}_{k,l-1}}),
$$

and

$$
{\mathcal D}(Y^{N}_{h}\, |\, X^{Y^{N}_{h},p}(k)=1,\ldots, X^{Y^{N}_{h},p}(l)=1) =  {\mathcal D}(Y^{N}_{\tilde{h}_{k,l}}),
$$
where $\tilde{h}_{k,l}$ is given in the statement of Proposition~\ref{p.isingmod}.
It is well known and easy to prove (see \cite{ELLIS} p.148) that  for all $i$ and $j$,
$$
\frac{\partial {\mathbf E}[Y^{N}_{h}(e_{j,j+1})]}{\partial h(e_{i,i+1})}={\bf Cov}(Y^{N}_{h}(e_{i,i+1}),Y^{N}_{h}(e_{j,j+1})).
$$
This implies that
\begin{eqnarray}\label{e.corrint}
\lefteqn{{\mathbf E}[Y^{N}_{\tilde{h}_{k,l}}(e_{0,1})]-{\mathbf E}[Y^{N}_{\tilde{h}_{k,l-1}}(e_{0,1})] }\nonumber \\& & = \int_{h(e_{l-1,l})}^{h(e_{l-1,l})-(\log p )/2}{\bf Cov}(Y^{N}_{s}(e_{0,1}),Y^{N}_{s}(e_{l-1,l}))d s(e_{l-1,l}),
\end{eqnarray}
where $s(e_{i,i+1})=\tilde{h}_{k,l-1}(e_{i,i+1})$ for $i\neq l-1$. As $N\to \infty$, the right hand 
side of~\eqref{e.corrint} converges (by the bounded convergence theorem) to 
$$
\int_{h(e_{l-1,l})}^{h(e_{l-1,l})-(\log p )/2}{\bf Cov}(Y_{s}(e_{0,1}),Y_{s}(e_{l-1,l}))d s(e_{l-1,l}).
$$
Since $J>0$, this last expression is strictly positive. 
(Percus' equality (\cite{P75}, see also \cite{ELLIS} p.142) gives the weaker fact that the expression
is nonnegative.) Now~\eqref{e.enoughts} follows. \qed

\medskip

In what follows, we write, as in the proof of Lemma~\ref{l.percus}, $Y_{h}=Y^{\Z}_{J,h}$.

\begin{thm}\label{t.notnmarkov}

Let $J> 0$ and let the external field $h$ be constant but arbitrary. Then the color process 
$X^{Y_h,p}$ is not $n$-step Markov for any $n\ge 1$ unless $p\in \{0,1\}$.
\end{thm}

\pf Observe that

\begin{align*}
{\mathbf P}(&X^{Y_h,p}(0)=1| X^{Y_h,p}(1)=1,\ldots, X^{Y_h,p}(n)=1)\\ &= {\mathbf P}(Y_h(e_{0,1})=1\,|\, X^{Y_h,p}(1)=1,\ldots, X^{Y_h,p}(n)=1)\\ & +p\, {\mathbf P}(Y^{h}(e_{0,1})=0\,|\, X^{Y_h,p}(1)=1,\ldots, X^{Y_h,p}(n)=1)\\& =p+(1-p){\mathbf P}(Y_h(e_{0,1})=1\,|\, X^{Y_h,p}(1)=1,\ldots, X^{Y_h,p}(n)=1).
\end{align*}

Lemma~\ref{l.percus} says that the last expression is strictly increasing in $n$ and so the theorem 
is proved. \qed

\section{Stochastic domination of product measures}\label{s.dom}

Given $\nu$ and $p$, it is natural to ask which product measures the color process $\Phi_p(\nu)$ 
stochastically dominates. In this section, we present results in this direction. We write 
$\mu_1\preceq \mu_2$ if $\mu_2$ stochastically dominates $\mu_1$ which we recall means that 
the two measures can be coupled so that the joint distribution is concentrated on pairs of 
configurations where the realization for $\mu_1$ is below the realization for $\mu_2$.

To begin with, the following definition is natural.

\begin{df}
Let $V$ be a finite or countable set and let $\nu\in\rer_V$.
For $p\in (0,1)$, let $d(\nu,p):=\max\{\alpha:\Pi_\alpha\preceq \Phi_p(\nu)\}$.
We also let $d(\nu):=\lim_{p\to 1}d(\nu,p)$. 
($\Pi_s$ denotes as before product measure on $\{0,1\}^{V}$ with density $s$.)
\end{df}

\subsection{Some general results for stochastic domination}

At first, one might think that $d(\nu)$ should often be 1. However, this is usually not 
the case; see e.g.\ Proposition~\ref{p.expdomlemma1}(ii) below. Our first proposition tells 
us that $d(\nu)=1$ does hold if the cluster sizes are bounded.

\begin{prop}\label{p.BoundedCluster}
Suppose that $\nu\in \rer_{V}$ where $V$ is an arbitrary set and that
\begin{equation}\label{e.ClusterBounded}
\nu(\{\pi\,:\,|\phi|\le M\mbox{ for all }\phi\in \pi\})=1.
\end{equation}
Then for all $p\in (0,1)$,
$$
d(\nu,p)\ge 1-(1-p)^{\frac{1}{M}}
$$
and hence $d(\nu)=1$.
\end{prop}

\pf
Suppose first that $\pi\in {\rm Part}_{V}$ is such that $\pi$ contains only 
equivalence classes of size at most $M$. Letting $\alpha:=1-(1-p)^{\frac{1}{M}}$,
it is straightforward to show that $\Pi_{\alpha}\preceq \Phi_p(\delta_{\pi})$ where $\delta_\pi$ 
stands for the point measure at $\pi$.  Now write 
$$
\Phi_p(\nu)=\int_{\pi\in {\rm Part}_{V}} \Phi_p(\delta_{\pi})d\nu(\pi).
$$
The claim now follows, since $\Pi_{\alpha}\preceq \Phi_p(\delta_{\pi})$ for $\nu$-almost every $\pi$. \qed

\medskip

The next proposition, due to Olle H\"{a}ggstr\"{o}m, shows that having uniformly bounded cluster sizes is 
not a necessary condition for $d(\nu)=1$.

\begin{prop}\label{p.Olle}
There exists an RER $\nu$ with $d(\nu)=1$ for which the supremum of the cluster sizes is infinite a.s.
\end{prop}

\pf
The main step is to first construct an RER $\nu$ with $d(\nu)=1$ for which~\eqref{e.ClusterBounded} fails 
for each $M$.
To do this, let $V_2,V_3,\ldots$ be disjoint finite sets with $|V_k|=k$ for each $k$ and let $V=\cup_{k\ge 2} V_k$.
Given a sequence $(\epsilon_k)$, we consider the RER $\nu$ on $V$ obtained as follows.
Independently for different $k$, we let $V_k$ be a cluster with probability $\epsilon_k$
and we let all the elements of $V_k$ to be singletons with probability $1-\epsilon_k$. Clearly if 
$\epsilon_k>0$ for each $k$, then~\eqref{e.ClusterBounded} fails for each $M$.
We now claim that if $\epsilon_k=\frac{1}{2^{k^2}}$, then $d(\nu)=1$. We need to show that for each 
$\alpha <1$, there is $p<1$ so that $\Pi_\alpha\preceq \Phi_p(\nu)$. 
Since the behavior on different $V_k$'s is independent under $\nu$,
we only need to check the stochastic domination for each $V_k$. We first check that we can obtain the 
desired inequality for the (decreasing) event of having all 0's. This inequality is then
$$
(1-\alpha)^k \ge \eps_k (1-p)+(1-\eps_k) (1-p)^k
$$
and it is easy to check that with $\epsilon_k=\frac{1}{2^{k^2}}$ as above, 
given any $\alpha <1$, there is $p<1$ so that 
this inequality holds for all $k$. Theorem 1.3 in \cite{LS06} states that a finite exchangeable process which
satisfies the FKG lattice condition dominates a given product measure once one has the appropriate 
inequality for the
event of having all 0's. It is not hard to see that the color process above on $V_k$ is exchangeable and
satisfies the FKG lattice condition therefore yielding the desired stochastic domination.

Finally, once we have an RER $\nu$ with $d(\nu)=1$ for which~\eqref{e.ClusterBounded} fails for each $M$,
we can obtain what is claimed in the proposition simply by considering an infinite number of independent such systems.
\qed

\medskip

The next proposition relates stochastic domination with the behavior of the 
number of clusters intersecting a large box.

\begin{prop}\label{p.expdomlemma1}
Let $d\ge 1$, $\nu\in \rer_{\zd}^{\stat}$ and $C_n=C^\nu_n$ be the number of clusters 
intersecting $[-n,n]^d$.

(i). If $p,\alpha \in (0,1)$ is such that we have
$\Pi_\alpha\preceq \Phi_p(\nu)$, then for all $n\ge 0$ and all $k\ge 1$,
\begin{equation}\label{e.dominationLDnewversion}
\nu(C_n\le k)\le \frac{1}{(1-p)^k} (1-\alpha)^{(2n+1)^d}.
\end{equation}

(ii).  If
\begin{equation}\label{e.suff.cond.dzeroagain}
\liminf_{n\to\infty}\frac{-\log \nu(C_n\le \delta  (2n+1)^d)}{(2n+1)^d}\le \epsilon,
\end{equation}
then $d(\nu,p)\le 1-\frac{(1-p)^{\delta}}{e^{\epsilon}}$.
In particular if this $\liminf$ is 0, then $d(\nu,p)\le 1-(1-p)^{\delta}$.

(iii). If there exists $k_n=o(n^d)$ such that
\begin{equation}\label{e.suff.cond.dzero}
\liminf_{n\to\infty}\frac{-\log \nu(C_n\le k_n)}{(2n+1)^d}\le \epsilon,
\end{equation}
then $d(\nu)\le 1-e^{-\epsilon}$. 
In particular if this $\liminf$ is 0, then $d(\nu)=0$. 

(iv).  If $\nu(C_n=1)\ge \gamma^{(2n+1)^d}$ for infinitely many values of $n$,
then $d(\nu)\le 1-\gamma$.
\end{prop}

\pf (i). Fix $p,\alpha \in (0,1)$ with $\Pi_\alpha\preceq \Phi_p(\nu)$ and let $n\ge 0$ and $k\ge 1.$
Then
\begin{eqnarray}\label{e.domincube1}
\nonumber (1-\alpha)^{(2n+1)^d} =\Pi_{\alpha}(X|_{[-n,n]^d}\equiv 0)
\ge \Phi_p(\nu)(X|_{[-n,n]^d}\equiv 0)
\\   =E[(1-p)^{C_n}]\ge \nu(C_n\le k) (1-p)^k.
\end{eqnarray}

(ii). This follows from (i) in a straightforward manner.

(iii). This follows from (ii) in a straightforward manner.

(iv). This follows from (iii) in a straightforward manner.
\qed

\medskip

We next have the following proposition for RERs concentrated on connected classes.

\begin{prop}\label{p.statconncorr}
(i). Let $\nu\in \rer_{\Z}^{\stat,\conn}$. If $p,\alpha \in (0,1)$ is such that we have
$\Pi_\alpha\preceq \Phi_p(\nu)$, then for all $n\ge 1$
\begin{equation}\label{e.dominationLDagain}
\nu(|\pi(0)|\ge n)\le (n+2) \frac{1}{1-p} (1-\alpha)^{2\lfloor n/2 \rfloor +1}.
\end{equation}
It follows that if $\nu(|\pi(0)|\ge n)\ge C\gamma^n$ for infinitely many $n$ for some $C>0$, then
$d(\nu)\le 1-\gamma$.

(ii). There exists $\nu\in \rer_{\Z}^{\stat}$ and $p,\alpha \in (0,1)$ such that 
$\Pi_\alpha\preceq \Phi_p(\nu)$ but where the LHS of~\eqref{e.dominationLDagain} does not go to 0 with $n$.

(iii). There exists $d\ge 2$, $\nu\in\rer_{\zd}^{\rm stat,conn}$ and $p,\alpha \in (0,1)$ such that 
$\Pi_\alpha\preceq \Phi_p(\nu)$ but where the LHS of~\eqref{e.dominationLDagain} does not go to 0 with $n$.

(iv).
Let $\nu\in \rer_{\zd}^{\stat,\conn}$. If $p,\alpha \in (0,1)$ is such that we have
$\Pi_\alpha\preceq \Phi_p(\nu)$, then for all $n\ge 1$
\begin{equation}\label{e.dominationLDagainagain}
\nu(|\pi(0)| \ge n ) \le \frac{(7^d(1-\alpha))^n}{1-p}.
\end{equation}
(This only has content if $\alpha\in (1-7^{-d},1)$.)
It follows that if $\nu(|\pi(0)|\ge n)\ge C\gamma^n$ for infinitely many $n$ for some $C>0$, then
$d(\nu)\le 1-\frac{\gamma}{7^d}$.
\end{prop}

\pf (i). Observe that since $\nu$ produces only connected equivalence classes a.s.\  the 
following inclusion holds a.s.
$$
\{|\pi(0)|\ge n\}\subseteq \bigcup_{i=-\lceil n/2 \rceil}^{\lceil n/2 \rceil} \{\pi(i)\supseteq [i-\lfloor n/2 \rfloor,i+\lfloor n/2 \rfloor]\}.
$$

Hence
\begin{eqnarray}
\lefteqn{\nu(|\pi(0)|\ge n)\le \sum_{i=-\lceil n/2 \rceil}^{\lceil n/2 \rceil} \nu(\pi(i)\supseteq [i-\lfloor n/2 \rfloor,i+\lfloor n/2 \rfloor])}\nonumber \\ & & \le 
(n+2) \frac{1}{1-p} (1-\alpha)^{2\lfloor n/2 \rfloor +1},
\end{eqnarray}
using Proposition~\ref{p.expdomlemma1}(i) with $k=1$ in the last inequality, finishing the proof.
The last statement follows easily.

(ii). 
We use Proposition~\ref{p.dominationexch} which comes later in this section. Assume we 
have a paintbox with $p_1 >0$ and $\sum_i p_i <1$. 
Since $\sum_i p_i <1$, Proposition~\ref{p.dominationexch} says that 
$\Pi_{\alpha} \preceq\Phi_p(\nu)$ for some $\alpha, p\in (0,1)$.
However, since $p_1 >0$, $\nu(|\pi(0)|=\infty)>0$ and so
the LHS  of~\eqref{e.dominationLDagain} does not go to 0 with $n$. 

(iii). Let $\nu\in\rer_{\zd}^{\stat,\conn}$ be the random cluster model with $J>J_c$. 
Then, using the fact that the random cluster model has a unique infinite cluster,
the color process $\Phi_{1/2}(\nu)$ is necessarily given by 
$(\mu_{J}^{\zd,+}+\mu_{J}^{\zd,-})/2$ where these two measures are respectively the plus and 
minus states for the Ising model with coupling constant $J$.
It is well known that there is some $\epsilon=\epsilon(J,d)>0$ such that  
$\Pi_{\epsilon}\preceq \mu_{J}^{\zd,-} (\preceq \mu_{J}^{\zd,+})$ and
hence $\Pi_{\epsilon}\preceq \Phi_{1/2}(\nu)$. However $\nu(|\pi(0)|=\infty)>0$ and
hence the LHS of~\eqref{e.dominationLDagain} does not go to 0 with $n$.

(iv). Let $S_n$ be the set of connected subsets of $\zd$ of size $n$ containing the origin.
It is known that $|S_n|\le 7^{dn}$, see p.$81$ of \cite{grimmett}. We then have
\begin{equation}\label{e.domination11}
\nu( |\pi(0)| \ge n  ) \le \sum_{\phi\in S_n} \nu(\phi \subseteq \pi(0)).
\end{equation}

Since by assumption $\Pi_{\alpha} \preceq\Phi_p(\nu)$, we get, using domination in the 
second inequality, that for any $\phi\in S_n$
$$
(1-p)\nu(\phi \subseteq \pi(0))\le\Phi_p(\nu)(X|_{\phi}\equiv 0)\le (1-\alpha)^n,
$$
so that
\begin{equation}\label{e.domination12}
\nu(\phi \subseteq \pi(0))\le \frac{(1-\alpha)^n}{1-p}.
\end{equation}

From~\eqref{e.domination11} and~\eqref{e.domination12} it follows that
$$ 
\nu( |\pi(0)| \ge n  )\le |S_n| \frac{(1-\alpha)^n}{1-p}\le  \frac{(7^d(1-\alpha))^n}{1-p},
$$
as claimed. The last statement follows easily.
\qed

\begin{remark}
The essential reason that (i) does not hold when $d\ge 2$ is that the number of connected sets of size $n$ 
containing the origin is exponential in $n$ rather than linear in $n$ as in $d=1$. 
\end{remark}

The next proposition says that no matter how fast $\nu(|\pi(0)|\ge n)$ decays to 0 for $d=1$, 
there is no guarantee that $\Phi_p(\nu)$ will dominate any product measure, even for
$\nu\in \rer_{{\mathbb Z}}^{\stat,\conn}$. This shows in particular that the converse 
of Proposition~\ref{p.statconncorr}(i) is false.

\begin{prop}\label{p.nodomination}
Let $(b_n)_{n\ge 1}$ be a decreasing sequence of real numbers such that $b_n\to 0$ as 
$n\to \infty$ and $b_n>0$ for all $n$. Then there exists 
$\nu\in \rer_{{\mathbb Z}}^{\stat,\conn}$ such that $\nu(|\pi(0)|\ge n)\le b_n$ for 
all $n\ge 2$ but $d(\nu)=0$.
\end{prop}

\pf For $n\ge 1$ let $(K_n)_{n\ge 1}$ be uniform on $\{0,\ldots,n-1\}$. For $k\in {\mathbb Z}$ and 
$n\ge 1$, let $I_{k,n}=\{kn,\ldots,kn+n-1\}$. For $n\ge 1$ let $\pi_n$ be the {\rm RER} with 
equivalence classes given by $(I_{k,n}+K_n)_{k\in {\mathbb Z}}$ and let $\nu_n$ be the law of 
$\pi_n$. Let $(p_n)_{n\ge 1}$ satisfy $p_n\in(0,1)$ for all $n$ and $\sum_{n\ge 1}p_n=1$ and
then put $\nu=\sum_{n\ge 1} p_n \nu_n$. We now show that the sequence
$(p_n)$ can be chosen so that $\nu$ satisfies the properties required. 

First, we see that the decay of the probabilities $\nu(|\pi(0)|\ge n)$ can be given the 
desired behavior by an appropriate choice of the sequence $(p_n)_{n\ge 1}$. For example one 
can let $p_1:=1-b_2$ and then $p_n:=b_n-b_{n+1}$ for $n\ge 2$. This gives 
$\nu(|\pi(0)|\ge n)= b_n$ for all $n\ge 2$.

To show that $d(\nu)=0$, we proceed as follows. If $d(\nu)> 0$, then there would exist
$\epsilon,p\in (0,1)$ such that $\Pi_{\epsilon} \preceq\Phi_p(\nu)$. 
Next consider the ergodic decomposition of any stationary coupling of
$\Pi_{\epsilon}$ and $\Phi_p(\nu)$ which couples the former below the latter.
Since $\Pi_{\epsilon}$ is ergodic, it follows that
$\Pi_{\epsilon} \preceq\Phi_p(\sum_{n\ge 1} p_n \nu_n)$
can only occur if $\Pi_{\epsilon} \preceq\Phi_p(\nu_n)$ for each $n$.
However, $\Pi_{\epsilon} \preceq\Phi_p(\nu_n)$ implies that 
$$
\frac{1-p}{n}\le \Phi_p(\nu_n)(X|_{1,\ldots,n}\equiv 0)\le (1-\epsilon)^n
$$
which is clearly false for large $n$. \qed

\subsection{Stochastic domination for the infinitely exchangeable case}

We now turn to the infinitely exchangeable case and give a formula (see 
Proposition~\ref{p.dominationexch} below) $d(\nu,p)$.
Suppose first that $\mu\in {\rm EP}_{\N}$. Recall (see~\eqref{e.nuxi}) that
$$
\mu=\int_{s=0}^1 \Pi_s \, d\rho_{\mu}(s),
$$
for some unique measure $\rho_{\mu}$  on $[0,1]$. The proof of the next lemma is 
straightforward and certainly known, so we omit it.

\begin{lma}\label{l.dominationlemma}
Suppose that $\mu \in {\rm EP}_{\N}$. Then
\begin{equation}
\sup\{s\,:\,\Pi_s\preceq \mu\}=\inf {\rm supp} \,\rho_{\mu}.
\end{equation}
\end{lma}

\bigskip\noindent
Recall (see Theorem~\ref{t.kingman}) that for any $\nu\in \rer_{\N}^{\exch}$, there is 
a unique measure $\rho_{\nu}$ on $\rer_{\N}^{\exch,\pure}$ such that 
\begin{equation}\label{e.rerepr}
\nu=\int_{\nu_{\bf p}\in \rer_{\N}^{\exch,\pure}} \nu_{\bf p} \, d\rho_{\nu}(\nu_{\bf p}).
\end{equation}
\noindent
As an application of Lemma~\ref{l.dominationlemma} to exchangeable color processes, we 
have the following.

\begin{prop}\label{p.dominationexch}
If $\nu_{\bf p}\in \rer_{\N}^{\exch,\pure}$ with ${\bf p}=(p_1,p_2,\ldots)$, then for all $p\in (0,1)$
\begin{equation}\label{e.dominationexch1}
d(\nu_{\bf p},p)=p\left(1-\sum_{i\ge 1} p_i\right).
\end{equation}
Hence
$$
d(\nu_{\bf p})=1-\sum_{i\ge 1} p_i.
$$
More generally, if $\nu\in \rer_{\N}^{\exch}$ then for all $p\in (0,1)$
\begin{equation}\label{e.dominationexch2}
d(\nu,p)=\inf\left\{p\left(1-\sum_{i\ge 1} p_i\right)\,:\,\nu_{\bf p}\in {\rm supp}\,\rho_{\nu}\right\}.
\end{equation}
Hence
$$
d(\nu)=1-\sup\left\{\sum_{i\ge 1} p_i\,:\,\nu_{\bf p}\in {\rm supp}\,\rho_{\nu}\right\}.
$$

\end{prop}

\pf Statement~\eqref{e.dominationexch1} follows from Lemma~\ref{l.dominationlemma} by inspection 
of~\eqref{e.xiequiv}. The general statement~\eqref{e.dominationexch2} follows 
from~\eqref{e.dominationexch1} and the upper semicontinuity of the map \\
$\nu_{\bf p}\mapsto p(1-\sum_{i=1}^{\infty} p_i)$ (which in fact is not continuous) by 
observing that
$$
\Phi_p(\nu)\stackrel{~\eqref{e.rerepr}}{=}\int_{\nu_{\bf p}\in \rer_{\N}^{\exch,\pure}} \Phi_p(\nu_{\bf p}) \, d\rho_{\nu}(\nu_{\bf p}).
$$ 
\qed

\medskip

Next we present a result for the infinite exchangeable case projected to a finite set
which follows from a result in~\cite{LS06}.
For $\nu\in \rer_{\N}^{\exch}$ we let $\nu_{[n]}\in\rer_{[n]}^{\exch}$ stand for the RER on $[n]$ induced by $\nu$. Similarly for $\mu\in {\rm EP}_{\N}$ let $\mu_{[n]}$ be the measure induced by $\mu$ on $\{0,1\}^{[n]}$. 
Corollary $1.1$ in~\cite{LS06} says that for all $\mu\in {\rm EP}_{\N}$ and all $n\ge 1$
$$
\sup\{s\,:\,\Pi_s\preceq \mu_{[n]} \}=1-\left(\int_{s=0}^{1}(1-s)^n d\rho_{\mu}(s)\right)^{1/n}.
$$
This immediately implies the following proposition which we therefore give without proof. Recall the 
definition of $\xi_{{\bf p},p}$ from~\eqref{e.xiequiv}.

\begin{prop}
Let $n\ge 1$ and suppose that $\nu\in \rer_{\N}^{\exch}$. Then
$$
\sup\{s\,:\,\Pi_s\preceq \Phi_p(\nu_{[n]})\}=1-\left(\int_{\nu_{\bf p}\in \rer_{\N}^{\exch,\pure} }\int_{s=0}^1 (1-s)^n\,d F_{\xi_{{\bf p},p}}(s)d\rho_{\nu}(\nu_{\bf p}) \right)^{1/n}.
$$
\end{prop}

\subsection{Stochastic domination for our various models}

In this subsection, we examine what the earlier results in this section tell us about 
stochastic domination for some of our standard models.

\subsubsection{Random walk in random scenery}

\begin{prop}\label{p.recurrent.mean0}
(i). Consider a recurrent random walk on $\zd$ and let $\nu$ be the associated RER on $\Z$. 
Then $d(\nu)=0$.\\
(ii). Consider a random walk on $\zd$ whose steps have mean 0 and let $\nu$ be 
the associated RER on $\Z$. Then $d(\nu)=0$.
\end{prop}

While (ii) is much stronger in some sense than (i), it does not actually imply it since there are recurrent 
random walks with infinite mean.

{\bf Proof.}
(i). It is well known and easy to show that for any recurrent random walk, $E(R_n)=o(n)$ where
$R_n$ is the range of the random walk up to time $n$, i.e., the cardinality of the set
$\{S_0,S_1,\ldots,S_{n-1}\}$. It is clear that $R_n$ is exactly the number of clusters 
intersecting $[0,n-1]$ in the associated RER. Using a trivial modification of 
Proposition~\ref{p.expdomlemma1}(iii) (where $[-n,n]$ is simply replaced by $[0,n-1]$),
we let $k_n:=2E(R_n)$. Then $k_n=o(n)$ and $\nu(R_n\ge k_n)\le \frac{1}{2}$ by 
Markov's inequality and hence $\nu(R_n\le k_n)\ge \frac{1}{2}$. It follows 
that~\eqref{e.suff.cond.dzero} holds in this case with $\epsilon=0$ and hence $d(\nu)=0$ by
Proposition~\ref{p.expdomlemma1}(iii).

(ii). We will use Lemma 2.2 in \cite{JS} which is the following.

\begin{lma} \label{lem:kesten} 
Consider a random walk on $\zd$ whose steps have mean 0. Then for every $\eps >0$, it is 
the case that
$$
P\left({R_n \over n} \le \eps\right) \ge \left({1\over 2}\right)^{\eps n}
$$
holds for large $n$.
\end{lma}

The key ingredient in the proof of the above lemma is Lemma 5.1 in \cite{DONVAR} which 
gives a much stronger result when the distribution of the steps is compact or even 
satisfies much weaker assumptions.

It is easy to see that Lemma~\ref{lem:kesten} implies that we can choose $(\eps_n)$ 
going to 0 such that for all $n\ge 1$
$$
P\left({R_n \over n} \le \eps_n\right) \ge \left({1\over 2}\right)^{\eps_n n}.
$$
Now let $k_n:=n\eps_n$ which is clearly $o(n)$. The above inequality yields 
that~\eqref{e.suff.cond.dzero} holds in this case with $\epsilon=0$ as well and hence $d(\nu)=0$ by 
Proposition~\ref{p.expdomlemma1}(iii).
\qed

\medskip\noindent
Understanding what happens with $d(\nu)$ for 1 dimensional random walk with drift
seems to be an interesting question; see Question~\ref{q.bhs}.

\subsubsection{Stationary distributions for the voter model in $d\ge 3$}

Recall that in this case, the RER $\nu_d$ is described by taking independent coalescing random walkers starting 
at each point of  $\zd$ and running to time $\infty$ and letting two points be in the same class if the random 
walkers started at those two points ever coalesce. 

\begin{prop}\label{p.voting.sd}
For all $d\ge 1$, $d(\nu_d)=0$.
\end{prop}

{\bf Proof.}
For $d=1,2$, $\nu_d$ has a.s.\ 1 cluster and therefore the result is trivial.
For $d\ge 3$, it is stated (in different terminology) on p.\ 60 in \cite{LEBSCHON} that 
$E(C_n)\le O(n^{d-2})$. Letting $k_n:=n^{d-1}(=o(n^d))$ and using Markov's inequality, 
we obtain $\nu(C_n\le k_n)\to 1$ as $n\to\infty$. It follows that~\eqref{e.suff.cond.dzero} holds in this 
case with $\epsilon=0$ and hence $d(\nu)=0$ by Proposition~\ref{p.expdomlemma1}(iii). \qed

\subsubsection{1-dimensional Random Cluster Model}

Consider the RER, denoted by $\nu_s$, in $\rer_{{\mathbb Z}}^{\stat,\conn}$ where one performs i.i.d.\
percolation with parameter $s$ on ${\mathbb Z}$ and considers the connected components.
(This is exactly the RER that arises in Definition~\ref{d.rergen} where the $Y$ process is i.i.d.\ with marginal
probability $s$.) 

\begin{prop}\label{p.1Drcm}
$d(\nu_s,p)= p-ps$ and hence, by letting $p\to 1$, $d(\nu_s)= 1-s$. 
\end{prop}

{\bf Proof.} By the proof of Proposition~\ref{t.markovcolor}, 
as we vary $s$ and $p$, the collection of color processes that we obtain
are exactly the set of 2 state Markov chains with nonnegative correlations and the correspondence
is given by $s=p_{1,1}-p_{0,1}$ and $p=p_{0,1}/(p_{0,1}+p_{1,0})$.

Now, by Proposition 5.1 in \cite{LS06}, the maximal density product measure that our Markov
chain dominates has density $p_{0,1}$. Next, we want to express this in terms of $s$ and $p$.
Inverting the above set of equations yields $p_{0,1}=p-ps$ and $p_{1,1}=s+p-ps$.
It follows that $d(\nu_s,p)= p-ps$, as desired. \qed

\medskip

We point out that, in the terminology of Proposition~\ref{p.expdomlemma1}, we clearly have that
$\nu_s(C_n=1)= s^{2n}$ and hence we can conclude from Proposition~\ref{p.expdomlemma1}(iv)
that $d(\nu_s)\le 1-s$. Hence Proposition~\ref{p.expdomlemma1}(iv) is sharp in this case.

Finally, we recall that the above set of color processes (as $s$ and $p$ vary)
corresponds to the set of 1-dimensional nearest neighbor Ising models as we vary 
$J\ge 0$ and $h\in {\mathbb R}$. 
Using the exact correspondence given in \cite{Georgii}, p. 50-51 between Ising models on ${\mathbb Z}$
and the above processes, one can can determine the largest product measure which the Ising model
with parameters $J\ge 0$ and $h\in {\mathbb R}$ dominates.

\subsubsection{Random Cluster models in $\zd$}

We refer to \cite{GrimRC} for all background concerning the random cluster model. Given $d\ge 2$, $\alpha\ge 0$ 
and $q\ge 1$, we let $\nu^{\rm{RCM}}_{d,\alpha,q}$ be the random cluster model on $\zd$ with
parameters $\alpha$ and $q$ which is a probability measure on $\{0,1\}^E$, where $E$ are the edges in $\zd$,
obtained by taking a limit of the random cluster models on finite boxes as defined in 
Subsection~\ref{ss.Examples}. We then 
think of $\nu^{\rm{RCM}}_{d,\alpha,q}$ as an RER on $\zd$ by considering the induced connected components.
(For the experts, using one of the possible definitions of a random cluster model,
there might be more than one such measure on $\zd$; nonetheless, our definition of
$\nu^{\rm{RCM}}_{d,\alpha,q}$ above is well-defined as this limit exists.) Recall that $q=1$ 
corresponds to the classical divide and color model.

\begin{prop}\label{p.RCM}
(i). For all $d\ge 2$, $q\ge 1$ and $\alpha> 0$, $d(\nu^{\rm{RCM}}_{d,\alpha,q})< 1$. \\
(ii). (\cite{BBT1}) For all $d\ge 2$, $\alpha> 0$ and $p>0$, $d(\nu^{\rm{RCM}}_{d,\alpha,1},p)>0$.
\end{prop}

{\bf Proof.} (i). It is easy to show that for all $d\ge 2$, $q\ge 1$ and $\alpha> 0$, there exists $C,\gamma >0$ so that for all $n$,
$\nu^{\rm{RCM}}_{d,\alpha,q}(|\pi(0)|\ge n)\ge C\gamma^n$. It follows from Proposition~\ref{p.statconncorr}(iv)
that $d(\nu^{\rm{RCM}}_{d,\alpha,q})< 1$. \\
(ii). This is stated in Theorem 3.1 in \cite{BBT1}.
\qed

\section{Ergodic results in the translation invariant case}\label{s.transfer}

In this section, the main theme is to investigate the ergodic theoretic properties of 
our color processes in the translation invariant case.
These will turn out to depend both on the ergodic behavior of the 
RERs generating the color process as well as on the structure of the clusters which arise.
We therefore assume in this section that $V={\mathbb Z}^d$ and we only consider RERs in
$\rm{RER}^{\rm{stat}}_V$.

We will refer to \cite{EW} and \cite{W} for the standard definitions in ergodic theory 
and will not, in view of space, recall these definitions here. The ergodic concepts which
we will consider are (1) ergodicity, (2) weak-mixing,
(3) mixing, (4) $k$-mixing, (5) $K$-automorphism and (6) Bernoullicity.
Importantly, in \cite{EW},
these definitions are also stated for ${\mathbb Z}^d$. 
In addition, we will assume familiarity with the notion of the entropy of a dynamical system or 
a stationary process. We recall that one stationary process is a {\it factor}
of another stationary process if the former can be expressed as a translation invariant 
function of the latter. All the standard ergodic properties (in particular all those 
considered in this paper) are easily shown (or known to be) preserved by factor maps.
In addition, it is known that i.i.d.\ processes satisfy all of the ergodic
properties that we study and that, in addition, if we have a stationary process $\mu$
satisfying one of our ergodic properties, then the joint stationary process
where (1) the first marginal is $\mu$, (2) the second marginal is an i.i.d.\ process
and (3) the two processes are independent also satisfies this given ergodic property.

In what follows, $(\pi,X^{\nu,p})$ is our {\it joint} RER and color process where 
$\pi$ is the random partition with distribution $\nu$ and $X^{\nu,p}$ is the corresponding
color process with parameter $p$; the latter of course has distribution $\Phi_p(\nu)$. 
The distribution of the joint law will be denoted by $\bfp=\jointlaw$.
With $d$ specified, we let $B_n:=[-n,n]^d\cap{\mathbb Z}^d$
(so that $|B_n|=(2n+1)^d$). For a subset $A\subseteq{\mathbb Z}^d$ and $x\in {\mathbb Z}^d$, 
define the translation of $A$ by $x$ by $T^x A:=\{y\,:\,y-x\in A\}$ and
for subsets $B\subseteq\{0,1\}^{{\mathbb Z}^d}$ and $x\in {\mathbb Z}^d$, 
$T^x B$ will also have the obvious meaning.

\subsection{Positive density clusters imply nonergodicity of the color process}

Essentially following Burton-Keane~\cite{BK89}, we first
make the following definition.

\begin{df}\label{d.clustdens}
We say that a subset $S$ of ${\mathbb Z}^d$ has density $\alpha$ if
$$
\lim_{i\to\infty} \frac{|S\cap B_i|}{|B_i|}=\alpha.
$$ 
We say that $S$ has \emph{upper density} $\alpha$ if 
$$
\varlimsup_{i\to \infty} \frac{|S\cap B_i|}{|B_i|}=\alpha.
$$
\end{df}

The proof of Theorem 1 in \cite{BK89} easily yields the following result.

\begin{thm}\label{t.bkrer}
Suppose that $\nu\in \rer_{\zd}^{\stat}$. Then 
\begin{equation}
\nu(\mbox{every }\phi\in \pi\mbox{ has  a density})=1.
\end{equation}
\end{thm}

The main result of this subsection is the following result.

\begin{thm}\label{t.ergod1}
Fix $d\ge 1$, $p\in (0,1)$ and suppose that $\nu\in \rer_{\zd}^{\stat}$. If 
$$
\nu(\exists \phi\in \pi\,:\,\phi\mbox{ has positive density})>0,
$$ 
then $\munup$ is not ergodic. In particular, if under $\nu$ there are a positive finite number 
of infinite clusters with positive probability, then $\munup$ is not ergodic. 
\end{thm}

To prove this, we begin with the following lemma.

\begin{lma}\label{l.ergthm1}
Suppose $\nu\in \rer_{\zd}^{\stat}$ and that 
$$
\nu(\exists \phi\in \pi\,:\,\phi\mbox{ has positive density})>0.
$$ 
Then there exists a set $S\subseteq {\mathbb Z}^d$ of positive upper density and a number 
$\delta=\delta_ S>0$ such that $$\nu(\pi(0)=\pi(x))\ge \delta\mbox{ for all }x\in S.$$
\end{lma}

\pf We proceed by contradiction. Assume that there does not exist a set $S$ with positive upper density and a 
$\delta>0$ such that  $\nu(\pi(0)=\pi(x))\ge \delta\mbox{ for all }x\in S.$ 
Let $\epsilon>0$ be arbitrary and let $$S_{\epsilon}=\{x\,:\,\nu(\pi(0)=\pi(x))\ge\epsilon\}.$$ Our 
assumptions imply that $S_\epsilon$ has upper density $0$. We now get that 
\begin{eqnarray*}
\lefteqn{\varlimsup_{n\to \infty}\frac{E_{\nu}[|\pi(0)\cap B_n|]}{(2n+1)^d}  = \varlimsup_{n\to\infty}\sum_{x\in B_n} \frac{\nu(\pi(0)=\pi(x))}{(2n+1)^d} }\\  & & =\varlimsup_{n\to\infty} \sum_{x\in B_n\cap S_{\epsilon}} \frac{\nu(\pi(0)=\pi(x))}{(2n+1)^d} + \varlimsup_{n\to\infty} \sum_{x\in B_n\cap S_{\epsilon}^c} \frac{\nu(\pi(0)=\pi(x))}{(2n+1)^d}\\ & & \le \varlimsup_{n\to\infty} \sum_{x\in B_n\cap S_{\epsilon}} \frac{1}{(2n+1)^d} + \varlimsup_{n\to\infty} \sum_{x\in B_n\cap S_{\epsilon}^c} \frac{\epsilon}{(2n+1)^d}\\ & & \le 0 +\epsilon=\epsilon,
\end{eqnarray*}
using  that $S_{\epsilon}$ has upper density $0$ in the last inequality. Since $\epsilon>0$ was arbitrary, it follows that

\begin{equation}\label{e.bk1}
\lim_{n\to \infty}\frac{E_{\nu}[|\pi(0)\cap B_n|]}{(2n+1)^d} =0.
\end{equation}

On the other hand, by Theorem~\ref{t.bkrer},

\begin{equation}\label{e.bk2}
\lim_{n\to\infty} \frac{|\pi(0)\cap B_n|}{(2n+1)^d}=L\mbox{ a.s.},
\end{equation}
for some random variable $L$. The assumption 
$\nu(\exists \phi\in \pi\,:\,\phi\mbox{ has positive density})>0$ implies that $P_\nu(L>0)>0$, 
so that $E_{\nu}[L]>0$.  Hence, using~\eqref{e.bk2} and the bounded convergence theorem,

\begin{equation}\label{e.bk3}
\lim_{n\to\infty}\frac{E_{\nu}[|\pi(0)\cap B_n|]}{(2n+1)^d}=E_{\nu}[L]>0.
\end{equation}

However~\eqref{e.bk3} contradicts~\eqref{e.bk1}, finishing the proof.\qed

{\bf Proof of Theorem~\ref{t.ergod1}.} If $\munup$ is ergodic, then
\begin{equation}\label{e.bk4}
\lim_{n\to\infty} \sum_{x\in B_n}\frac{\munup(X(0)=X(x)=1)}{(2n+1)^d}=p^2.
\end{equation}
From Lemma~\ref{l.ergthm1}, it follows that there is a deterministic set $S\subseteq \zd$ of 
positive upper density and a $\delta=\delta_S>0$ such that 
\begin{equation}\label{e.bk5}
\nu(\pi(0)=\pi(x))\ge \delta \mbox{ for }x\in S.
\end{equation}
Hence,
\begin{equation}\label{e.bk6}
\munup(X(0)=X(x)=1)\ge \delta p + (1-\delta) p^2\mbox{ for }x\in S.
\end{equation}
However, we also know that 
\begin{equation}\label{e.bk7}
\munup(X(0)=X(x)=1)\ge p^2\mbox{ for all }x.
\end{equation}

Equations~\eqref{e.bk6}, ~\eqref{e.bk7} and the fact that $S$ has positive upper density imply that
\begin{equation}\label{e.bk8}
\varlimsup_{n} \sum_{x\in B_n} \frac{\munup(X(0)=X(x)=1)}{(2n+1)^d}>p^2,
\end{equation} 
which implies that $\munup$ is not ergodic due to~\eqref{e.bk4}.
The final statement follows easily from the ergodic theorem.\qed

\begin{remark}
We will see later (Theorem~\ref{t.ergod3}), that the converse to 
Theorem~\ref{t.ergod1} holds when $\nu$ is ergodic. 
\end{remark}

\subsection{When does the color process inherit ergodic properties from the RER?}

The first  theorem in this subsection tells us that, when all clusters are finite, then
any ergodic property of $\nu$ is automatically passed on to $\jointlaw$ and hence to 
$\munup$. This is really just an extension (with the same proof) of Theorem 3.1 
in \cite{JS} where this 
was proved for the particular property of being Bernoulli. Nonetheless, since the proof is short, 
we include it for completeness. We mention that Bernoulliness for the $\ttinv$-process 
(and consequently for random walk in random scenery) in the transient case (which is a special 
case of having finite clusters) was proved earlier by a different method in \cite{dHS97}.

\begin{thm}\label{t.finiteclust}
Fix $d\ge 1$, $p\in (0,1)$ and $\nu\in \rer_{\zd}^{\stat}$. Assume that $\nu$ satisfies 
$$
\nu(\forall \phi\in \pi\,:\,\phi\mbox{ is finite})=1.
$$
Then, letting $\mu_p$ denote product measure with density $p$ on ${\zd}$, we have that
$\jointlaw$ is a factor of $\nu\times \mu_p$. In particular, if
${\mathfrak p}$ denotes any one of the ergodic properties being studied here, then 
$\nu$ has property ${\mathfrak p}$ if and only if $\jointlaw$ has property 
${\mathfrak p}$. 
In particular, if $\nu$ has property ${\mathfrak p}$,
then $\munup$ has property ${\mathfrak p}$.
\end{thm}

\pf Concerning the middle statement, first, since $\pi$ is a factor of $(\pi,\xpip)$ and all of 
these properties are preserved under factors, the ``if'' direction follows. Secondly, for the 
``only if'' direction, we observe that if $\nu$ has property ${\mathfrak p}$, then so does 
$\nu\times \mu_p$ and hence $\jointlaw$ in turn has this property being, as claimed, a factor of the 
latter. Since $\munup$ is a factor of $\jointlaw$, the final statement is immediate.

For the first and main statement, let $Y=(Y(z))_{z\in {\mathbb Z}^d}$ be an i.i.d.\ 
field with $P(Y(z)=1)=p=1-P(Y(z)=0)$ 
for $z\in {\mathbb Z}^d$, and let $Z=(\pi,Y)$. We will now obtain  $(\pi,\xpip)$ as a factor of $Z$. 
For the first marginal, we just copy the first marginal of $Z$. For the second marginal, we proceed as 
follows. Choose an arbitrary lexicographic ordering of ${\mathbb Z}^d$. For $x\in {\mathbb Z}^d$, let 
$y_x$ be that element $y$ of $\pi(x)$ which minimizes $y-x$ with respect to the above ordering. Finally, 
we let $\xpip(x)=Z(y_x)$. It is easy to see that this yields the desired factor map. \qed

\medskip\noindent
Theorems~\ref{t.ergod1} and~\ref{t.finiteclust} suggest to us
that the interesting case is when $\pi$ contains no equivalence class of 
positive density but contains some infinite equivalence class, necessarily of 
0 density. Theorems~\ref{t.ergod3},~\ref{t.ergod4},~\ref{t.numixmumix} 
and~\ref{t.numixmumix-kmixing}
below cover this case for some ergodic properties.

\begin{thm}\label{t.ergod3}
Fix $d\ge 1$ and $p\in (0,1)$ and assume that $\nu\in \rer_{\zd}^{\stat}$ satisfies 
$$
\nu(\exists \phi\in \pi\,:\,\phi\mbox{ has positive density})=0. 
$$
Then $\nu$ is ergodic if and only if $\jointlaw$ is ergodic. 
In particular, if $\nu$ is ergodic, then $\munup$ is ergodic.

\end{thm}

\pf First assume that $(\pi,\xpip)$ is ergodic. Since $\pi$ is a factor of $(\pi,\xpip)$ and ergodicity 
is preserved under factors, the if part of the theorem follows. Similarly, we obtain the last statement 
of the theorem from the first statement since $\xpip$ is a factor of $(\pi,\xpip)$.

We move on to the only if part of the theorem. Assume that $\nu$ is ergodic. Suppose that $K_1$ and $K_2$ are 
finite subsets of ${\zd}$. For $i=1,2$, suppose that $E_i$ is an event depending only on the 
color process $\xpip$ restricted to $K_i$, that is $E_i\in \sigma(\xpip(z)\,:\,z\in K_i)$. 
For $i=1,2$, fix $\phi_i\in\Delta_{K_i}$ and let $F_i= \{\pi_{K_i}=\phi_i\}$. 
By standard approximation by cylinder sets, it suffices to show that

\begin{equation}\label{e.nts}
\lim_{n\to \infty} \frac{1}{(2n+1)^d}\sum_{x\in B_n} \Pr(E_1, F_1, T^x E_2, T^x F_2) =\Pr(E_1,F_1)\Pr(E_2,F_2).
\end{equation}

For $x\in {\mathbb Z}^d$, let $C_x=C_x(K_1,K_2)$ be the event that there is some $z_1\in K_1$ and some 
$z_2\in K_2$ such that $\pi(z_1)=\pi(T^x(z_2))$. 

We will be done if we show that 

\begin{equation}\label{e.nts2}
\lim_{n\to \infty} \frac{1}{(2n+1)^d}\sum_{x\in B_n} \Pr(C_x^c,E_1, F_1, T^x E_2, T^x F_2) =\Pr(E_1,F_1)\Pr(E_2,F_2)
\end{equation}

and

\begin{equation}\label{e.nts3}
\lim_{n\to \infty} \frac{1}{(2n+1)^d}\sum_{x\in B_n} \Pr(C_x,E_1, F_1, T^x E_2, T^x F_2) =0.
\end{equation}

We start with~\eqref{e.nts3}. Clearly, it suffices to show

\begin{eqnarray}\label{e.densclust}
 \lim_{n\to \infty} \frac{1}{(2n+1)^d}\sum_{x\in B_n} \Pr( C_x)=0.
\end{eqnarray}

We get that

\begin{eqnarray}\label{e.innersum}
\lefteqn{\sum_{x\in B_n} \frac{\Pr( C_x)}{(2n+1)^d}}\nonumber\\ & & \le \sum_{x\in B_n} 
\sum_{z_1\in K_1,z_2\in K_2} \frac{\Pr(\pi(z_1)=\pi(T^x z_2))}{(2n+1)^d}\nonumber\\ & & = \sum_{z_1\in K_1,z_2\in K_2}  \sum_{x\in B_n} \frac{\Pr(\pi(z_1)=\pi(T^x z_2))}{(2n+1)^d}.
\end{eqnarray}
For fixed $z_1$ and $z_2$, the inner sum in~\eqref{e.innersum} converges to $0$  as $n\to\infty$ since every cluster of $\pi$ has density $0$ a.s.\, Hence,~\eqref{e.densclust} follows, and so~\eqref{e.nts3} is established.

We now move on to prove~\eqref{e.nts2}. We write

\begin{eqnarray}\label{e.condind}
\lefteqn{\Pr(C_x^c,E_1, F_1, T^x E_2, T^x F_2) } \\ & & =\Pr(C_x^c,F_1,T^x F_2) \Pr(E_1,T^x E_2|C_x^c,F_1,T^x F_2)\nonumber\\ & & =\Pr(C_x^c,F_1,T^x F_2)\Pr(E_1|F_1)\Pr(T^x E_2| T^x F_2),\nonumber\\ & & =\Pr(C_x^c,F_1,T^x F_2)\Pr(E_1|F_1)\Pr( E_2| F_2)\nonumber
\end{eqnarray}

where in the second equality we used the fact that $E_1$ and $T^x E_2$ are conditionally independent given 
the event $\{C_x^c,F_1,T^x F_2\}$, and translation invariance was used in the last equality.

Next, we argue that for $\phi_1,\phi_2$ fixed,

\begin{equation}\label{e.nts4}
\lim_{n\to \infty} \frac{1}{(2n+1)^d} \sum_{x\in B_n } \Pr(C_x^c,F_1,T^x F_2)= \Pr(F_1)\Pr(F_2).
\end{equation}
To see this, we observe that 

\begin{equation}\label{e.ineq1}
\Pr(C_x^c,F_1,T^x F_2)\le \Pr(F_1,T^x F_2) 
\end{equation}

and 

\begin{equation}\label{e.ineq2}
\Pr(C_x^c,F_1,T^x F_2)\ge \Pr(F_1,T^x F_2)-\Pr(C_x). 
\end{equation}

By ergodicity of $\nu$, 

\begin{equation}\label{e.nts5}
\lim_{n\to \infty} \frac{1}{(2n+1)^d} \sum_{x\in B_n }\Pr(F_1,T^x F_2)= \Pr(F_1)\Pr(F_2).
\end{equation}

Next, we already proved in~\eqref{e.densclust} that

\begin{equation}\label{e.nts6}
\lim_{n\to \infty} \frac{1}{(2n+1)^d} \sum_{x\in B_n }\Pr(C_x)=0.
\end{equation}

Hence, Equation~\eqref{e.nts4} follows from~\eqref{e.ineq1},~\eqref{e.ineq2},~\eqref{e.nts5} and~\eqref{e.nts6}.
We are now ready to obtain~\eqref{e.nts2} from~\eqref{e.condind} and~\eqref{e.nts4}. We get

\begin{eqnarray*}
\lefteqn{\lim_{n\to \infty} \frac{1}{(2n+1)^d}\sum_{x\in B_n} \Pr(C_x^c,E_1, F_1, T^x E_2, T^x F_2) }\\ & & \stackrel{~\eqref{e.condind}}{=}\lim_{n\to \infty} \frac{1}{(2n+1)^d}\sum_{x\in B_n} \Pr(C_x^c,F_1,T^x F_2)\Pr(E_1|F_1)\Pr( E_2| F_2) \\ & & \stackrel{~\eqref{e.nts4}}{=} \Pr(F_1)\Pr(F_2)\Pr(E_1|F_1)\Pr( E_2| F_2) =\Pr(E_1,F_1)\Pr(E_2,F_2).
\end{eqnarray*}

Hence,~\eqref{e.nts2} is established. Since $K_1$ and $K_2$ are arbitrary finite sets, ergodicity of 
$\munup$ follows. \qed

\begin{thm}\label{t.ergod4}
Fix $d\ge 1$ and $p\in (0,1)$ and assume that $\nu\in \rer_{\zd}^{\stat}$ satisfies 
$$
\nu(\exists \phi\in \pi\,:\,\phi\mbox{ has positive density})=0. 
$$
Then $\nu$ is weakly mixing if and only if $\jointlaw$ is weakly mixing. In particular, if $\nu$ is 
weakly mixing, then $\munup$ is weakly mixing.
\end{thm}

\pf First assume that $(\pi,\xpip)$ is weakly mixing. Since $\pi$ is a factor of $(\pi,\xpip)$ and weak 
mixing is preserved under factors, the if part of the theorem follows. Similarly, we obtain the last 
statement of the theorem from the first statement since $\xpip$ is a factor of $(\pi,\xpip)$.

We move on to the only if part of the theorem. Assume that $\nu$ is weak mixing. Suppose that $K_1$ and $K_2$ 
are finite subsets of $\zd$. For $i=1,2$, suppose that $E_i$ is an event depending only on the color process 
$\xpip$ restricted to $K_i$, that is $E_i\in \sigma(\xpip(z)\,:\,z\in K_i)$. For $i=1,2$, 
fix $\phi_i\in\Delta_{K_i}$ and let $F_i= \{\pi_{K_i}=\phi_i\}$. By approximation by cylinder sets, it 
suffices to show that

\begin{equation}
\lim_{n\to\infty} \frac{\sum_{x\in B_n} | \Pr(E_1, F_1, T^x E_2, T^x F_2) -\Pr(E_1,F_1)\Pr(E_2,F_2) |}{(2n+1)^d}=0.
\end{equation}

Define $C_x$ in the same way as in the proof of Theorem~\ref{t.ergod3}. By the triangle inequality, 
we have for each
$n$

\begin{eqnarray}\label{e.twoterms}
\lefteqn{\frac{\sum_{x\in B_n} | \Pr(E_1, F_1, T^x E_2, T^x F_2) -\Pr(E_1,F_1)\Pr(E_2,F_2) |}{(2n+1)^d} } \\ & &\le \frac{\sum_{x\in B_n}\Pr(C_x)}{(2n+1)^d}+ \frac{\sum_{x\in B_n} | \Pr(C_x^c,E_1, F_1, T^x E_2, T^x F_2) -\Pr(E_1,F_1)\Pr(E_2,F_2)|}{(2n+1)^d}.\nonumber
\end{eqnarray}

The first term in the last line of~\eqref{e.twoterms} converges to $0$ as $n\to \infty$ due to 
Equation~\eqref{e.densclust} above, so we can focus on the second term. We get that

\begin{eqnarray}\label{e.prodrewrite}
\lefteqn{\frac{\sum_{x\in B_n} | \Pr(C_x^c,E_1, F_1, T^x E_2, T^x F_2) -\Pr(E_1,F_1)\Pr(E_2,F_2)|}{(2n+1)^d}}\\ & & =\frac{\sum_{x\in B_n} | \Pr(C_x^c,F_1, T^x F_2)\Pr(E_1|F_1 )\Pr(E_2|F_2 ) -\Pr(E_1|F_1)\Pr(E_2|F_2)\Pr(F_1)\Pr(F_2)|}{(2n+1)^d}  \nonumber\\ & & \le \frac{\sum_{x\in B_n} | \Pr(C_x^c,F_1, T^x F_2) -\Pr(F_1)\Pr(F_2)|}{(2n+1)^d}\nonumber\\ & &\le\frac{\sum_{x\in B_n}\Pr(C_x)}{(2n+1)^d}+  \frac{\sum_{x\in B_n} | \Pr(F_1, T^x F_2) -\Pr(F_1)\Pr(F_2)|}{(2n+1)^d}\to 0,\nonumber
\end{eqnarray}
as $n\to \infty$ due to the weak mixing of $\nu$ and the comment above Equation~\eqref{e.prodrewrite}.
Since $K_1$ and $K_2$ are arbitrary finite sets, this finishes the proof.\qed

\begin{thm}\label{t.numixmumix}
Fix $d\ge 1$ and $p\in (0,1)$ and assume that $\nu\in \rer_{\zd}^{\stat}$ satisfies 
$$
\lim_{x\to\infty}\nu(\pi(x)=\pi({\bf 0}))= 0.
$$
Then $\nu$ is mixing if and only if $\jointlaw$ is  mixing. In particular, if $\nu$ is mixing,
then $\munup$ is mixing.
\end{thm}

\begin{remark}
It is elementary to check that the condition that $\nu(\pi(x)=\pi({\bf 0}))\to 0$ as $x\to\infty$ 
is necessary for mixing, since if this fails, pairwise correlations in the color process do 
not converge to $0$ and hence mixing does not hold.
\end{remark}

\begin{remark}
Clearly the condition $\nu(\pi(x)=\pi(y))\to 0$ as $|x-y|\to\infty$ implies 
the condition $\nu(\exists \phi\in \pi\,:\,\phi\mbox{ has positive density})=0$. To see that 
the converse does not hold, consider the following deterministic example in ${\mathbb Z}^2$. 
Let $\pi$ be the partition into horizontal lines. Then clearly each cluster has density $0$, 
but $\nu(\pi(x)=\pi(y))=1$ if $x_2=y_2$. Obviously, a similar example exists for any $d\ge 2$.
\end{remark}

\pf First assume that $(\pi,\xpip)$ is mixing. Since $\pi$ is a factor of $(\pi,\xpip)$ and mixing is 
preserved under factors, the if part of the theorem follows. Similarly, we obtain the last statement of 
the theorem from the first statement since $\xpip$ is a factor of $(\pi,\xpip)$.

We move on to the only if part of the theorem. Assume that $\nu$ is mixing. Suppose that $K_1$ and 
$K_2$ are finite subsets of $V$. For $i=1,2$, suppose that $E_i$ is an event depending only on the 
color process $\xpip$ restricted to $K_i$, that is $E_i\in \sigma(\xpip(z)\,:\,z\in K_i)$. For 
$i=1,2$, fix $\phi_i\in\Delta_{K_i}$ and let $F_i= \{\pi_{K_i}=\phi_i\}$. By approximation by cylinder 
sets, it suffices to show that

\begin{equation}
\lim_{|x|\to\infty} \Pr(E_1, F_1, T^x E_2, T^x F_2) =\Pr(E_1,F_1)\Pr(E_2,F_2).
\end{equation}

For $x\in {\mathbb Z}^d$, let $C_x$ be the event as defined in the proof of Theorem~\ref{t.ergod3}. 

Then
\begin{eqnarray}\label{e.firstsum}
\lefteqn{\Pr(E_1, F_1, T^x E_2, T^x F_2)} \\ & & = \Pr(C_x,E_1, F_1, T^x E_2, T^x F_2)+\Pr(C_x^c,E_1, F_1, T^x E_2, T^x F_2).\nonumber
\end{eqnarray}

Since $K_1$ and $K_2$ are finite, the property that $\Pr(\pi(x)=\pi(y))\to 0$ as $|x-y|\to\infty$ 
implies that $\Pr(C_x)\to 0$ as $|x|\to \infty$. Hence, the first term in the right hand side 
of~\eqref{e.firstsum} converges to $0$ as $|x|\to\infty$, and we can focus on the second term.

Observe that as in the proof of Theorem~\ref{t.ergod3},
\begin{eqnarray}\label{e.spliteq1}
\Pr(C_x^c,E_1, F_1, T^x E_2, T^x F_2)  =\Pr(C_x^c,F_1,T^x F_2)\Pr(E_1|F_1)\Pr( E_2| F_2).
\end{eqnarray}

Using the mixing property of $\nu$ and the fact that $\Pr(C_x^c)\to 1$ as $|x|\to \infty$, we get
\begin{eqnarray}\label{e.spliteq3}
\lefteqn{\lim_{|x|\to \infty} \Pr(C_x^c,F_1,T^x F_2)\Pr(E_1|F_1)\Pr( E_2| F_2)}\\ & & =\Pr(F_1)\Pr( F_2)\Pr(E_1|F_1)\Pr( E_2| F_2)=\Pr(E_1,F_1)\Pr(E_2,F_2).\nonumber
\end{eqnarray}

Since $K_1$ and $K_2$ are arbitrary finite sets, this establishes the mixing property of $\munup$ and 
the proof is finished. \qed

The following theorem also holds. Its proof is a straightforward modification of the proof of 
Theorem~\ref{t.numixmumix} and hence is left to the reader. In addition, also here the condition 
$\nu(\pi(x)=\pi(y))\to 0$ as $|x-y|\to \infty$ clearly cannot be weakened.

\begin{thm}\label{t.numixmumix-kmixing}
Fix $d\ge 1$ and $p\in (0,1)$ and assume that $\nu\in \rer_{\zd}^{\stat}$ satisfies 
$$
\lim_{x\to\infty}\nu(\pi(x)=\pi({\bf 0}))= 0.
$$
Then $\nu$ is $k$-mixing if and only if $\jointlaw$ is $k$-mixing. 
In particular, if $\nu$ is $k$-mixing, then $\munup$ is $k$-mixing.
\end{thm}

Theorem~\ref{t.finiteclust} tells us that when all the clusters are finite,
all ergodic properties of $\nu$ are passed to $\jointlaw$ and
Theorems~\ref{t.ergod3},~\ref{t.ergod4},~\ref{t.numixmumix} and~\ref{t.numixmumix-kmixing} 
tell us that four specific ergodic properties are passed from $\nu$ to $\jointlaw$ under the 
weaker assumption (and even under weaker assumptions for two of these) that
$$
\lim_{x\to\infty}\nu(\pi(x)=\pi({\bf 0}))= 0.
$$
However, it turns out interestingly that the important property of being 
Bernoulli is not necessarily passed from $\nu$ to $\jointlaw$ under this latter 
assumption. We call the following a theorem although it is actually just an observation 
based on Kalikow's famous work (see \cite{Kal}) on the $\ttinv$-process.

\begin{thm}\label{t.kalikow.example}
There exists $\nu\in \rer_{{\mathbb Z}}^{\stat}$ which is Bernoulli satisfying
$$
\lim_{x\to\infty}\nu(\pi(x)=\pi({\bf 0}))= 0
$$
but for which ${\mathbf P}_{\nu,1/2}$ is not Bernoulli and even for which
$\Phi_{1/2}(\nu)$ is not Bernoulli.
\end{thm}

\pf Let $(X(i))_{i\in {\mathbb Z}}$ be an i.i.d.\ sequence such that 
$P(X(i)=1)=P(X(i)=-1)=1/2$.
Let $\nu \in \rer_{{\mathbb Z}}^{\stat}$ be the distribution of the RER given by $j < k$ are 
put in the same cluster if $\sum_{i=j+1}^k X(i)=0$. 
(This is of course just our RER for random walk in random scenery from Subsection~\ref{ss.Examples}.)
Being a factor of an i.i.d.\ process, 
$\nu$ is Bernoulli and one easily has $\lim_{x\to\infty}\nu(\pi(x)=\pi({\bf 0}))= 0$. The 
fact however that ${\mathbf P}_{\nu,1/2}$ is not Bernoulli is Kalikow's famous theorem 
(\cite{Kal}). The stronger fact that even $\Phi_{1/2}(\nu)$ is not Bernoulli was proved 
by Hoffman (\cite{Hoff}). One should however stress that the latter proof relies on 
Kalikow's theorem.\qed

\subsection{Can the color process enjoy more ergodic properties than the RER?}

While $\jointlaw$ cannot of course exhibit stronger ergodic behavior than 
$\nu$ itself (since the latter is a factor of the former), $\munup$ could 
possibly exhibit stronger ergodic behavior than $\nu$.

Our first example shows that $\Phi_{1/2}$ is not injective on 
$\rer_{{\mathbb Z}}^{\stat}$ and as a consequence gives us a nonergodic
RER whose color process is ergodic.

\begin{prop}\label{p.differinf}
There exist $\nu_3,\nu_4\in \rer_{{\mathbb Z}}^{\stat}$ with
$\nu_3\neq \nu_4$, $\Phi_{1/2}(\nu_3)=\Phi_{1/2}(\nu_4)$ and such that the latter process is ergodic.
It follows that there is a nonergodic RER whose color process is ergodic.
\end{prop}

\pf Construct $\nu_3$ as follows: on each subset of the type 
$\{i,i+1,i+2\}$ for $i$ divisible by $3$, independently use
the RER $\nu_1$ from the proof of Theorem~\ref{t.bigfinitetheorem}(A)
on these 3 points. Next,
shift the configuration, uniformly at random by $0$, $1$ or $2$ steps to the 
right to construct a stationary RER. Next construct $\nu_4$ in the same way, 
using $\nu_2$ from the proof of Theorem~\ref{t.bigfinitetheorem}(A).
Since $\nu_1$ and $\nu_2$ yield the same color processes in the setting with 
three elements, it follows easily that $\Phi_{1/2}(\nu_3)=\Phi_{1/2}(\nu_4)$.
Ergodicity (but not mixing) of the latter is easily established. 
The final claim is established by considering any nontrivial
convex combination of $\nu_3$ and $\nu_4$.\qed

\begin{remark}\label{r.nonerg}
(1) Using Theorem~\ref{t.bigfinitetheorem}(E), one can even, in the same way, find
$\nu_3,\nu_4\in \rer_{{\mathbb Z}}^{\stat}$ with
$\nu_3\neq \nu_4$ such that $\Phi_{p}(\nu_3)=\Phi_{p}(\nu_4)$ for all $p$. \\
(2). The above also shows that ergodicity of the color process may depend on $p$.
If we take $\nu_3$ and $\nu_4$ from the above proof and take any $p\neq 0,1/2,1$, then 
$\Phi_{p}(\nu_3)\neq \Phi_{p}(\nu_4)$ (since now $\nu_1$ and $\nu_2$
yield different color processes for such $p$ by Theorem~\ref{t.bigfinitetheorem}(C))
and hence the image of any
nontrivial convex combination of $\nu_3$ and $\nu_4$ is nonergodic, being
a nontrivial convex combination of the respective color processes.
\end{remark}

One can strengthen Proposition~\ref{p.differinf}, obtaining examples where the
color process is Bernoulli.

\begin{prop}\label{p.nonergtomix}
There exist a non-ergodic $\nu\in \rer_{{\mathbb Z}}^{\stat}$ such that 
$\Phi_{1/2}(\nu)$ is Bernoulli.
\end{prop}

\pf We will only sketch the proof.
Define $\nu_3\in\rer_{{\mathbb Z}}^{\stat}$ as follows. Let $(Z(i))_{i\in {\mathbb Z}}$ be an i.i.d.\ 
sequence such that $P(Z(i)=1)=P(Z(i)=0)=1/2$. Call all vertices $i$ with $Z(i)=0$ white, and all 
vertices $i$ with $Z(i)=1$ blue. Replace each blue vertex with three green vertices. Let each white 
vertex be its own equivalence class. The green vertices come in blocks of length divisible by $3$. 
Partition the green blocks independently using $\nu_1$ as in 
Proposition~\ref{p.differinf} yielding what we also call here $\nu_3$.
Define $\nu_4$ in the same way as $\nu_3$ but using 
$\nu_2$ from Proposition~\ref{p.differinf} instead of $\nu_1$. Again 
$\Phi_{1/2}(\nu_3)= \Phi_{1/2}(\nu_4)$
but now it is easily seen that the latter process is Bernoulli. 
Now take a nontrivial convex combination of $\nu_3$ and $\nu_4$ as above.\qed

\begin{remark}
Again, $\Phi_{p}(\nu_3)\neq \Phi_{p}(\nu_4)$ for any $p\neq 0,1/2,1$ and so we
see that a color process can change from being Bernoulli to being nonergodic as $p$ varies.
\end{remark}

We should confess at this point, although we felt it important to point out the above results to the reader,
we do feel at the same time that using nonergodic RERs in this context
is a little bit of a cheat.

We give another result which gives some restriction on the ergodic behavior of the color process in terms
of a restriction on the RER.

\begin{prop}\label{p.0entropy,notK}
If $\nu\in  \rer_{\zd}^{\stat}$ has 0 entropy and is not the RER which 
assigns probability 1 to the ``all singletons'' partition, then for any $p\in (0,1)$, 
$\Phi_{p}(\nu)$ is not a $K$-automorphism.
\end{prop}

\pf Case 1. $\nu$ is deterministic; i.e., there exists $\pi\in \rm{Part}_{{\zd}}$ 
such that $\nu({\pi})=1$. In this case, since $\pi$, by assumption, is not the
``all singletons'' partition, there must exist $x\neq y$ so that $x$ and $y$ are in the 
same cluster with positive probability and hence with probability 1. By translation 
invariance, there are points arbitrarily far away which are in the same cluster with probability 1.
This clearly rules out even mixing.

Case 2. $\nu$ is nondeterministic. Considering the joint process 
$(\pi,X^{\nu,p})$, it is 
easy to see that if $\nu$ is nondeterministic, then the two processes $\pi$ and 
$X^{\nu,p}$ cannot be independent. However, it has been proved by H. Furstenberg
(see Theorem 18.16 in \cite{glasner}) that if one has a 0 entropy system and a 
$K$-automorphism, then the only stationary joint process (so-called {\sl joining}) for them 
is when they are independently coupled. (When two processes have this latter property, they are 
called {\sl disjoint}.) Therefore, since $\pi$ is assumed to have 0 
entropy, $X^{\nu,p}$ cannot be a $K$-automorphism. \qed

\subsection{Constructing color processes with various ergodic behavior}

The first observation in this subsection that we want to make is that we can find 
$\nu\in \rer_{{\mathbb Z}}^{\stat,\conn}$ which falls anywhere in the ergodic hierarchy
(e.g., weak-mixing but not mixing). This is an immediate consequence of the following
lemma and of course the fact that we can find stationary 0,1-valued processes anywhere
in the ergodic hierarchy.

\begin{lma}\label{l.zconstruct}
Given a stationary 0-1 valued process $\{X_n\}$ on ${\mathbb Z}$, there is 
$\nu\in \rer_{{\mathbb Z}}^{\stat,\conn}$ which is isomorphic to 
$\{X_n\}$; i.e., there is a translation invariant invertible measure preserving transformation
between them.
\end{lma}

\pf This is nothing other than what we considered in Section~\ref{s.conn}.
If $X_n=1$, then we place $n$ and $n+1$ in the same class and then we saturate this so that it is
an equivalence relation. (So, essentially, the clusters will correspond to intervals of 1's in 
$\{X_n\}$.) This map is clearly invertible, proving the lemma.\qed

\medskip

We first mention that constructing a color process which is ergodic but not weak-mixing is a triviality.
Let $\{X_n\}$ be the stationary 0-1 valued process on ${\mathbb Z}$ which goes back and forth between 0 and
1 and consider the associated $\nu\in \rer_{{\mathbb Z}}^{\stat,\conn}$ given in the proof of
Lemma~\ref{l.zconstruct}. It is immediate that for all $p\in (0,1)$, the associated color process is
ergodic but not weak-mixing. We next have the following proposition.

\begin{prop}\label{p.Chacon}
There exists $\nu\in \rer_{{\mathbb Z}}^{\stat,\conn}$ so that for all $p\in (0,1)$, 
the associated color process $\munup$ is weak-mixing but not mixing.

\end{prop}

\pf We start with a stationary 0-1 valued process $\{X_n\}$ on ${\mathbb Z}$ which is weak-mixing
but for which $\limsup_n \Pr(X_0=X_n=1)> \Pr(X_0=1)^2$ (and hence is not 
mixing). An example of such a process is the so-called Chacon example; see for example page 216 
in \cite{Petersen}. Next consider the associated $\nu\in \rer_{{\mathbb Z}}^{\stat,\conn}$ given in 
the proof of Lemma~\ref{l.zconstruct}. Clearly, $\lim_{x\to\infty}\nu(\pi(x)=\pi({\bf 0}))= 0$ 
and hence Theorem~\ref{t.ergod4} implies that $\munup$ (and in fact $\jointlaw$) is weak-mixing.
To show that $\munup$ is not mixing, consider the two events $A:=\{\xpip_0=\xpip_1\}$ 
and $B_n:=\{\xpip_n=\xpip_{n+1}\}$. An elementary computation left to the reader gives that 
$$
\limsup_n\Pr(A\cap B_n) >\Pr(A)^2
$$
which implies that $\munup$ is not mixing.\qed

\section{Questions and further directions}\label{s.ques}

In this final section, we list a number of questions and 
a number of directions which might be interesting to pursue.
The questions certainly might be of somewhat varying difficulty but all seem natural to us.

\begin{question}\label{q.gaussian}
Let $X=(X_1,\ldots,X_n)$ be a $n$-dimensional Gaussian random variable, where each $X_i$ is $N(0,1)$ 
and the pairwise correlations are given by $\sigma_{i,j}$. Assume $\sigma_{i,j}\ge 0$ for 
all $i,j$. Let $h\in (-\infty,\infty)$ and $Y^h=(Y_1^h,\ldots,Y_n^h)$ be, as earlier, given by $Y_i^h=1$ if $X_i\ge h$
and $Y_i^h=0$ if $X_i<h$.  When is $Y^h$ a color process? 
\end{question}
\begin{remark}
Note that if $n=3$ and $h=0$, then $Y^h$ is a color process by Lemma~\ref{l.symm}.
The next three questions are special cases of the above question.
\end{remark}

\begin{question}
Concerning the exchangeable Gaussian process described in Subsection~\ref{s.gaussian},
which nonzero thresholds yield color processes?
\end{question}

\begin{question}
Given $\rho \in [0,1]$, consider the Markov chain on $\Rf$ where if in state $x$, then the 
next state has distribution $\rho^{1/2}  x + (1-\rho)^{1/2}  Z$ where $Z$ is standard normal.
Clearly the stationary distribution is a standard normal and we consider the corresponding
stationary Markov Chain $(Z_i)_{i\in {\mathbb Z}}$.
Fix $h$ and define the process $Y=(Y_i)_{i\in {\mathbb Z}}$ where $Y_i=1$ if $Z_i\ge h$ and
$Y_i=0$ if $Z_i<h$. For which $\rho$ and $h$ is $Y$ a color process? 
\end{question}

\begin{question}
Consider a centered Gaussian free field $(Z(x))_{x\in {\mathbb Z}^d}$ with $d\ge 3$. 
Fix $h$ and consider the process $Y^h=(Y^h(x))_{x\in {\mathbb Z}^d}$ where 
$Y^h(x)=1$ if $Z(x)\ge h$ and $Y^h(x)=0$ if $Z(x)<h$. When is $Y$ a color process?
\end{question}

\begin{question}\label{q.ising}
On which graphs and for which values of the parameters $J\ge 0$ and $h>0$ is 
the Ising model a color process?
\end{question}
\begin{remark}
(i). Unlike in the case $h=0$, the marginal distributions of the Ising model with
$J\ge 0$ and $h>0$ need not be the same in which case it of course cannot be a color process;
this happens for example for a path of length 2. One might therefore restrict to
transitive graphs for this question. \\
(ii). In \cite{Alexander}, an asymmetric random cluster model is studied and it is shown
how one can obtain the Ising model with $J\ge 0$ and $h>0$ using this model. However,
this procedure does not correspond to a color process in our sense as it does in the case $h=0$.\\
(iii). Theorem~\ref{t.bigfinitetheorem}(B) and (D) in Section~\ref{s.finitecase} yield that
there is more than one RER generating the Ising model on $K_3$ (the complete graph on 3 vertices) 
when $J>0$ and $h=0$ while there is at most one RER generating the Ising model on $K_3$ 
when $J>0$ and $h>0$. Mathematica gives a (necessarily unique) solution for the latter RER
for positive $h$ which interestingly does not coverge, as $h\to 0$, to the RER corresponding to 
the random cluster model but rather converges to a different RER. One might conclude from this that
the random cluster RER is not the natural RER which yields the Ising model on $K_3$ with $J>0$ and $h=0$
since it cannot be perturbed to obtain the $h>0$ case.
\end{remark}

\begin{question}
For $p\neq 1/2$, determine those $\nu\in \rer_{\N}^{\exch}$ which are 
$(\rer_{\N}^{\exch},p)$-unique. Is it all of $\rer_{\N}^{\exch}$ (which is equivalent to
$\Phi_p$ being injective)?
\end{question}

\begin{question}
What are all the possible limiting distributions (after normalization) of
$$
\sum_{i\in B_n} \xpip(i)
$$
which one can obtain by varying $\nu$ and $p$?
\end{question}
\begin{remark}
It was shown in \cite{KS79} that one can obtain a large number of limiting distributions for the special 
case of random walk in random scenery. Also, it is known (see \cite{NW}) that if $\{X_n\}_{n\ge 0}$ is a 
stationary and positively associated process with $\sum_n {\rm Cov}(X_0,X_n)<\infty$, then one obtains a 
central limit theorem. This, together with~\eqref{e.nonneg.cor}, could be used to show that certain 
classes of color processes obey a central limit theorem. In addition, a central limit theorem and various 
other results concerning the original divide and color model are obtained in \cite{Garet}.
\end{remark}

\begin{question}
If an RER $\nu_1$ is finer than another RER $\nu_2$, in the sense that $\nu_1$ and $\nu_2$ can be 
coupled so that the clusters of $\nu_2$ are unions of clusters of $\nu_1$, does it follow that
$d(\nu_1,p)\ge d(\nu_2,p)$ for each $p$?
\end{question}
\begin{remark}
We note that for $d\ge 1$ and $J_1 <J_2$, the RER for the random cluster model with parameters $q=2$ and
$J_1$ is finer than the RER for $q=2$ and $J_2$ and in this case, Proposition 1.6 in \cite{LS06} states 
the asked for inequality above for the special case $p=1/2$, in which case the color process is just the Ising 
model. There is a minor additional point here. In the color process, even the infinite clusters are colored 
using $p=1/2$ while in Proposition 1.6 in \cite{LS06}, one was looking at the plus states for the Ising model
which is obtained by coloring the unique (if there is any) infinite cluster 1. However,
by Proposition 1.2 in \cite{LS06}, the set of product measures that one dominates is the same 
whether the infinite cluster is colored 1 (corresponding to the plus state) or colored $-1$ 
(corresponding to the minus state) and therefore also for the above color process which lies inbetween.
\end{remark}

\begin{question}
If an RER $\nu$ is such that $d(\nu)>0$, does it follow that $d(\nu,p)>0$ for all $p >0$?
\end{question}

\begin{question}
Let $\nu^{\rm{RCM}}_{d,\alpha,2}$ be the random cluster model on $\zd$ with $q=2$ and parameter $\alpha$.
One would perhaps expect that (1) $d(\nu^{\rm{RCM}}_{d,\alpha,2},p)$ is jointly continuous in $\alpha$ 
and $p$ and decreasing in $\alpha$ for fixed $p$, (2) $d(\nu^{\rm{RCM}}_{d,\alpha,2})$ is continuous in $\alpha$, 
(3) $\lim_{\alpha\to 0}d(\nu^{\rm{RCM}}_{d,\alpha,2})=1$ and 
(4) $\lim_{\alpha\to \infty}d(\nu^{\rm{RCM}}_{d,\alpha,2})=0$. 
Verify as much of this picture as possible. Does anything interesting happen near the critical value 
$\alpha_c(d)$?.
\end{question}

\begin{question}\label{q.bhs}
Consider a 1-dimensional random walk which moves to the right with probability $\frac{1}{2}+\sigma$
and to the left with probability $\frac{1}{2}-\sigma$ where $\sigma>0$. Let $\nu_\sigma$ be the 
associated RER on $\Z$ (whose color process is then random walk in random scenery). 
What results can one obtain concerning $d(\nu_\sigma,p)$ and $d(\nu_\sigma)$?
Is there some phase transition in the parameter $\sigma$?
\end{question}
\begin{remark}
In \cite{BHS}, a phase transition in $\sigma$ is shown for random walk in random scenery,
concerning Gibbsianness of the process. Is it possible that this could be related to a phase
transition concerning the stochastic domination behavior?
\end{remark}

\begin{question}
Provide natural examples of RERs for which all clusters are infinite and $d(\nu)>0$.
\end{question}

\begin{question}\label{q.voter.bern}
Are the stationary distributions for the voter model (which we have seen are color processes)
in $d\ge 3$ dimensions Bernoulli shifts?
\end{question}
\begin{remark}
If we look at the RER corresponding to coalescing random walks in $d\ge 4$ dimensions and 
we restrict the clusters down to a $d-3$ dimensional sublattice, then all the clusters become finite.
It follows from Theorem~\ref{t.finiteclust} that the restriction of the stationary distributions
for the voter model to this $d-3$ dimensional sublattice is a Bernoulli shift and the fact that the RER 
itself in any dimension is a Bernoulli shift. The latter is most easily seen by noting that the entire 
evolution of the process of coalescing random walks (which  yields the RER) can be generated by uniform 
$[0,1]$ random variables at each of the points of $\zd$ and hence must be a Bernoulli shift
being a factor of an i.i.d.\ process.
\end{remark}

\begin{question}\label{q.musd}
If one cannot provide an affirmative answer to Question~\ref{q.voter.bern},
can one give an example of an RER which has infinite clusters but the corresponding color
process is Bernoulli? 
\end{question}

\medskip\noindent {\bf Acknowledgements. } 
We thank Olle H\"{a}ggstr\"{o}m for providing us with Proposition~\ref{p.Olle} and
Russell Lyons for the key part of the proof of Proposition~\ref{p.Russ}.

\bibliography{tykesson}
\bibliographystyle{plain}

\end{document}